\documentclass[a4paper,reqno]{amsart} 

\usepackage{amssymb}
\usepackage[mathcal]{euscript}
\allowdisplaybreaks 

\usepackage[table]{xcolor}

\usepackage{tikz-cd} 

\newcommand{\bydef}{:=} 
\newcommand{\defby}{=:}
 
\newcommand{\veps}{\varepsilon}
\newcommand{\ul}[1]{\underline{#1}}

\newcommand{\wh}[1]{\widehat{#1}}
\newcommand{\wt}[1]{\widetilde{#1}}
\newcommand{\wb}[1]{\overline{#1}}

\DeclareMathOperator{\Skew}{\mathrm{Skew}} 

\newcommand{\id}{\mathrm{id}}
\newcommand{\ex}{\mathrm{ex}}

\newcommand{\lspan}[1]{\mathrm{span}\left\{#1\right\}}

\newcommand{\diag}{\mathrm{diag}}

\DeclareMathOperator*{\ot}{\otimes}

\newcommand{\gr}{\mathrm{gr}} 
\newcommand{\op}{\mathrm{op}}

\newcommand{\Cl}{\mathfrak{Cl}} 

\newcommand{\Gal}{\mathrm{Gal}} 
\newcommand{\Br}{\mathrm{Br}} 
\DeclareMathOperator{\supp}{\mathrm{Supp}}

\DeclareMathOperator{\Cent}{\mathrm{Cent}}

\DeclareMathOperator{\Int}{\mathrm{Int}} 
\DeclareMathOperator{\Arf}{\mathrm{Arf}} 

\newcommand{\bi}{\mathbf{i}}
\newcommand{\bj}{\mathbf{j}}
\newcommand{\bk}{\mathbf{k}}

\newcommand{\cA}{\mathcal{A}}
\newcommand{\cB}{\mathcal{B}} 
\newcommand{\cC}{\mathcal{C}}
\newcommand{\cD}{\mathcal{D}} 
 
\newcommand{\cF}{\mathcal{F}} 

\newcommand{\cH}{\mathcal{H}} 
\newcommand{\cI}{\mathcal{I}}
 
\newcommand{\cK}{\mathcal{K}}
\newcommand{\cL}{\mathcal{L}} 
\newcommand{\cM}{\mathcal{M}}
 
\newcommand{\cO}{\mathcal{O}}
\newcommand{\cQ}{\mathcal{Q}} 
\newcommand{\cR}{\mathcal{R}}
\newcommand{\cS}{\mathcal{S}}

\newcommand{\cV}{\mathcal{V}}
\newcommand{\cW}{\mathcal{W}}

\newcommand{\Hc}{\textup{H}}

\newcommand{\dc}{\textup{d}}
\DeclareMathOperator{\Ext}{\mathrm{Ext}}

\newcommand{\ZZ}{\mathbb{Z}}

\newcommand{\RR}{\mathbb{R}} 
\newcommand{\CC}{\mathbb{C}}
\newcommand{\HH}{\mathbb{H}} 

\newcommand{\FF}{\mathbb{F}} 
\newcommand{\KK}{\mathbb{K}}
\newcommand{\LL}{\mathbb{L}}
\newcommand{\Falg}{{\overline{\FF}}} 
\newcommand{\Fsep}{\FF_{\textrm{sep}}} 
\newcommand{\chr}[1]{\mathrm{char}\,#1}

\DeclareMathOperator{\AlgF}{\mathrm{Alg_{\FF}}} 
\DeclareMathOperator{\rad}{\mathrm{rad}} 
\DeclareMathOperator{\Hom}{\mathrm{Hom}}
\DeclareMathOperator{\End}{\mathrm{End}}
\DeclareMathOperator{\Alg}{\mathrm{Alg}}

\DeclareMathOperator{\Aut}{\mathrm{Aut}} 

\DeclareMathOperator{\Stab}{\mathrm{Stab}} 
\DeclareMathOperator{\Der}{\mathrm{Der}} 

\newcommand{\Ad}{\mathrm{Ad}}

\newcommand{\frsl}{{\mathfrak{sl}}}
\newcommand{\frsp}{{\mathfrak{sp}}}
\newcommand{\frso}{{\mathfrak{so}}}
\newcommand{\frpsl}{{\mathfrak{psl}}}
\newcommand{\frgl}{{\mathfrak{gl}}}
\newcommand{\frpgl}{{\mathfrak{pgl}}}

\newcommand{\frsu}{{\mathfrak{su}}}

\newcommand{\GL}{\mathrm{GL}}

\newcommand{\AGL}{\mathrm{AGL}} 
\newcommand{\Ort}{\mathrm{O}}

\newcommand{\AO}{\mathrm{AO}}

\newcommand{\SP}{\mathrm{Sp}}

\DeclareMathOperator{\AAut}{\mathbf{Aut}} 
\DeclareMathOperator{\AAntaut}{\overline{\mathbf{Aut}}} 
\DeclareMathOperator{\InAAut}{\mathbf{Int}} 

\newcommand{\Gs}{\mathbf{G}}

\newcommand{\PGLs}{\mathbf{PGL}} 
\newcommand{\Os}{\mathbf{O}}

\newcommand{\Diags}{\mathbf{Diag}}
\newcommand{\bmu}{\boldsymbol{\mu}}

\newtheorem{theorem}{Theorem}[section]
\newtheorem{proposition}[theorem]{Proposition}
\newtheorem*{proposition*}{Proposition}
\newtheorem{lemma}[theorem]{Lemma}
\newtheorem{corollary}[theorem]{Corollary}

\theoremstyle{definition} 
\newtheorem{df}[theorem]{Definition}
\newtheorem{example}[theorem]{Example}

\theoremstyle{remark} \newtheorem{remark}[theorem]{Remark}
\numberwithin{equation}{section}

\newcommand{\qup}{\textup{q}}

\begin{document}

\title[Gradings on algebras with involution and classical Lie algebras]{Gradings on associative algebras with involution\\ and real forms of classical simple Lie algebras}

\author[A.~Elduque]{Alberto Elduque}
\address[A.\,E., A.\,R.-E.]{Departamento de
Matem\'{a}ticas e Instituto Universitario de Matem\'aticas y
Aplicaciones, Universidad de Zaragoza, 50009 Zaragoza, Spain}
\email{elduque@unizar.es, rodrigo@unizar.es} 
\thanks{The first and thirds authors are supported by ``Agencia Estatal de
Investigaci\'on'' under grant
MTM2017-83506-C2-1-P, and by ``Departamento de Ciencia, Universidad y
Sociedad del Conocimiento del Gobierno de Arag\'on'' under grant E22\_20R: 
``\'Algebra y Geometr{\'\i}a''.}

\author[M.~Kochetov]{Mikhail Kochetov} 
\address[M.\,K.]{Department of Mathematics and Statistics, 
Memorial University of Newfoundland, 
St. John's, NL, A1C5S7, Canada}
\email{mikhail@mun.ca}
\thanks{The second author is supported by Discovery Grant 2018-04883 of the Natural Sciences and Engineering Research Council (NSERC) of Canada}

\author[A.~Rodrigo-Escudero]{Adri\'an Rodrigo-Escudero}



\date{}

\begin{abstract}

We study gradings by abelian groups on associative algebras with involution over an arbitrary field. Of particular importance are the fine gradings (that is, those that do not admit a proper refinement), because any grading on a finite-dimensional algebra can be obtained from them via a group homomorphism (although not in a unique way). 
We classify up to equivalence the fine gradings on simple associative algebras with involution over the field of real numbers (or any real closed field) and, as a consequence, on the real forms of classical simple Lie algebras.

\end{abstract}

\maketitle
\tableofcontents

\section{Introduction}\label{se:intro}

Let $G$ be a group and let $\cA$ be an algebra (not necessarily associative) over a field $\FF$. A \emph{$G$-grading}, or a \emph{grading by $G$}, on $\cA$ is a family $\Gamma=\{\cA_g\}_{g\in G}$ of $\FF$-subspaces of $\cA$, called \emph{homogeneous components}, that form a direct sum decomposition of $\cA$ and satisfy $\cA_{g_1}\cA_{g_2}\subset\cA_{g_1g_2}$ for all $g_1,g_2\in G$. We will write 
\[
\Gamma:\;\cA=\bigoplus_{g\in G}\cA_g
\]
and call $(\cA,\Gamma)$ a \emph{$G$-graded algebra}. If the grading is clear from the context, we will say that $\cA$ is a graded algebra. This definition can be easily adapted to the more general situation where $\cA$ has any number of multilinear operations. The \emph{support}, denoted $\supp\Gamma$ or $\supp\cA$, is the set 
$\{g\in G\mid\cA_g\ne 0\}$.

The most frequently encountered gradings are by the group $\ZZ$. Gradings by free abelian groups have long been prominent in the theory of Lie algebras and their representations; gradings by $\ZZ_2$ are well known in the context of superalgebras, starting with the early work [Wal63] of C.T.C. Wall on graded Brauer groups.
 A systematic study of gradings by arbitrary groups on Lie algebras was started in \cite{PZ}; many of the concepts introduced there apply to algebras of any kind and are briefly reviewed in Section \ref{se:generalities} (see Chapter 1 in the monograph \cite{EKmon} for more details). Group gradings have been extensively studied in the past two decades for many classes of algebras. When dealing with Lie algebras, it is natural to assume that the grading group $G$ is abelian, because in this case a $G$-grading on an associative algebra $\cR$ is also a grading on the corresponding Lie algebra $\cR^{(-)}$ (i.e., the underlying vector space of $\cR$ with the operation $[x,y]=xy-yx$), and a $G$-grading on a Lie algebra $\cL$ extends naturally to the universal enveloping algebra of $\cL$. In fact, if $\cL$ is a simple Lie algebra then the support of any $G$-grading on $\cL$ generates an abelian subgroup of $G$ (see e.g. \cite[Proposition 1.12]{EKmon}). 

The classical simple Lie algebras are closely related to (finite-dimensional) central simple associative algebras. Group gradings on the latter were described in \cite{BSZ,BZ02} over an algebraically closed field of characteristic $0$. For abelian groups, this description immediately gives a classification of fine gradings up to equivalence (see Section \ref{se:generalities} for terminology), and with some extra work a classification of all $G$-gradings, for a fixed abelian group $G$, up to isomorphism, over an algebraically closed field of any characteristic (see \cite{BK10} or Chapter 2 in \cite{EKmon}). 

To deal with classical simple Lie algebras, one also needs antiautomorphisms of central simple associative algebras, especially the involutive antiautomorphisms, which we call \emph{involutions} for short. In \cite{BShZ,BZ06}, it was proposed to study gradings (by abelian groups) on classical simple Lie (and Jordan) algebras over an algebraically closed field of characteristic $0$ through gradings on matrix algebras, possibly equipped with an involution. Involutions that preserve the components of a grading on $M_n(\FF)$, where $\FF$ is an algebraically closed field, were described in \cite{BShZ,BZ07} for the case $\chr{\FF}=0$ and in \cite{BG08} for $\chr{\FF}\ne 2$. The approach to study gradings on  classical simple Lie algebras through matrix algebras also works under the assumption $\chr{\FF}\ne 2$, through the use of \emph{automorphism group schemes} (which we briefly review at the beginning of Section \ref{se:Lie}). This was used in \cite{BK10} to classify all $G$-gradings on simple Lie algebras of series $A$, $B$, $C$ and $D$ (except $D_4$) up to isomorphism, over an algebraically closed field $\FF$ with $\chr{\FF}\ne 2$. Fine gradings on the same Lie algebras were described earlier in \cite{HPP98} over $\CC$, in terms of the \emph{maximal abelian diagonalizable} (MAD) subgroups of their automorphism groups, and then classified up to equivalence in \cite{E10}, including type $D_4$, over an algebraically closed field of characteristic $0$. This latter work used an additional ingredient, namely, representing the matrix algebra with a grading as $\End_\cD(\cV)$ where $\cD$ is a graded-division algebra and $\cV$ is a (finitely generated) graded right $\cD$-module. A $G$-graded algebra $\cD$ is said to be a \emph{graded-division algebra} if $\cD$ is unital and all nonzero homogeneous elements of $\cD$ are invertible. The approaches of \cite{BK10} and \cite{E10} were combined and presented in \cite{EKmon}. Finally, for type $D_4$, the classification of all $G$-gradings up to isomorphism was obtained in \cite{EK_D4}, over an algebraically closed field of characteristic different from $2$. 

As one would expect, the situation becomes substantially more complicated if we do not assume that the ground field is algebraically closed, even in the case of a real closed field such as $\RR$. One reason is that, due to the lack of roots of unity, we cannot translate gradings to actions by automorphisms and, in particular, fine gradings to MAD subgroups of the automorphism group (which were described for real forms of classical Lie algebras in \cite{HPP00}); one has to consider maximal diagonalizable subgroupschemes instead (see Subsection \ref{sse:transfer}). But the main reason is the complexity of the structure of finite-dimensional graded-division algebras, even under the assumption that the grading group is abelian. Over $\RR$, such graded-division algebras that are simple as algebras were classified in \cite{BZ16,R16}. Also, to obtain classical simple Lie algebras from associative algebras and involutions, one has to consider associative algebras that are not necessarily central simple, but rather \emph{central simple as algebras with involution}, which means that they have no proper nonzero ideals that are invariant under the involution and the set of symmetric elements in the center is equal to the ground field. With the exception of type $D_4$, this link to central simple associative algebras with involution allows one to transfer the study of gradings on any classical central simple Lie algebra to the associative context, over any ground field $\FF$ with $\chr{\FF}\ne 2$. (In the case of algebraically closed $\FF$, non-simple associative algebras can be avoided, but there is a price: to obtain all gradings in series $A$, one has to consider antiautomorphisms that are not necessarily involutive, so, even in this case, the setting of `central simple as an algebra with involution' leads to more elegant arguments.) This transfer was used in \cite{BKR_Lie} to classify up to isomorphism all $G$-gradings on the real forms of classical simple Lie algebras except $D_4$; this work relies on the study of involutions on real graded-division algebras carried out in \cite{BKR_inv}. Finally, a classification of $G$-gradings on real forms of $D_4$ was obtained in \cite{EK_G2D4}.

The main purpose of the present work is to classify fine gradings on all real forms of classical simple Lie algebras up to equivalence. To achieve this, we study gradings on associative algebras with involution, which are of independent interest. We will be working with these latter algebras at three levels of generality: 
 
\smallskip

\noindent$\bullet$\quad $\cR$ is a graded-simple associative algebra, 
over any field $\FF$, with DCC on graded left ideals or, equivalently, $\cR\simeq\End_\cD(\cV)$ where $\cD$ is a graded-division algebra and $\cV$ is a finitely generated graded right $\cD$-module of finite rank. In Section \ref{se:involutions}, we show that any involution $\varphi$ on $\cR$ that preserves the components of the grading is given by the adjunction with respect to a nondegenerate homogeneous $\varphi_0$-sesquilinear form $B:\cV\times\cV\to\cD$ that is either hermitian or skew-hermitian, where $\varphi_0:\cD\to\cD$ is an involution that preserves the components of the grading on $\cD$ (see Theorem \ref{th:involutions_from_B}). 
In Section~\ref{se:fine}, we construct a family of such graded algebras with involution that is complete in a certain sense: Theorem \ref{th:completeness} and Corollary \ref{cor:completeness}. In Section \ref{se:equivalence}, we consider the equivalence problem: Theorem \ref{th:equivalence_with_anti} gives a criterion for two graded algebras with involution $(\cR,\varphi)$ as above to be equivalent to each other and Theorem \ref{th:equivalenceMM'} solves the equivalence problem for the members of the family we constructed.

\smallskip

\noindent$\bullet$\quad $(\cR,\varphi)$ or, equivalently, $(\cD,\varphi_0)$ is finite-dimensional and central simple as an algebra with involution. This is the setting that can be applied to study forms (over any field $\FF$ with $\chr{\FF}\ne 2$) of classical simple Lie algebras and their gradings. 
Here the involution $\varphi_0$ can be fixed arbitrarily (Corollary~\ref{cor:involutions_from_B}) and, with a fixed $\varphi_0$, the equivalence criterion for the family we constructed can be conveniently expressed in terms of group actions (Corollary~\ref{cor:equivalenceMM'cs}). We also give sufficient conditions for the gradings in this family to be fine (Theorem \ref{th:GMfine}, Proposition \ref{pr:GMfineq2s0} and Theorem \ref{th:FineMex}). 

\smallskip

\noindent$\bullet$\quad $\FF$ is a real closed field, for example, $\RR$. In this case, a classification of $(\cD,\varphi_0)$ as in the previous item is available, so we can make explicit computations. In Section~\ref{se:fine}, we narrow down our complete family of graded algebras with involution to keep only those members whose grading is fine: Theorem \ref{th:fine_real}. In Section~\ref{se:equivalence}, we compute the relevant groups and make the equivalence criterion more explicit: Theorem \ref{th:equivalenceMM'real}. This criterion completes the classification of fine gradings by telling which of those appearing in Theorem \ref{th:fine_real} are equivalent to each other (see Corollary \ref{cor:fine_real}). 

\smallskip

Finally, in Section \ref{se:Lie}, we review the transfer technique mentioned above and obtain, as a direct consequence of our results on associative algebras with involution, a classification of fine gradings up to equivalence for all real forms of classical simple Lie algebras except $D_4$: see Theorems \ref{th:AI_RH}, \ref{th:AI_C}, \ref{th:AII_RH} and \ref{th:AII_C} for series $A$, Theorem~\ref{th:B} for series $B$ (which is much easier than other series and could be treated with simpler techniques), Theorem \ref{th:C} for Series $C$ and Theorem \ref{th:D} for series $D$. In the final subsection we complete the classification of fine gradings on real forms of type $D_4$. We combine our results on associative algebras with involution and the results of \cite{EK_G2D4} to classify what we call Type I and II fine gradings, to complement the classification of Type III fine gradings already obtained in \cite{EK_G2D4}.

When discussing $G$-gradings on algebras in general terms, we will use multiplicative notation for $G$ (even when it is abelian), because the operation in $G$ is related with the multiplication in the algebra. The identity element of $G$ will be denoted by $e$. For any abelian group $G$ and integer $n$, we have an endomorphism $[n]:G\to G$ sending $g\mapsto g^n$. We will use the following notation for its kernel and image:
\[
G_{[n]}\bydef\{g\in G\mid g^n=e\}\;\text{ and }\;G^{[n]}\bydef\{g^n \mid g\in G\}.
\]
We will often drop the subscript $\FF$ when working with homomorphisms or tensor products of vector spaces over the ground field $\FF$. We will, however, use a subscript when necessary to avoid confusion --- for example, when a tensor product over $\CC$ is considered in the context of real algebras.

\section{Generalities on gradings}\label{se:generalities}

In this section we will briefly review some general facts and terminology concerning gradings by groups on algebras. We will also introduce notation that is used throughout the paper. The reader is referred to Chapter~1 of the monograph \cite{EKmon} for more details.

\subsection{Group gradings}

Let $G$ be a group. A \emph{$G$-graded vector space} is a vector space $V$ with a fixed $G$-grading $\Gamma$, i.e., a direct sum decomposition $V=\bigoplus_{g\in G}V_g$. If a nonzero vector $v$ belongs to $V_g$, we will say that $v$ is \emph{homogeneous of degree $g$} and write $\deg v=g$. The zero vector is also considered homogeneous, but its degree is undefined. The class of $G$-graded vector spaces is a category in which the morphisms from $V$ to $W$ are the linear maps $f:V\to W$ that preserve degree, i.e., $f(V_g)\subset W_g$ for all $g\in G$. In particular, we can speak of isomorphism of $G$-graded vector spaces.

In this paper we will deal with graded associative and Lie algebras, as well as (associative) algebras with involution. To treat them in a uniform manner, we look at a more general setting. Let $\cA$ be an algebra with any number of multilinear operations. $\cA$ is said to be a \emph{$G$-graded algebra} if its underlying vector space is $G$-graded such that, for any operation $\varphi$ defined on $\cA$, we have $\varphi(\cA_{g_1},\ldots,\cA_{g_n})\subset\cA_{g_1\cdots g_n}$ for all $g_1,\ldots,g_n\in G$, where $n$ is the number of arguments taken by $\varphi$. For the usual case $n=2$, we recover the definition in the Introduction. For $n=1$, we get 
$\varphi(\cA_g)\subset\cA_g$ for all $g\in G$, which is what will be required for involutions. (Graded algebras in this general sense have been recently studied in \cite{Bahturin-Yukihide} from the point of view of their (graded) polynomial identities, based on deep results of \cite{Razmyslov}.)

$G$-graded modules over $G$-graded associative algebras are defined in the same vein. A \emph{homomorphism of $G$-graded algebras} is a homomorphism of algebras that is also a homomorphism of $G$-graded spaces. Two $G$-gradings, $\Gamma$ and $\Gamma'$, on the same algebra $\cA$ are said to be \emph{isomorphic} if there exists an isomorphism $(\cA,\Gamma)\to(\cA,\Gamma')$ or, in other words, there exists an automorphism of the algebra $\cA$ that maps each component of $\Gamma$ onto the component of $\Gamma'$ of the same degree.

Any group homomorphism $\alpha:G\to H$ gives a functor from $G$-graded vector spaces to $H$-graded ones: for $V=\bigoplus_{g\in G}V_g$, we define the $H$-graded vector space ${}^\alpha V$ to be the same space $V$ but equipped with the $H$-grading $V=\bigoplus_{h\in H} V'_h$ where $V'_h\bydef\bigoplus_{g\in\alpha^{-1}(h)}V_g$. This functor is the identity on morphisms: if $f$ is a homomorphism of $G$-graded spaces $V\to W$, then it is also a homomorphism of $H$-graded spaces ${}^\alpha V\to{}^\alpha W$. If the $G$-grading on $V$ is denoted by $\Gamma$, we will write ${}^\alpha\Gamma$ for the corresponding $H$-grading on $V$; the homogeneous elements of degree $g$ with respect to $\Gamma$ become homogeneous of degree $\alpha(g)$ with respect to $^\alpha{}\Gamma$.

A $G$-graded algebra $\cA$ and an $H$-graded algebra $\cB$ are \emph{weakly isomorphic} if there exist a group isomorphism $\gamma:G\to H$ and an algebra isomorphism $\psi:\cA\to\cB$ such that $\psi(\cA_g)=\cB_{\gamma(g)}$ for all $g\in G$ or, in other words, $\psi$ is an isomorphism of $G$-graded algebras $\cA\to{}^{\gamma^{-1}}\cB$.

Given an element $g\in G$ and a $G$-graded vector space $(V,\Gamma)$, the \emph{right shift} $V^{[g]}$ is the same space $V$ but equipped with the $G$-grading $\Gamma^{[g]}:V=\bigoplus_{g'\in G}V'_{g'}$ where $V'_{g'}\bydef V_{g'g^{-1}}$. In other words, the homogeneous elements of degree $g'$ with respect to $\Gamma$ become homogeneous of degree $g'g$ with respect to $\Gamma^{[g]}$. The {left shift} ${}^{[g]}V$ is defined similarly.

For $G$-graded vector spaces $V$ and $W$, define 
\[
\Hom(V,W)_g\bydef\{f\in\Hom(V,W)\mid f(V_{g'})\subset W_{gg'}\text{ for all }g'\in G\}.
\]
In particular, if we take $W=V$, any $G$-grading on $V$ induces a $G$-grading on the algebra $\End^\gr(V)\bydef\bigoplus_{g\in G}\End(V)_g$ and its graded subalgebras. Note that these gradings are not affected by \emph{right} shifts of the grading on $\cV$.

The tensor product $V\otimes W$ is also $G$-graded: 
\[
(V\otimes W)_g\bydef\bigoplus_{g_1,g_2\in G: g_1g_2=g}V_{g_1}\otimes W_{g_2}
\]
If $G$ is abelian and $\cA$ and $\cB$ are $G$-graded associative algebras, then $\cA\otimes\cB$ becomes a $G$-graded associative algebra in this way.

\subsection{The universal group and the Weyl group of a grading}\label{sse:universal}

There is a more general concept of a grading on an algebra $\cA$, namely, a set $S$ of nonzero subspaces of $\cA$, which we write as $\Gamma=\{\cA_s\}_{s\in S}$ for convenience, such that $\cA=\bigoplus_{s\in S}\cA_s$ and, for any $n$-ary operation $\varphi$ defined on $\cA$ and any $s_1,\ldots,s_n\in S$, we have $\varphi(\cA_{s_1},\ldots,\cA_{s_n})\subset\cA_s$ for some $s\in S$. Any $G$-grading on $\cA$ becomes a grading in this sense if we take $S=\supp\cA$. 

An \emph{equivalence} from $\cA=\bigoplus_{s\in S}\cA_s$ to $\cB=\bigoplus_{t\in T}\cB_t$ is an algebra isomorphism $\psi:\cA\to\cB$ such that, for any $s\in S$, we have $\psi(\cA_s)=\cB_t$ for some $t\in T$. Since we assume that all $\cA_s$ are nonzero, $\psi$ defines a bijection $\gamma:S\to T$ such that $\psi(\cA_s)=\cB_{\gamma(s)}$ for all $s\in S$. Two gradings, $\Gamma$ and $\Gamma'$, on the same algebra $\cA$ are said to be \emph{equivalent} if there exists an equivalence $(\cA,\Gamma)\to(\cA,\Gamma')$ or, in other words, there exists an automorphism of the algebra $\cA$ that maps each component of $\Gamma$ onto a component of $\Gamma'$.

For a given grading $\Gamma$ on an algebra $\cA$ as above, there may or may not exist a realization of $\Gamma$ in a group $G$, by which we mean an injective map $\iota:S\to G$ such that assigning the nonzero elements of $\cA_s$ degree $\iota(s)\in G$, for all $s\in S$, defines a $G$-grading on $\cA$. For example, if $\cA$ has a unary operation $\varphi$, a necessary condition is that $\varphi(\cA_s)\subset\cA_s$ for all $s\in S$. If such pairs $(G,\iota)$ exist, there is a universal one among them: the group $U=U(\Gamma)$ generated by the set $S$ subject to all relations of the form $s_1\cdots s_n=s$ whenever $0\ne\varphi(\cA_{s_1},\ldots,\cA_{s_n})\subset\cA_s$ for an $n$-ary operation $\varphi$ on $\cA$, and $\iota$ is the inclusion map $S\to U$. This is called \emph{the universal group} of the grading $\Gamma$. 

If we start with a $G$-grading $\Gamma$ with support $S$, then the inclusion $\iota:S\to U$ realizes $\Gamma$ as a $U$-grading, which we denote by $\wt{\Gamma}$, and the universal property of $U$ gives a unique group homomorphism $\alpha:U\to G$ such that $\alpha\iota=\iota'$ where $\iota'$ is the inclusion $S\to G$. In other words, $\alpha$ is characterized by the property ${}^{\alpha}\wt{\Gamma}=\Gamma$.

When we use universal groups to realize gradings on $\cA$ and $\cB$, any equivalence $\psi:\cA\to\cB$ becomes a weak isomorphism: the bijection of the supports $\gamma:S\to T$ determined by $\psi$ extends to a unique isomorphism of the universal groups.

Given a grading $\Gamma$ on an algebra $\cA$, we can consider the group $\Aut(\Gamma)$ of all equivalences from the graded algebra $(\cA,\Gamma)$ to itself. Applying the above property of universal groups (with $\cA=\cB$), we see that the permutation of the support of $\Gamma$ defined by any element $\Aut(\Gamma)$ extends to a unique automorphism of the universal group $U=U(\Gamma)$. This gives us a group homomorphism $\Aut(\Gamma)\to\Aut(U)$, whose kernel $\Stab(\Gamma) $ consists of all degree-preserving automorphisms, i.e., isomorphisms from the graded algebra $(\cA,\Gamma)$ to itself. The image of this homomorphism $\Aut(\Gamma)\to\Aut(U)$ is known as the \emph{Weyl group} of the grading $\Gamma$:
\[
W(\Gamma)\bydef\Aut(\Gamma)/\Stab(\Gamma)\hookrightarrow\Aut(U(\Gamma)).
\]

\subsection{Fine gradings}

A $G$-grading $\Gamma:\cA=\bigoplus_{g\in G}\cA_g$ is said to be a \emph{refinement} of an $H$-grading $\Gamma':\cA=\bigoplus_{h\in H}\cA'_h$ (or $\Gamma'$ a \emph{coarsening} of $\Gamma$) if, for any $g\in G$, there exists $h\in H$ such that $\cA_g\subset\cA'_h$. If the inclusion is proper for at least one $g\in\supp\Gamma$, the refinement (or coarsening) is called \emph{proper}. For example, if $\alpha:G\to H$ is a group homomorphism, then ${}^\alpha\Gamma$ is a coarsening of $\Gamma$, which is proper if and only if $\alpha$ is not injective on the support of $\Gamma$. If $G$ is the universal group of $\Gamma$ then any coarsening $\Gamma'$ has the form 
${}^\alpha\Gamma$ for a unique group homomorphism 
$\alpha:G\to H$. 

A grading is said to be \emph{fine} if it has no proper refinement. 
If $\cA$ is finite-dimensional then any grading is a coarsening of a fine grading. Hence, if $\{\Gamma_i\}_{i\in I}$ is a set of representatives of the equivalence classes of fine gradings on $\cA$ and $G_i$ is the universal group of $\Gamma_i$, then any $G$-grading $\Gamma$ on $\cA$ is isomorphic to ${}^\alpha\Gamma_i$ for some $i\in I$ and a group homomorphism $\alpha:G_i\to G$.

\subsection{Gradings and actions}

Given a $G$-grading $\Gamma:\cA=\bigoplus_{g\in G}\cA_g$, any group homomorphism $\chi:G\to\FF^\times$ acts as an automorphism of $\cA$ as follows: $\chi\cdot a=\chi(g)a$ for all $a\in\cA_g$ and $g\in G$. Note that this is actually an automorphism of $\cA$ as a graded algebra, as it leaves each component $\cA_g$ invariant --- in fact, acts on it as a scalar operator. Thus $\Gamma$ defines a group homomorphism from $\Hom(G,\FF^\times)$ to the automorphism group of the graded algebra $\cA$, which is particularly useful if $G$ is abelian and $\FF$ is algebraically closed and of characteristic $0$, because then $\Hom(G,\FF^\times)$ separates points of $G$ and, therefore, the grading $\Gamma$ can be recovered as a simultaneous eigenspace decomposition with respect to these automorphisms.

Conversely, if $\psi$ is an automorphism of an algebra $\cA$ satisfying $\psi^n=\id_\cA$ and if $\FF$ contains a primitive $n$-th root of unity, then the eigenspace decomposition of $\cA$ with respect to $\psi$ is a $\ZZ_n$-grading.

We will further elaborate on the connection between gradings (by abelian groups) and actions (by diagonalizable group schemes) in Subsection \ref{sse:transfer}.

\section{Involutions on graded-simple associative algebras}\label{se:involutions}

In this section, all algebras are assumed associative, but the ground field $\FF$ is arbitrary. 

\subsection{Graded-simple algebras with DCC}

Let $G$ be a group. 
A $G$-graded algebra $\cR$ is said to be \emph{simple as a graded algebra}, or \emph{graded-simple} for short, 
if $\cR^2\ne 0$ and the only graded (two-sided) ideals of $\cR$ are $0$ and $\cR$. As shown in \cite[\S 2.1]{EKmon}, the graded-simple algebras satisfying the descending chain condition on graded left ideals are precisely the graded algebras of the form $\End_\cD(\cV)$ where $\cD$ is a graded-division algebra and $\cV$ is a graded right $\cD$-module of finite rank. Any graded $\cD$-module is free: it admits a \emph{graded basis}, i.e., a $\cD$-basis consisting of homogeneous elements. If we fix a graded basis $\{v_1,\ldots,v_k\}$ for $\cV$, with $\deg v_i=g_i$, then $\End_\cD(\cV)$ can be identified with the matrix algebra $M_k(\cD)\simeq M_k(\FF)\otimes\cD$ with the following grading:
\begin{equation}\label{eq:degEijd}
\deg(E_{ij}\otimes d)=g_i (\deg d) g_j^{-1}\;\text{ for any }d\in\cD^\times_\gr,
\end{equation}
where $\cD^\times_\gr$ denotes the set of nonzero homogeneous elements of $\cD$, which is a group under multiplication. Note that $\deg:\cD^\times_\gr\to G$ is a group homomorphism whose image is $T\bydef\supp\cD$. For any $d\in\cD^\times_\gr$, we denote by $\Int(d)$ the corresponding inner automorphism:
\[
\Int(d):\cD\to\cD,\; x\mapsto dxd^{-1}.
\]
Note that $\Int(d)$ is actually an isomorphism of graded algebras $\cD\to{}^{[t^{-1}]}\cD^{[t]}$ where $t=\deg d$.

\subsection{Antiautomorphisms and involutions that preserve degree}

We are interested in involutions of a graded algebra $\cR$ as above, by which we mean the involutive $\FF$-linear antiautomorphisms of $\cR$ that preserve degrees, i.e., leave every homogeneous component of $\cR$ invariant. First note that, since $\cR$ is graded-simple, the existence of a degree-preserving antiautomorphism implies that the support of $\cR$ generates an abelian subgroup of $G$ (see \cite{BShZ} or \cite[Proposition 2.49]{EKmon}), so \emph{we will assume for the remainder of this section that $G$ is abelian}.

\begin{theorem}\label{th:involutions_from_B}
Let $G$ be an abelian group and consider the $G$-graded algebra $\cR=\End_\cD(\cV)$ where $\cD$ is a graded-division algebra and $\cV$ is a nonzero graded right $\cD$-module of finite rank. 
\begin{enumerate}
\item[(1)] If $\varphi$ is an antiautomorphism of the graded algebra $\cR$, then there exists an antiautomorphism $\varphi_0$ of the graded algebra $\cD$ and a nondegenerate $\varphi_0$-sesquilinear form $B:\cV\times\cV\rightarrow\cD$, by which we mean a nondegenerate $\FF$-bilinear mapping that is $\varphi_0$-sesquilinear over $\cD$, i.e., 
\begin{enumerate}
\item[(i)] 
$B(vd,w)=\varphi_0(d)B(v,w)$ and $B(v,wd)=B(v,w)d$ for all $d\in\cD$ and $v,w\in\cV$,
\end{enumerate}
and homogeneous of some degree $g_0\in G$, i.e.,
\begin{enumerate}
\item[(ii)] $B(\cV_a,\cV_b)\subset\cD_{g_0ab}$ for all $a,b\in G$, 
\end{enumerate}
such that $\varphi$ is the adjunction with respect to $B$, i.e.,
\begin{enumerate}
\item[(iii)] $B(rv,w)=B(v,\varphi(r)w)$ for all $r\in\cR$ and $v,w\in \cV$. 
\end{enumerate}
\item[(2)] Another pair $(\varphi'_0,B')$ satisfies these conditions if and only if there exists $d\in\cD^\times_\gr$ such that $B'=dB$ and $\varphi'_0=\Int(d)\circ\varphi_0$.
\item[(3)] If $\varphi$ is an involution, then the pair $(\varphi_0,B)$ as in part (1) can be chosen so that
$\varphi_0$ is an involution and $B$ is hermitian or skew-hermitian, by which we mean that $B(w,v)=\delta \varphi_0\bigl(B(v,w)\bigr)$ for all $v,w\in\cV$, where $\delta=1$ (hermitian) or $\delta=-1$ (skew-hermitian).
\item[(4)] Let $(\varphi_0,B)$ be a pair chosen for an involution $\varphi$ as in part (3). Then:
\begin{enumerate}
\item[(i)] Any other such pair $(\varphi'_0,B')$ has the form $(\Int(d)\circ\varphi_0,dB)$ 
where $d\in\cD^\times_\gr$ satisfies $\varphi_0(d)=d$ (symmetric) or $\varphi_0(d)=-d$ (skew-symmetric). More precisely, we have $B'(w,v)=\delta' \varphi'_0\bigl(B'(v,w)\bigr)$ for all $v,w\in\cV$ where $\delta'=\delta$ if $d$ is symmetric and $\delta'=-\delta$ if $d$ is skew-symmetric.
\item[(ii)] If $\varphi'_0$ is a degree-preserving involution of $\cD$ such that $\varphi'_0\varphi_0^{-1}$ is an inner automorphism of $\cD$, then there exists $d\in\cD^\times_\gr$ such that $\varphi'_0=\Int(d)\circ\varphi_0$ and the pair $(\varphi'_0,dB)$ satisfies part (3).
\end{enumerate}
\end{enumerate}
\end{theorem}

\begin{proof}
Parts (1) and (2) are given by \cite[Theorem 2.57]{EKmon}, so we proceed to prove part (3). 
Let $\varphi$ be an involution and let $(\varphi_0,B)$ be as in part (1). We will show that there exists $\lambda\in\cD^\times_e$ such that $\varphi'_0\bydef\Int(\lambda)\circ\varphi_0$ is an involution and $B'\bydef \lambda B$ is hermitian or skew-hermitian.

Consider $\wb{B}(v, w) \bydef \varphi_0^{-1}\bigl(B(w,v)\bigr)$.
This is a $\varphi_0^{-1}$-sesquilinear form of the same degree as $B$, 
and the adjunction with respect to $\wb{B}$ is $\varphi^{-1}=\varphi$. 
Therefore, by part (2), there exists $0\ne\delta\in\cD_e$ such that $\wb{B}=\delta B$ and 
\begin{equation}\label{eq:phi0}
\varphi_0^{-1}=\Int(\delta)\circ\varphi_0.
\end{equation}
Note that this equation implies that $\varphi_0^2=\id_\cD$ if and only if $\delta$ is central. In particular, if $\wb{B}=\delta B$ for $\delta\in\{\pm 1\}$ then $\varphi_0$ is automatically an involution.

A straightforward computation (see e.g. \cite[Lemma 5]{BKR_Lie}) shows that changing $B$ to $dB$, for any $d\in\cD^\times_\gr$, replaces $\delta$ by 
\begin{equation}\label{eq:delta'}
\delta'=\delta\varphi_0(d)d^{-1},
\end{equation}
and it follows by considering $\delta B$ and the equation $B=\wb{\delta B}$ that 
\begin{equation}\label{eq:delta}
\delta\varphi_0(\delta)=1.
\end{equation}
 
Let $\LL$ be the subfield of the division algebra $\cD_e$ generated by $\FF$ and $\delta$. 
Since $\varphi_0|_\FF=\id_\FF$ by $\FF$-linearity and $\varphi_0(\delta)=\delta^{-1}$ by equation \eqref{eq:delta}, $\varphi_0$ restricts to an involutive automorphism $\sigma$ of $\LL$ over $\FF$. By definition, $\sigma(\delta)=\delta^{-1}$.

If $\sigma=\id_\LL$ then $\delta^{-1}=\sigma(\delta)=\delta$ and hence $\delta\in\{\pm 1\}$, 
so $\varphi_0$ is an involution and $B$ is hermitian or skew-hermitian. 

If, on the contrary, $\sigma\ne\id_\LL$, let $\LL_0=\{x\in\LL\mid \sigma(x)=x\}$. 
Then $\LL/\LL_0$ is a quadratic Galois field extension with Galois group $\langle\sigma\rangle$ and, 
since $\delta\sigma(\delta)=1$, Hilbert's Theorem 90 gives an element $\lambda\in\LL\subset \cD_e$ such that 
$\delta=\lambda\sigma(\lambda)^{-1}$. 
Replacing $B$ by $B'=\lambda B$ changes $\delta$ to $\delta'=\delta\varphi_0(\lambda)\lambda^{-1}=1$ by equation \eqref{eq:delta'}, so $\varphi'_0=\Int(\lambda)\circ\varphi_0$ is an involution and $B'$ is hermitian.

It remains to prove part (4). Let $(\varphi_0,B)$ be as in part (3): $\varphi_0^2=\id_\cD$ and $\wb{B}=\delta B$ for $\delta\in\{\pm 1\}$. Then (i) follows immediately from part (2) and formula \eqref{eq:delta'}.

Now let $\varphi'_0$ be as in (ii). Since the inner automorphism 
$\varphi'_0\varphi_0^{-1}$ preserves degrees, it must be  of the form $\Int(d')$ for some $d'\in\cD^\times_\gr$ (see \cite[Lemma 3.3]{E10} or 
\cite[Theorem 2.1]{R20}), so we have 
$\varphi'_0=\Int(d')\circ\varphi_0$. Letting $B'\bydef d'B$, we see from part (2) that the pair $(\varphi'_0,B')$ satisfies the conditions in part (1). Applying to this pair the arguments in the proof of part (3), we see that $\wb{B'}=\delta'B'$ for some $0\ne\delta'\in\cD_e$, and this $\delta'$ must be central because $\varphi'_0$ is an involution (see equation \eqref{eq:phi0}). Moreover, we have $\delta'\varphi'_0(\delta')=1$ by equation \eqref{eq:delta}. Considering the subfield $\LL$ of $\cD_e$ generated by $\FF$ and $\delta'$, we see that either $\delta'\in\{\pm 1\}$ or there exists $\lambda\in\LL\subset Z(\cD)\cap\cD_e$ such that $\delta'=\lambda\sigma(\lambda)^{-1}$. In the former case, we take $d=d'$ and, in the latter case, we take $d=\lambda d'$, which satisfies $\varphi'_0=\Int(d)\circ\varphi_0$ because $\lambda$ is central. 
\end{proof}

If we fix a graded basis $\{v_1,\ldots,v_k\}$ of the graded $\cD$-module $\cV$ and represent the elements of $\End_\cD(\cV)$ by matrices in $M_k(\cD)$, the adjunction with respect to $B$ can be expressed in matrix form:
\begin{equation}\label{eq:phi_with_matrices}
\varphi(X)=\Phi^{-1}\varphi_0(X^T)\Phi\;\text{ for all }X\in M_k(\cD),
\end{equation}
where $X^T$ is the transpose of $X$, $\varphi_0$ is applied to matrices entry-wise, and $\Phi$ is the matrix representing $B$ relative to the chosen graded basis: $\Phi=\bigl(B(v_i,v_j)\bigr)_{1\le i,j\le k}$.
Note that, since $G$ is abelian, equation \eqref{eq:degEijd} tells us that $M_k(\cD)$ is isomorphic to the tensor product of graded algebras $M_k(\FF)$ and $\cD$ where the grading on $M_k(\FF)$ is defined by the rule $\deg E_{ij}=g_ig_j^{-1}$ and is called the \emph{elementary grading} determined by the $k$-tuple $(g_1,\ldots,g_k)$, where $g_i=\deg v_i$. 

\subsection{The central simple case}\label{sse:central_simple_basics}

Of particular importance to us is the case of finite-dimensional algebras with involution that are  
\emph{central simple as algebras with involution}, i.e., have no proper nonzero ideals that are invariant under 
the involution and have no symmetric elements in the center other than the scalars. Let $(\cR,\varphi)$ be such an algebra with involution and let $\KK=Z(\cR)$. Then $\KK$ is a so-called \'etale (i.e., commutative and separable) algebra over $\FF$ of dimension at most $2$:
\begin{itemize}
\item If $\KK=\FF$, $\cR$ is central simple over $\FF$ and $\varphi$ is said to be of the first kind;
\item If $\KK$ is a separable quadratic field extension of $\FF$, then $\cR$ is central simple over $\KK$ and $\varphi$ is said to be of the second kind;
\item If $\KK\simeq\FF\times\FF$, then $\cR$ is not simple: it is isomorphic to $\cS\times\cS^\op$ for a central simple algebra $\cS$ over $\FF$, with the exchange involution: $(x,y)\mapsto(y,x)$.
\end{itemize}
Now suppose that $(\cR,\varphi)$ is given a grading by an abelian group $G$. It is easy to see that the center is graded (see e.g. \cite[Lemma 4]{BKR_Lie}). If $\KK\ne\FF$, it will be important to distinguish two cases: if $\KK$ is graded trivially, we will say that the grading is of \emph{Type I} and otherwise of \emph{Type II}. In the latter case, we have $\KK=\KK_e\oplus\KK_f$ where $\KK_e=\FF$ and $f\in G$ is an element of order $2$, which we will call the \emph{distinguished element}. Since $\varphi$ leaves $\KK_f$ invariant and the nonzero elements of $\KK_f$ cannot be symmetric, it follows that they are skew-symmetric and $\chr\FF\ne 2$. 

In all cases except gradings of Type I with $\KK\simeq\FF\times\FF$, the algebra $\cR$ is graded-simple and hence we can apply Theorem \ref{th:involutions_from_B} to describe $(\cR,\varphi)$ as a graded algebra with involution. The case of Type I gradings when $\KK\simeq\FF\times\FF$ is actually easy: the ideals $\cS$ and $\cS^\op$ are graded and, since the exchange involution preserves degrees, $\cS$ and $\cS^\op$ must have the same grading, so the situation reduces to gradings on central simple algebras without involution.

We know that $\cR\simeq\End_\cD(\cV)$ as a graded algebra, so let $(\varphi_0,B)$ be as in part (3) of Theorem \ref{th:involutions_from_B}. We can identify the centers of $\End_\cD(\cV)$ and $\cD$ as graded algebras with involution, so the pair $(\cD,\varphi_0)$ is of the same kind as $(\cR,\varphi)$. However, $\varphi_0$ is not determined by $\varphi$. In fact, we have the following

\begin{corollary}\label{cor:involutions_from_B}
Let $\varphi$ be an involution on the graded algebra $\cR = \End_\cD(\cV)$ where $\cD$ is a finite-dimensional graded-division algebra and $\cV$ is a nonzero graded right $\cD$-module of finite rank. 
Assume that $(\cR,\varphi)$ is central simple as an algebra with involution. 
Then $\cD$ admits a degree-preserving involution of the same kind as $\varphi$, and for any such involution $\varphi_0$, there exists a nondegenerate homogeneous $\varphi_0$-sesquilinear form $B:\cV\times\cV\to\cD$ such that $\varphi$ is the adjunction with respect to $B$. Moreover,
\begin{itemize}
\item[(i)] If $\cR$ is central simple and $\chr{\FF}\ne 2$, then $B$ is determined up to a nonzero factor in $\FF$ and is hermitian if $\varphi$ and $\varphi_0$ have the same type (orthogonal or symplectic) and skew-hermitian if they have the opposite type; 
\item[(ii)] If $\cR$ is central simple and $\chr{\FF}=2$, then $B$ is determined up to a nonzero factor in $\FF$ and is hermitian;
\item[(iii)] If $\cR$ is not central simple, then $B$ can be chosen hermitian, and this condition determines it up to a nonzero factor in $\FF$. 
\end{itemize}
\end{corollary}

\begin{proof}
By part (4) of Theorem \ref{th:involutions_from_B}, we can find a hermitian or skew-hermitian $B$ for any $\varphi_0$, because any two involutions of the same kind differ by an inner automorphism of $\cD$ (Noether-Skolem Theorem). Moreover, when $\varphi_0$ is fixed, the homogeneous $\varphi_0$-sesquilinear form $B$ is determined up to a factor in $Z(\cD)^\times_\gr$, which immediately implies (i) and (ii) if we take into account that $\cR\simeq M_k(\cD)$ as an ungraded algebra and $\varphi$ is given by equation \eqref{eq:phi_with_matrices}. If $\cR$ is not central simple and $B$ is not hermitian, we can multiply $B$ by any skew-symmetric element of $Z(\cD)^\times_\gr$.
\end{proof}

Since we are interested in fine gradings, of particular importance to us will be the case where the division algebra $\cD_e$ is $\FF$. Then, for any $t\in T\bydef\supp\cD$ and $0\ne x_t\in\cD_t$, we have $\cD_t=\cD_e x_t=\FF x_t$, which means that $\cD$ is a twisted group algebra of $T$ (with coefficients in $\FF$) equipped with its natural grading. 
Since $T$ is abelian, there exist scalars $\beta(s,t)\in\FF^\times$, for all $s,t\in T$, such that $x_s x_t = \beta(s,t) x_t x_s$, and it is clear that the function $\beta:T\times T\to\FF^\times$ does not depend on the choice of the elements $x_t$, so we have
\begin{equation}\label{eq:def_beta}
xy = \beta(s,t) yx\;\text{ for all }x\in\cD_s,y\in\cD_t\text{ and }s,t\in T.
\end{equation} 
It follows that the function $\beta$ is a \emph{bicharacter}, i.e., multiplicative in each variable. Moreover, $\beta$ is \emph{alternating} in the sense that $\beta(t,t)=1$ for all $t\in T$, which implies that $\beta$ is (multiplicatively) skew-symmetric: $\beta(t,s)=\beta(s,t)^{-1}$ for all $s,t\in T$ (which is also obvious from equation \eqref{eq:def_beta}). By analogy with (skew-)symmetric bilinear forms, we define
\begin{equation}\label{eq:def_rad_beta}
\rad\beta\bydef\{s\in T\mid\beta(s,t)=1\text{ for all }t\in T\}.
\end{equation}
It is easy to see that $Z(\cD)=\bigoplus_{s\in\rad\beta}\cD_s$, so in our case either $\rad\beta=\{e\}$  (a \emph{nondegenerate} bicharacter) if $\KK=\FF$ or 
$\rad\beta=\{e,f\}$ (a \emph{slightly degenerate} bicharacter) if $\KK$ is quadratic over $\FF$ and the grading is of Type II. In the slightly degenerate case, we will sometimes consider the quotient group $\wb{T}\bydef T/\rad\beta$ and the nondegenerate bicharacter $\bar{\beta}:\wb{T}\times\wb{T}\to\FF^\times$ induced by 
$\beta$.

We will also need the case of Type I gradings where $\cD_e$ is $\KK$, which is a quadratic field extension of $\FF$. Then the same considerations as above lead us to a nondegenerate alternating bicharacter $\beta:T\times T\to\KK^\times$.

Whenever a finite abelian group $T$ has a nondegenerate alternating bicharacter 
$\beta:T\times T\to\LL^\times$, where $\LL$ is any field, $T$ admits a \emph{symplectic basis} (see e.g. \cite[2,\S 2]{EKmon}), 
i.e., a generating set of the form  $\{a_1,b_1,\ldots,a_m,b_m\}$ with the order of both $a_i$ and $b_i$ equal to some 
$n_i\ge 2$, $i=1,\ldots,m$, such that 
\begin{equation}\label{eq:symplectic_basis}
T=\langle a_1\rangle\times\langle b_1\rangle\times\cdots\times\langle a_m\rangle\times\langle b_m\rangle,
\end{equation}
and $\beta(a_i,b_i)=\zeta_i$, with $\zeta_i\in\LL$ a primitive root of unity of degree $n_i$, 
while $\beta(a_i,b_j)=1$ for $i\neq j$ and $\beta(a_i,a_j)=\beta(b_i,b_j)=1$ for all $i,j$.
In particular, $T$ is the direct product of two isomorphic subgroups:
$\langle a_1,\ldots,a_m\rangle$ and $\langle b_1,\ldots,b_m\rangle$.
Now if $T$ is the support of a graded-division algebra $\cD$ with $\cD_e=\LL=Z(\cD)$ and $\beta$ is defined by equation \eqref{eq:def_beta}, then any elements $0\ne X_i\in\cD_{a_i}$ and $0\ne Y_i\in\cD_{b_i}$ generate $\cD$ as an $\LL$-algebra and satisfy the following defining relations:
\begin{equation}\label{eq:genrel}
\begin{split}
&X_i^{n_i}=\mu_i,\; Y_i^{n_i}=\nu_i,\; X_i Y_i=\zeta_i Y_i X_i,\\
&X_i X_j=X_j X_i,\; Y_i Y_j=Y_j Y_i,\text{ and } X_i Y_j=Y_j X_i\text{ for }i\neq j.
\end{split}
\end{equation}

\begin{remark}\label{rm:change_roots}
The primitive root of unity $\zeta_i=\beta(a_i,b_i)$ can be chosen arbitrarily at the expense of changing the symplectic basis.
\end{remark}

This can be restated by saying that $\cD$ is a tensor product of \emph{cyclic} or \emph{symbol algebras} (see e.g. \cite[p.~27]{KMRT}):
\begin{equation}\label{eq:symbol_alg}
\cD\simeq (\mu_1,\nu_1)_{\zeta_1^{-1},\LL}\otimes_\LL\cdots\otimes_\LL (\mu_m,\nu_m)_{\zeta_m^{-1},\LL}.
\end{equation}
Since the elements $X_i$ and $Y_i$ of degrees $a_i$ and $b_i$, respectively, are determined up to a factor in $\LL^\times$, and this factor is raised to the power $n_i$ in defining relations \eqref{eq:genrel}, the graded-division algebras $\cD$ with support $T$ and $\cD_e=\LL=Z(\cD)$ are classified up to isomorphism of graded algebras by the bicharacter $\beta$ and the images of the scalars $\mu_i$ and $\nu_i$ (assigned to a fixed symplectic basis of $T$) in the quotient group $\LL^\times/\LL^{\times n_i}$. However, in this paper we are interested in classifying graded algebras up to equivalence. For graded-division algebras this is the same as the classification up to weak isomorphism, since the fact that $\cD_s\cD_t=\cD_{st}$ for all $s,t\in T$ immediately implies that $T$ is the universal group of $\cD$. 

If $\LL$ is algebraically closed, then each symbol algebra in the right-hand side of equation \eqref{eq:symbol_alg} is of the form $M_n(\LL)$, equipped with the so-called \emph{Pauli grading} by $\langle a,b\rangle\simeq\ZZ_n^2$: we can take for the generator $X$ of degree $a$ the diagonal matrix with consecutive powers of $\zeta=\beta(a,b)$ and for the generator $Y$ of degree $b$ the permutation matrix of the cycle $(1,\ldots,n)$. Therefore, over an algebraically closed field, the finite-dimensional graded-division algebras that are simple as algebras and have abelian support are classified up to equivalence by the isomorphism class of their support, as was first shown in \cite{BSZ}. Over the field of real numbers, such a classification was obtained in \cite{BZ16,R16} and is considerably more complicated, because here $\cD_e$ can be not only $\RR$, but also $\CC$ or $\HH$. Over an arbitrary field $\LL$, an explicit classification up to equivalence seems out of reach.   
Here we only make an observation about Weyl groups, to be used later:

\begin{proposition}\label{pr:Weyl_group_C}
If $(\cD,\Gamma_\cD)$ is a graded-division algebra over a field $\LL$ with an abelian support $T$ and $\cD_e=\LL$, then $W(\Gamma_\cD)$ is contained in 
\[
\Aut(T,\beta)\bydef\{\alpha\in\Aut(T)\mid\beta\bigl(\alpha(s),\alpha(t)\bigr)=\beta(s,t)\text{ for all }s,t\in T\},
\]
where $\beta:T\times T\to\LL^\times$ is the alternating bicharacter defined by equation \eqref{eq:def_beta}.
If $\Ext_\ZZ(T,\LL^\times)$ is trivial (for example, if $\LL$ is algebraically closed or $T$ is finite of exponent $n$ and $\LL^\times=\LL^{\times n}$), then $W(\Gamma_\cD)=\Aut(T,\beta)$.
\end{proposition}

\begin{proof}
 
Since $\cD_e=\LL$, $\cD$ is the twisted group algebra $\LL^\sigma T$ for some $2$-cocycle $\sigma:T\times T\to\LL^\times$ (with trivial action of $T$ on $\LL^\times$). It is not hard to see (\cite[Proposition 2.14]{EKmon}) that $W(\Gamma_\cD)$ is the stabilizer of the class $[\sigma]\in\Hc^2(T,\LL^\times)$ in $\Aut(T)$. Since $\beta(s,t)=\sigma(s,t)/\sigma(t,s)$ for all $s,t\in T$, $\beta$ is the image of $[\sigma]$ under the well-known homomorphism $\Hc^2(T,\LL^\times)\to\Hom(T\wedge_\ZZ T,\LL^\times)$, which is valid for any abelian groups in place of $T$ and $\LL^\times$. The kernel of this homomorphism is the group of classes of symmetric $2$-cocycles, which can be identified with $\Ext_\ZZ(T,\LL^\times)$.
\end{proof}

We now return to our situation where $\cD$ has a degree-preserving involution $\varphi_0$ that makes it central simple as an algebra with involution over $\FF$, and $\cD_e$ is either $\FF$ or $\KK\bydef Z(\cD)$. In the first case this forces $T$ to be $2$-elementary, as shown in the next lemma, which generalizes analogous results in \cite{BZ06,BKR_Lie}:

\begin{lemma}\label{lm:2elementary}
Suppose that a finite-dimensional graded-division algebra $\cD$ with $\cD_e=\FF$ admits a degree-preserving involution $\varphi_0$ such that $(\cD,\varphi_0)$ is central simple as an algebra with involution. Then the support $T$ of $\cD$ is an elementary abelian $2$-group. Moreover, if $\chr\FF=2$ then $T$ is trivial, so $\cD=\FF$. 
\end{lemma}

\begin{proof}
Let $0\ne x\in\cD_s$ and $0\ne y\in\cD_t$. Applying $\varphi_0$ to both sides of equation \eqref{eq:def_beta}, we get $\varphi_0(y)\varphi_0(x)=\beta(s,t) \varphi_0(x)\varphi_0(y)$. But $\varphi_0(x)\in\cD_s$ and $\varphi_0(y)\in\cD_t$, so $\varphi_0(x)\varphi_0(y)=\beta(s,t) \varphi_0(y)\varphi_0(x)$, by equation \eqref{eq:def_beta} again, and hence $\beta(s,t)=\beta(s,t)^{-1}$ for all $s,t\in T$, which means that $\beta$ takes values in $\{\pm 1\}$. The nondegeneracy of $\bar{\beta}$ now implies that $\wb{T}$ is $2$-elementary and in fact trivial in the case $\chr\FF=2$. If the center $\KK$ equals $\FF$, we have $\wb{T}=T$, so we are done. Otherwise $\chr\FF\ne 2$, $\KK=\KK_e\oplus\KK_f$, and the elements of $\KK_f$ are skew-symmetric with respect to $\varphi_0$. Since $\rad\beta=\{e,f\}$, it suffices to prove that we cannot have an element $t\in T$ with $t^2=f$. Indeed, if $0\ne x\in\cD_t$ and $t^2=f$, then $\cD_f=\FF x^2$. But $\cD_t$ is a $1$-dimensional invariant subspace for $\varphi_0$, so $\varphi_0(x)=\pm x$, which implies that $x^2$ is symmetric --- a contradiction.    
\end{proof}

Thus, under the conditions of Lemma \ref{lm:2elementary}, we can consider $T$ as a vector space over the field $GF(2)$ and $\beta$ as an alternating bilinear form on this space. In particular, this allows us to get a version of symplectic basis for the degenerate case by taking a complement to the subspace $\rad\beta$:
\begin{equation}\label{eq:symplectic_basis_degen}
T=\langle a_1\rangle\times\langle b_1\rangle\times\cdots\times\langle a_m\rangle\times\langle b_m\rangle\times\langle f\rangle,
\end{equation}
where $\beta(a_i,b_i)=-1$ and the value of $\beta$ on all other pairs of basis elements is $1$. Therefore, $\cD$ is either the tensor product of quaternion algebras as in equation \eqref{eq:symbol_alg} with $\LL=\FF$ and all $\zeta_i=-1$ or else 
\begin{equation}\label{eq:symbol_alg_degen}
\cD\simeq (\mu_1,\nu_1)_{-1,\FF}\otimes_\FF\cdots\otimes_\FF (\mu_m,\nu_m)_{-1,\FF}\otimes_\FF\KK.
\end{equation} 
Now, since the homogeneous components $\cD_t$ are $1$-dimensional invariant subspaces of the involution $\varphi_0$, we obtain a function $\eta:T\to\{\pm 1\}$ defined by
\begin{equation}\label{eq:def_eta}
\varphi_0(x) = \eta(t) x\;\text{ for all }x\in\cD_t\text{ and }t\in T.
\end{equation}
In view of equation \eqref{eq:def_beta}, the condition $\varphi_0(xy)=\varphi_0(y)\varphi_0(x)$ translates to the following:
\begin{equation}\label{eq:eta_polarization}
\eta(st)=\eta(s)\eta(t)\beta(s,t)\;\text{ for all }s,t\in T.
\end{equation}
Thinking of $T$ as a vector space over $GF(2)$, this condition means that $\eta$ is a quadratic form and $\beta$ is its polarization.

It is well known that quadratic forms over $GF(2)$ whose polarization is nondegenerate are classified by their \emph{Arf invariant}, which in this case can simply be defined as the value that the quadratic form takes more often. Since $\eta$ takes values $\pm 1$, we will also define $\Arf(\eta)$ to take values $\pm 1$ rather than the usual $\bar{0}$ and $\bar{1}$. Since orthogonal involutions on a central simple algebra of degree $n$ are characterized by the property that the space of symmetric elements has a higher dimension than the space of skew-symmetric elements ($\frac{n(n+1)}{2}>\frac{n(n-1)}{2}$), we conclude that, if $\cD$ is central simple over $\FF$ and $\cD_e=\FF$, then
\begin{equation}\label{eq:Arf}
\varphi_0\text{ is }\begin{cases}
\text{orthogonal} & \text{if }\Arf(\eta)=1,\\
\text{symplectic} & \text{if }\Arf(\eta)=-1.
\end{cases}
\end{equation}
Note also that in this case $|T|=\dim_\FF\cD$ is an even power of $2$ and, disregarding the grading, $\cD\cong M_\ell(\Delta)$, where $\Delta$ is a central division algebra over $\FF$ and $|T|=\ell^2\dim_\FF\Delta$ by dimension count.

If $\cD$ is not central simple (but still assuming $\cD_e=\FF$), the radical of the polarization $\beta$ has dimension $1$ over $GF(2)$ and the quadratic form $\eta$ is \emph{nonsingular} in the sense that 
$\eta(s)$ is nontrivial for any nontrivial element $s$ 
in this  radical. Here $|T|$ is an odd power of $2$ and, disregarding the grading, $\cD$ is isomorphic to either $M_\ell(\Delta)$ or $M_\ell(\Delta)\times M_\ell(\Delta)^\op$ where $\Delta$ is a central division algebra over $\KK$ or $\FF$, respectively. 

Finally, if $\KK$ is a quadratic field extension of $\FF$ and $\cD_e=\KK$, then $T$ is not necessarily $2$-elementary and we cannot define a function $\eta$ by equation \eqref{eq:def_eta}, because $\varphi_0$ is not $\KK$-linear. In fact, an application of Hilbert's Theorem 90 to $\KK/\FF$ shows that we can choose $0\ne x_t\in\cD_t$ with $\varphi_0(x_t)=x_t$ for all $t\in T$. Also, $T$ can be nontrivial even if $\chr{\FF}=2$.

To summarize, we have the following possibilities under the assumption $\cD_e\subset\KK$:
\begin{itemize}
\item If $\KK=\FF$, then $\cD\simeq M_\ell(\Delta)$ as an ungraded algebra, where $\Delta$ is a central division algebra over $\FF$, the involution $\varphi_0$ (of the first kind) is either orthogonal or symplectic according to the rule \eqref{eq:Arf}, and the support $T$ is $2$-elementary with $|T|=\ell^2\dim_\FF\Delta$;
\item If $\KK$ is a separable quadratic field extension of $\FF$, then $\cD\simeq M_\ell(\Delta)$ as an ungraded algebra, where $\Delta$ is a central division algebra over $\KK$, the involution $\varphi_0$ is of the second kind, and
\[
|T|=\begin{cases}
\phantom{2}\ell^2\dim_\KK\Delta & \text{if the grading is of Type I},\\
2\ell^2\dim_\KK\Delta & \text{if the grading is of Type II}, 
\end{cases}
\]
where $T$ is $2$-elementary in the second case; 
\item If $\KK\simeq\FF\times\FF$, then $\cD\simeq M_\ell(\Delta)\times M_\ell(\Delta)^\op$ as an ungraded algebra, where $\Delta$ is a central division algebra over $\FF$, $\varphi_0$ is the exchange involution, the grading is of Type II, and its support $T$ is $2$-elementary with $|T|=2\ell^2\dim_\FF\Delta$.
\end{itemize}

\begin{remark}
The existence of $\varphi_0$ imposes restrictions on the class of $\Delta$ in the Brauer group: $[\Delta]$ must be an element of $\Br(\FF)$ of order at most $2$ if $\KK=\FF$ and an element of the kernel of the homomorphism $N_{\KK/\FF}:\Br(\KK)\to\Br(\FF)$ if $\KK$ is a separable quadratic field extension of $\FF$ (see e.g. \cite[\S 3]{KMRT}).
\end{remark}

\subsection{The real case}\label{sse:real_basics}

We now focus on the case $\FF=\RR$ (or, more generally, any real closed field). Here the possibilities for $\KK$ are $\RR$, $\CC$ and $\wt{\CC}\bydef\RR\times\RR$ (known as `split complex numbers'), while $\Delta$ can be only $\CC$ if $\KK=\CC$ and either $\RR$ or $\HH$ otherwise. As mentioned above, the classification up to equivalence of finite-dimensional real graded-division algebras that are simple as algebras and have abelian support was given in \cite{BZ16,R16}. This covers the possibilities $\KK\in\{\RR,\CC\}$; the remaining possibility $\KK=\wt{\CC}$ appeared in \cite{BKR_inv}. Here we will only need the cases with $\cD_e\subset\KK$, because otherwise the grading $\Gamma_\cD$ is not fine (see \cite[Proposition~20]{R16}, whose proof does not depend on the simplicity of $\cD$). 

We will first  discuss the case $\cD_e=\RR$, in which we will only consider $2$-elementary $T$, in view of Lemma \ref{lm:2elementary}. The symbol algebras of the form $(\mu,\nu)_{-1,\RR}$ that appear in the tensor product decompositions \eqref{eq:symbol_alg} and \eqref{eq:symbol_alg_degen} are isomorphic to $\HH$ if both $\mu$ and $\nu$ are negative and to $M_2(\RR)$ otherwise. To avoid fixing generators of $T$, it is convenient to introduce a function $\mu:T\to\RR^\times/\RR^{\times 2}\simeq\{\pm 1\}$ as follows:
\begin{equation}\label{eq:def_mu}
x^2\in\mu(t)\RR_{>0}\;\text{ for all }0\ne x\in\cD_t\text{ and }t\in T.
\end{equation}
Then equation \eqref{eq:def_beta} immediately implies that $\mu(st)=\mu(s)\mu(t)\beta(s,t)$ for all $s,t\in T$, so $\mu$ is a quadratic form with polarization $\beta$ and we can define an involution of $\cD$ by $x\mapsto\mu(t)x$ for all $x\in\cD_t$ and $t\in T$. This is called the \emph{distinguished involution} of $\cD$ in \cite{BKR_inv} and is characterized by the property $x\varphi_0(x)\in\RR_{>0}$ for all $x\in\cD^\times_\gr$. We are going to use it as $\varphi_0$ (recall Corollary \ref{cor:involutions_from_B}) except in the case $\KK=\wt{\CC}$, because the distinguished involution is trivial on $\KK$ in this case.  

In the case $\KK=\RR$, the restriction $\mu|_{\langle a_i,b_i\rangle}$ has trivial Arf invariant if and only if the corresponding quaternion algebra $(\mu_i,\nu_i)_{-1,\RR}$ in the tensor product decomposition \eqref{eq:symbol_alg} is split, because $\mu_i\in\mu(a_i)\RR_{>0}$ and $\nu_i\in\mu(b_i)\RR_{>0}$ by the definition of the function $\mu$. It follows that we have $\Delta=\RR$ if and only if $\Arf(\mu)=1$.

In the case $\KK=\wt{\CC}$, the same considerations apply to the decomposition \eqref{eq:symbol_alg_degen}, so we have $\Delta=\RR$ if and only if $\Arf(\bar{\mu})=1$ where $\bar{\mu}$ is the quadratic form induced on $\wb{T}\bydef T/\rad\beta$ by $\mu$, which is well-defined because in this case the restriction of $\mu$ to $\rad\beta$ is trivial. Also note that $\bar{\mu}$ has a nondegenerate polarization $\bar{\beta}$, so $\Arf(\bar{\mu})$ is defined.

We have the following analog of Proposition \ref{pr:Weyl_group_C}:

\begin{proposition}\label{pr:Weyl_group_R}
Let $(\cD,\Gamma_\cD)$ be a graded-division algebra over $\RR$ with an abelian support $T$ and $\cD_e=\RR$. Assume that $T$ is $2$-elementary and define a quadratic form $\mu:T\to\{\pm 1\}$ by condition \eqref{eq:def_mu}, where we regard $T$ as a vector space over $GF(2)$. Then $W(\Gamma_\cD)$ is the group 
\[
\Aut(T,\mu)\bydef\{\alpha\in\Aut(T)\mid\mu\bigl(\alpha(t)\bigr)=\mu(t)\text{ for all }t\in T\}.
\]
\end{proposition}

\begin{proof}
Let $\psi\in\Aut(\Gamma_\cD)$ and let $\alpha$ be the image of $\psi$ in $W(\Gamma_\cD)$. Applying $\psi$ to $x^2$ for $0\ne x\in\cD_t$ and using condition \eqref{eq:def_mu}, we get $\psi(x)^2\in\mu(t)\RR_{>0}$. Since $\psi(x)\in\cD_{\alpha(t)}$, we conclude that $\mu\bigl(\alpha(t)\bigr)=\mu(t)$ for all $t\in T$, i.e., $\alpha\in\Aut(T,\mu)$. Note that $\Aut(T,\mu)$ is a subgroup of $\Aut(T,\beta)$ where $\beta$ is the polarization of $\mu$. 

To prove the converse, we choose a basis $\{b_i\mid i\in I\}$ of $T$ as a vector space over $GF(2)$ and also choose a normalized element $X_t\in\cD_t$ for each $t\in T$, i.e., an element satisfying $X_t^2=\mu(t)$. Then the algebra $\cD$ is generated by the elements $X_i\bydef X_{b_i}$ with defining relations $X_i^2=\mu(b_i)$ and $X_i X_j=\beta(b_i,b_j)X_j X_i$ ($i,j\in I$). Now suppose $\alpha\in\Aut(T,\mu)$ and let $b'_i=\alpha(b_i)$. These elements form another basis of $T$, so $\cD$ is generated by $X'_i\bydef X_{b'_i}$. Since $\mu(b'_i)=\mu(b_i)$ and $\beta(b'_i,b'_j)=\beta(b_i,b_j)$, we have exactly the same defining relations for $X_i$ as for $X'_i$. Therefore, there exists an automorphism $\psi$ of the algebra $\cD$ sending $X_i\mapsto X'_i$. It follows that $\psi$ permutes the homogeneous components of $\cD$, i.e., $\psi\in W(\Gamma_\cD)$. By construction, the corresponding permutation on $T$ is exactly $\alpha$.
\end{proof}

The presentation of $\cD$ by generators and relations used in the above proof also shows that, over a real closed field, the finite-dimensional graded-division algebras with $2$-elementary abelian support and $1$-dimensional components are classified up to equivalence by the isomorphism class of their support regarded as a quadratic space over the field $GF(2)$. These latter are classified by their dimension, rank, and the Arf invariant of the associated nonsingular quadratic form (which is defined if the rank is even). Since for our purposes we need only the case $\dim_\RR Z(\cD)\le 2$, the quadratic space will have the rank equal to the dimension or one less.

Finally, we turn to the case $\cD_e=Z(\cD)\simeq\CC$ or, in other words, a Type~I grading with $\KK=\CC$. In this case the existence of a degree-preserving involution of the second kind does not force $T$ to be $2$-elementary. In fact, it is shown in \cite{BKR_inv} that such involutions always exist, but there is a single distinguished involution among them only in the case of $2$-elementary $T$. This involution is defined by the condition that, for any $t\in T$ and any \emph{symmetric} element $0\ne x\in\cD_t$ (recall that such an element always exists, defined up to a factor in $\RR$), we have $x^2\in\RR_{>0}$. In general, there is a \emph{distinguished class} of involutions, characterized by the same condition, but imposed only for $t$ satisfying $t^2=e$, which is equivalent to saying that $x\varphi_0(x)\in\RR_{>0}$ for all $0\ne x\in\cD_t$ with $t^2=e$. For example, if $|T|$ is odd then all degree-preserving involutions of the second kind belong to this class. However, for any involutions $\varphi_0$ and $\varphi'_0$ in this class, $(\cD,\varphi_0)$ and $(\cD,\varphi'_0)$ are isomorphic as graded algebras with involution, so it will not matter which one to choose (we make a specific choice below). 

There is also a nuance concerning the Weyl group in this case. Proposition \ref{pr:Weyl_group_C} (with $\LL=\CC$) tells us that the Weyl group of $\cD$ as a \emph{complex} algebra is $\Aut(T,\beta)$, but here we consider $\cD$ as a \emph{real} algebra, so the Weyl group may be larger (cf. \cite[Proposition 5.2]{R20}):  

\begin{proposition}\label{pr:Weyl_group_C_as_R}
Let $(\cD,\Gamma_\cD)$ be a graded-division algebra over $\CC$ with an abelian support $T$ and $\cD_e=\CC$.
Define an alternating bicharacter $\beta:T\times T\to\CC^\times$ by equation \eqref{eq:def_beta}. 
Then, regarding $\cD$ as an algebra over $\RR$, $W(\Gamma_\cD)$ is the stabilizer of the set $\{\beta,\iota\circ\beta\}$ in $\Aut(T)$, where $\iota$ denotes the complex conjugation. 
If $T$ is finite and $\beta$ is nondegenerate, then we can say more:
\[
W(\Gamma_\cD)=\begin{cases}
\Aut(T,\beta) & \text{if $T$ is $2$-elementary}, \\
\Aut(T,\beta)\rtimes\langle\tau\rangle & \text{otherwise},
\end{cases}
\]
where $\tau$ is the (involutive) automorphism of $T$ defined by $a_j\mapsto a_j^{-1}$ and $b_j\mapsto b_j$ for a symplectic basis of $T$ as in equation \eqref{eq:symplectic_basis}.
\end{proposition}

\begin{proof}
Consider the conjugate algebra $\wb{\cD}$, by which we mean $\cD$ as an algebra over $\RR$, but with the multiplication by scalars in $\CC$ twisted by $\iota$. Then a $\CC$-antilinear automorphism of $\cD$ is the same as a $\CC$-linear isomorphism $\cD\to\wb{\cD}$. Now, the bicharacter associated to $\wb{\cD}$ (using the same grading $\Gamma_\cD$) is clearly $\iota\circ\beta$. As in the proof of Proposition \ref{pr:Weyl_group_C}, $\cD$ and $\wb{\cD}$ are twisted group algebras of $T$, and the classes of the corresponding $2$-cocycles are determined by $\beta$ and $\iota\circ\beta$, respectively. Our first claim now follows from \cite[Theorem 2.13]{EKmon}.

Now suppose $T$ is finite and $\beta$ is nondegenerate. We claim that $\cD$ has a $\CC$-antilinear automorphism that permutes the components of $\Gamma_\cD$. Indeed, pick a symplectic basis $\{a_1,b_1,\ldots,a_m,b_m\}$ of $T$ and elements $0\ne X_j\in\cD_{a_j}$ and $0\ne Y_j\in\cD_{b_j}$  satisfying $X_j^{n_j}=1$ and $Y_j^{n_j}=1$. The algebra $\cD$ is generated over $\CC$ by these elements and defined by relations \eqref{eq:genrel} with $\mu_j=\nu_j=1$. $\wb{\cD}$ is also generated by these $X_j$ and $Y_j$, but its defining relations have the roots of unity $\zeta_j\bydef\beta(a_j,b_j)$ replaced by their complex conjugates $\zeta_j^{-1}=\beta(a_j^{-1},b_j)$. It follows that there exists an algebra isomorphism $\psi_0:\cD\to\wb{\cD}$ that sends $X_j\mapsto X_j^{-1}$ and $Y_j\mapsto Y_j$. Regarded as a $\CC$-antilinear automorphism of $\cD$, this $\psi_0$ permutes the components of $\Gamma_\cD$, and the corresponding permutation on $T$ is exactly the above automorphism $\tau$. It follows that $W(\Gamma_\cD)$ is generated by its normal subgroup $\Aut(T,\beta)$ and the element $\tau$. It remains to observe that $\tau$ is the identity if $T$ is $2$-elementary and does not belong to $\Aut(T,\beta)$ if $T$ is not $2$-elementary (since $\beta\ne\beta^{-1}$ in this case).   
\end{proof}

Let us summarize for future reference the classification of real graded-division algebras $\cD$ satisfying $\cD_e\subset Z(\cD)$ and admitting a degree-preserving involution $\varphi_0$ that makes them central simple as algebras with involution. We also fix a model for $\cD$ and $\varphi_0$ in each case. Below, the \emph{standard $\ZZ_2^2$-grading} on $\HH$ means the grading defined by $\deg(\mathbf{i})=(\bar{1},\bar{0})$ and $\deg(\mathbf{j})=(\bar{0},\bar{1})$, and the \emph{standard involution} (conjugation) on $\HH$, denoted by bar, maps $a+b\mathbf{i}+c\mathbf{j}+d\mathbf{k}\mapsto a-b\mathbf{i}-c\mathbf{j}-d\mathbf{k}$. Similarly, bar denotes the conjugation on $\CC$.
\begin{enumerate}
\item[(1)] If $\KK=\RR$, we have two cases according to $\Arf(\mu)\in\{\pm 1\}$:
\begin{itemize}
\item $\cD(2m;+1)=M_2(\RR)^{\otimes m}\simeq M_\ell(\RR)$ (via Kronecker product) where $\ell=2^m$, each factor $M_2(\RR)$ has the Pauli grading by the corresponding factor $\ZZ_2^2$ in $T=\ZZ_2^{2m}$, with the distinguished involution $\varphi_0(X)=X^T$ for all $X\in M_\ell(\RR)$;
\item $\cD(2m;-1)=M_2(\RR)^{\otimes m-1}\otimes\HH\simeq M_{\ell}(\HH)$ where $\ell=2^{m-1}$, each factor $M_2(\RR)$ has the Pauli grading by the corresponding factor $\ZZ_2^2$ and $\HH$ has the standard grading by the last factor $\ZZ_2^2$ in $T=\ZZ_2^{2m}$, with the distinguished involution $\varphi_0(X)=\wb{X}^T$ for all $X\in M_\ell(\HH)$. 
\end{itemize} 
\item[(2)] If $\KK=\CC$, we have two cases according to $\cD_e\in\{\RR,\CC\}$:
\begin{itemize}
\item $\cD\bigl(2m+1;\RR\bigr)=M_2(\RR)^{\otimes m}\otimes\CC\simeq M_\ell(\CC)$ where $\ell=2^m$, each factor $M_2(\RR)$ has the Pauli grading by the corresponding factor $\ZZ_2^2$ and $\CC$ has the nontrivial grading by the last factor $\ZZ_2$ in $T=\ZZ_2^{2m+1}$, with the distinguished involution $\varphi_0(X)=\wb{X}^T$ for all $X\in M_\ell(\CC)$.
\item $\cD\bigl(\ell_1,\ldots,\ell_m;\CC\bigr)=M_{\ell_1}(\CC)\otimes_\CC\cdots\otimes_\CC M_{\ell_m}(\CC)\simeq M_{\ell}(\CC)$ where $\ell=\ell_1\cdots\ell_m$, each factor $M_{\ell_j}(\CC)$ has the Pauli grading by the corresponding factor $\ZZ_{\ell_j}^2$ in $T=\ZZ_{\ell_1}^2\times\cdots\times\ZZ_{\ell_m}^2$, with the involution $\varphi_0=\varphi_{\ell_1}\otimes_\CC\cdots\otimes_\CC\varphi_{\ell_m}$, where 
\begin{equation}\label{eq:inv_bad_D}
\varphi_n(X)\bydef P^{-1}\wb{X}^TP\;\text{ for all }X\in M_n(\CC)
\end{equation} 
and $P$ is the matrix of the permutation that interchanges $1$ with $n-1$, $2$ with $n-2$,\ldots, $\lfloor\frac{n-1}{2}\rfloor$ with $\lfloor\frac{n+1}{2}\rfloor$;
\end{itemize}
\item[(3)] If $\KK=\wt{\CC}$, we have two cases according to $\Arf(\bar{\mu})\in\{\pm 1\}$:
\begin{itemize}
\item $\cD(2m+1;+1)=M_2(\RR)^{\otimes m}\otimes\wt{\CC}\simeq M_\ell(\RR)\times M_\ell(\RR)$ where $\ell=2^m$, each factor $M_2(\RR)$ has the Pauli grading by the corresponding factor $\ZZ_2^2$ and $\wt{\CC}$ has the nontrivial grading by the last factor $\ZZ_2$ in $T=\ZZ_2^{2m+1}$, with the involution $\varphi_0(X,Y)=(Y^T,X^T)$ for all $X,Y\in M_\ell(\RR)$;
\item $\cD(2m+1;-1)=M_2(\RR)^{\otimes m-1}\otimes\HH\otimes\wt{\CC}\simeq M_{\ell}(\HH)\times M_\ell(\HH)$ where $\ell=2^{m-1}$, each factor $M_2(\RR)$ has the Pauli grading by the corresponding factor $\ZZ_2^2$, $\HH$ has the standard grading by the next factor $\ZZ_2^2$, and $\wt{\CC}$ has the nontrivial grading by the last factor $\ZZ_2$ in $T=\ZZ_2^{2m+1}$, with the involution $\varphi_0(X,Y)=(\wb{Y}^T,\wb{X}^T)$ for all $X,Y\in M_\ell(\HH)$. 
\end{itemize} 
\end{enumerate}
To avoid redundancy in the case $\cD_e=Z(\cD)\simeq\CC$, we can take all $\ell_j$ to be prime powers.

Since $T$ is a finite abelian group, a nondegenerate alternating bicharacter $\beta$ sets up a group isomorphism $T\mapsto \wh{T}\bydef\Hom(T,\CC^\times)$ by sending $s\in T$ to the character $\beta(s,\cdot)$ on $T$. It follows that, if we define the `orthogonal complement' for a subgroup $A\subset T$ by 
\[
A^\perp\bydef\{s\in T\mid\beta(s,t)=1\text{ for all }t\in A\},
\]
then $|A^\perp|=[T:A]$ and, consequently, $(A^\perp)^\perp=A$ (since $A$ is contained in $(A^\perp)^\perp$ and has the same size). We will later need the following observation:

\begin{lemma}\label{lm:perp}
For any $n$, we have $(T_{[n]})^\perp=T^{[n]}$.
\end{lemma}

\begin{proof}
For any $s\in T_{[n]}$ and $t\in T$, we have $\beta(s,t^n)=\beta(s,t)^n=\beta(s^n,t)=1$, so $T^{[n]}\subset (T_{[n]})^\perp$. On the other hand, $T_{[n]}$ and $T^{[n]}$ are the kernel and image of the endomorphism $[n]:T\to T$ sending $t\mapsto t^n$. Hence $T^{[n]}$ and $(T_{[n]})^\perp$ have the same size, namely, $[T:T_{[n]}]$.
\end{proof}

We finish this section by observing that the grading on $\cD$ is obviously fine if $\cD_e=\RR$, whereas in the case $\cD_e=Z(\cD)\simeq\CC$, it is fine if and only if the support $T$ is not $2$-elementary. This latter was shown in \cite{R16} and can be easily seen from the proof of our Proposition \ref{pr:Weyl_group_C_as_R}. Indeed, if $T$ is $2$-elementary, then the involutive $\CC$-antilinear automorphism $\psi_0$ defined there, which corresponds to the mapping $X\mapsto\wb{X}$ in the matrix model above, gives us a real form of $\cD$ as a graded algebra, namely, $\cD(2m;+1)$, so the $T$-grading on $\cD\simeq\cD(2m;+1)\otimes\CC$ can be refined to a $T\times\ZZ_2$-grading, which yields the graded-division algebra $\cD(2m+1;\RR)$. On the other hand, if $T$ is not $2$-elementary, then $\cD$ does not admit a $\CC$-antilinear automorphism as a graded algebra (since $\beta\ne\beta^{-1}$), but any proper refinement would have to split $\cD_e=Z(\cD)$ and, therefore, any complex character of the new grading group that is nontrivial on the support of $Z(\cD)$ would have to act as a $\CC$-antilinear automorphism on $\cD$ and preserve the components of the original grading.

\section{Fine gradings on associative algebras with involution}\label{se:fine}

In this section we will construct a family $\{\cM_x\}_{x\in X}$ of associative algebras with involution, each equipped with a grading $\Gamma_x$ by an abelian group $\wt{G}_x$ such that $\cM_x$ is simple as a graded algebra with involution and satisfies DCC on graded left ideals, and will show that this family is complete in the following sense: for any $G$-graded associative algebra with involution $\cR$ that satisfies the same conditions, there exists $x\in X$ and a group homomorphism $\alpha:\wt{G}_x\to G$ such that $\cR$ is isomorphic to ${}^\alpha\cM_x$ as a $G$-graded algebra with involution. Since the support of $\cR$ generates an abelian subgroup of $G$, we will assume without loss of generality that $G$ is abelian.  All algebras in this section are assumed associative.

The parameter $x$ includes a graded-division algebra $\cD$, so technically $X$ is not a set, but it can be made a set if desired by choosing representatives of the equivalence classes of graded-division algebras, which we will indeed do when specializing to the finite-dimensional real case. For the construction of $\cM_x$, we do not make any assumptions on the ground field $\FF$ or on the dimension of $\cD$, but we will have to impose restrictions when investigating which of the gradings $\Gamma_x$ are fine. We roughly follow the approach of \cite[\S 3.2]{EKmon} (which itself goes back to \cite{E10}), but with significant changes required by our general setting: the homogeneous components of $\cD$ are not assumed $1$-dimensional and the support $T$ of $\cD$ is not necessarily $2$-elementary. Also, our approach to the presentation of the groups $\wt{G}_x$ is more conceptual than in \cite{EKmon}. 

\subsection{The graded algebras with involution $\cM(\cD,\varphi_0,q,s,\ul{d},\delta)$}\label{sse:def_MDpqsdd}

Let $\cD$ be a graded-division algebra and let $\varphi_0$ be a degree-preserving involution on $\cD$. Then the support $T$ of $\cD$ is an abelian group. Let $q$ and $s$ be nonnegative integers (not both zero), let $\delta\in\{\pm 1\}$ and let $\ul{d}=(d_1,\ldots,d_q)$ be a $q$-tuple of nonzero homogeneous elements of $\cD$ such that $\varphi_0(d_i)=\delta d_i$ for all $1\le i\le q$. (If $q=0$, we consider $\ul{d}$ to be empty.) 

Given this data, we first define the abelian group $\wt{G}=\wt{G}(T,q,s,\ul{t})$, where $t_i\bydef\deg d_i$. 
Let $k=q+2s$ and let $F$ be the free abelian group generated by the symbols $\tilde{g}_1,\ldots,\tilde{g}_k$. Then $\wt{G}$ is the quotient of $F\times T$ modulo the following relations:
\begin{equation}\label{eq:rel_Gtilde}
\tilde{g}_1^2 t_1^{-1}=\ldots=\tilde{g}_q^2 t_q^{-1}
=\tilde{g}_{q+1}\tilde{g}_{q+2}=\ldots=\tilde{g}_{q+2s-1}\tilde{g}_{q+2s}.
\end{equation}
(There are no relations if $q+s=1$.)

\begin{lemma}\label{lm:combinatorics}
Let $\pi:F\times T\to\wt{G}$ be the quotient map.
\begin{enumerate}
\item[(i)] 
The restriction of $\pi$ to $T$ is injective and hence $T$ can be identified with a subgroup of $\wt{G}$.
\item[(ii)]
The images of $\tilde{g}_1,\ldots,\tilde{g}_k$ under $\pi$ are pairwise distinct modulo $T$.
\item[(iii)] 
The images of the elements $\tilde{g}_i\tilde{g}_j^{-1}$ ($i\ne j$) under $\pi$ are pairwise distinct modulo $T$ except that the images of $\tilde{g}_i\tilde{g}_j^{-1}$ and $\tilde{g}_{\sigma(j)}\tilde{g}_{\sigma(i)}^{-1}$ are equal to one another modulo $T$, where $\sigma$ is the permutation of $\{1,\ldots,k\}$ that fixes each $i$ for $1\le i\le q$ and interchanges $q+2i-1$ and $q+2i$ for $1\le i\le s$.
\end{enumerate}
\end{lemma}

\begin{proof}
Let $r_i=\tilde{g}_i^2 t_i^{-1}$ for $1\le i\le q$ and $r_{q+i}=\tilde{g}_{q+2i-1}\tilde{g}_{q+2i}$ for $1\le i\le s$, so the kernel $K$ of the quotient map $T\times F\to\wt{G}$ is generated by the elements $r_ir_{i+1}^{-1}$ ($1\le i\le q+s-1$). It is clear that the projections of these elements to $F$, i.e., the elements $\tilde{g}_i^2\tilde{g}_{i+1}^{-2}$ ($1\le i\le q-1$), $\tilde{g}_{q+2i-1}\tilde{g}_{q+2i}\tilde{g}_{q+2i+1}^{-1}\tilde{g}_{q+2i+2}^{-1}$ ($1\le i\le s-1$), and $\tilde{g}_q^2\tilde{g}_{q+1}^{-1}\tilde{g}_{q+2}^{-1}$ (defined if both $q$ and $s$ are nonzero) form a basis of the subgroup that they generate in $F$. It follows that $K$ intersects $T$ trivially, which proves part (i).

Now let $\pi_0$ be the quotient map $F\times T\to\wt{G}/T$. Reading the relations \eqref{eq:rel_Gtilde} modulo $T$, it is straightforward to verify that $\pi_0\bigl(\tilde{g}_i\tilde{g}_j^{-1}\bigr)=\pi_0\bigl(\tilde{g}_{\sigma(j)}\tilde{g}_{\sigma(i)}^{-1}\bigr)$. To prove that no other equalities occur, and also that $\pi_0(\tilde{g}_i)\ne\pi_0(\tilde{g}_j)$ for $i\ne j$, we may replace $\pi_0$ by any homomorphism $\pi_1:F\times T\to G_1$ that kills the kernel of $\pi_0$. We can take $G_1=\ZZ_2^q\times\ZZ^s$ and define $\pi_1$ as follows: $t\mapsto 0$ for all $t\in T$, $\tilde{g}_i\mapsto\bar{e}_i$ for all $1\le i\le q$, $\tilde{g}_{q+2i-1}\mapsto e_i$ and $\tilde{g}_{q+2i}\mapsto-e_i$ for all $1\le i\le s$, where $\{\bar{e}_1,\ldots,\bar{e}_q\}$ is the standard basis of $\ZZ_2^q$ and $\{e_1,\ldots,e_s\}$ is that of $\ZZ^s$. Since $\pi_1(r_i)=0$ for all $1\le i\le q+s$ and the kernel of $\pi_0$ is generated by $T$ and $r_i r_j^{-1}$, we see that it is indeed killed by $\pi_1$. Clearly, the elements $\bar{e}_1,\ldots,\bar{e}_q,\pm e_1,\ldots,\pm e_s$ are pairwise distinct, which proves part (ii).

To finish the proof of (iii), note that, for $i\ne j$ in the set $\{1,\ldots,k\}$, we have $\sigma(j)<\sigma(i)$ if and only if $i>j$, except when $\{i,j\}=\{q+2r-1,q+2r\}$ for some $r\in\{1,\ldots,s\}$. Therefore, it is sufficient to consider where $\pi_1$ maps the elements $\tilde{g}_i\tilde{g}_j^{-1}$ with $i<j$ and also the elements $\tilde{g}_{q+2r}\tilde{g}_{q+2r-1}^{-1}$. It is clear that the images of $\tilde{g}_i\tilde{g}_j^{-1}$ with $i<j$ and $i\le q$ are distinct from each other and from the images of the remaining elements $\tilde{g}_{q+i}\tilde{g}_{q+j}^{-1}$ with $i<j$ or $(i,j)=(2r,2r-1)$. But these latter are mapped bijectively onto the subset $\{\pm e_i\pm e_j\mid i<j\}\cup\{\pm 2e_i\}$ of $\ZZ^s$ (the root system of type $C_s$).
\end{proof}

In view of Lemma \ref{lm:combinatorics}, we will regard $T$ as a subgroup of $\wt{G}$ and will also identify the elements $\tilde{g}_1,\ldots,\tilde{g}_k$ with their images in $\wt{G}$.

Now let $\cV=\cD^{[\tilde{g}_1]}\oplus\cdots\oplus\cD^{[\tilde{g}_k]}$ or, in other words, the right $\cD$-module $\cD^k$ on which we defined a $\wt{G}$-grading by assigning the elements $v_1,\ldots,v_k$ of the standard $\cD$-basis degrees $\tilde{g}_1,\ldots,\tilde{g}_k$, respectively. Note that since these $\tilde{g}_i$ are not only distinct, but distinct modulo $T$, all homogeneneous components of $\cV$ have dimension $1$ as right vector spaces over the division algebra $\cD_e$.

Then the graded algebra $\cR=\End_\cD(\cV)$ can be identified with the matrix algebra $M_k(\cD)\simeq M_k(\FF)\otimes\cD$ with the following $\wt{G}$-grading:
\begin{equation}\label{eq:gradingMDpqsdd}
\deg(E_{ij}\otimes d)=\tilde{g}_i \tilde{g}_j^{-1} t\;\text{ for any }0\ne d\in\cD_t.
\end{equation}
Lemma \ref{lm:combinatorics} implies that the homogeneous components of $\cR$ are as follows ($t\in T$):
\begin{align}
&\cR_t = \lspan{E_{1,1},\ldots,E_{k,k}}\otimes\cD_t, \label{eq:kd} \\
&\cR_{\tilde{g}_{q+2r-1}\tilde{g}_{q+2r}^{-1}t} = \FF E_{q+2r-1,q+2r}\otimes\cD_t, \label{eq:1d1} \\ 
&\cR_{\tilde{g}_{q+2r-1}^{-1}\tilde{g}_{q+2r}t} = \FF E_{q+2r,q+2r-1}\otimes\cD_t, \label{eq:1d2} \\
&\cR_{\tilde{g}_i\tilde{g}_j^{-1}t} = \FF E_{i,j}\otimes\cD_t\,\oplus\, \FF E_{\sigma(j),\sigma(i)}\otimes\cD_{t'}\;
\text{ with }t' \bydef \tilde{g}_i\tilde{g}_{\sigma(i)}\tilde{g}_j^{-1}\tilde{g}_{\sigma(j)}^{-1}t\in T, \label{eq:2d} 
\end{align}
where, in equation \eqref{eq:2d}, $i<j$ and $(i,j)\ne(q+2r-1,q+2r)$ for $1\le r\le s$.

Finally, there is an involution $\varphi$ on this graded algebra given by the nondegenerate $\varphi_0$-sesquilinear form $B$ represented (relative to the standard $\cD$-basis of $\cV$) by the following block-diagonal matrix (with $q$ blocks of size $1$ and $s$ of size $2$):
\begin{equation}\label{eq:PhiMDpqsdd}
\Phi=\diag\left(d_1,\ldots,d_q,
\begin{bmatrix} 0&1 \\ \delta&0 \end{bmatrix},\ldots,\begin{bmatrix} 0&1 \\ \delta&0 \end{bmatrix}\right).
\end{equation}
Indeed, the defining relations \eqref{eq:rel_Gtilde} of $\wt{G}$ ensure that $B$ is homogeneous of degree $\tilde{g}_0$, where $\tilde{g}_0^{-1}$ is the common value of the expressions that are forced to be equal to each other by those relations. Moreover, since $\varphi_0(d_i)=\delta d_i$, applying $\varphi_0$ to each entry of matrix $\Phi$ and transposing, we get $\delta\Phi$, hence $\wb{B}=\delta B$. Therefore, the adjunction with respect to $B$ is a degree-preserving involution on $\cR$. In matrix form, it is given by equation \eqref{eq:phi_with_matrices}. Note that $\varphi$ interchanges the two direct summands in the right-hand side of equation \eqref{eq:2d}.

\begin{df}\label{df:M}
The $\wt{G}(T,q,s,\ul{t})$-graded algebra with involution constructed above will be denoted $\cM(\cD,\varphi_0,q,s,\ul{d},\delta)$, and its grading by $\Gamma_\cM(\cD,\varphi_0,q,s,\ul{d},\delta)$.  
\end{df}

We will show in the next section  that the subgroup $\wt{G}^0(T,q,s,\ul{t})$ of $\wt{G}(T,q,s,\ul{t})$ generated by the support of $\Gamma_\cM(\cD,\varphi_0,q,s,\ul{d},\delta)$ is the universal group of this grading (see Proposition \ref{pr:G0_is_universal}) and will determine its isomorphism class in the case of finite $T$: the torsion part is given by equation \eqref{eq:iso_type_U} and the free part is $\ZZ^s$.

The proof of the following result is essentially that of \cite[Theorem 3.31]{EKmon} (which is itself a graded version, first given in \cite{E10}, of a classical argument in linear algebra), but we include it here for completeness.

\begin{theorem}\label{th:completeness}
Let $G$ be an abelian group, let $\cR$ be a $G$-graded algebra that is graded-simple and 
satisfies DCC on graded left ideals, and let $\varphi$ be a degree-preserving involution on $\cR$. 
Then there exist a graded-division algebra $\cD$ whose support $T$ is a subgroup of $G$, 
a degree-preserving involution $\varphi_0$ on $\cD$, integers $q\ge 0$ and $s\ge 0$ (not both zero), 
$\delta\in\{\pm 1\}$, a $q$-tuple $\ul{d}$ of elements $0\ne d_i\in\cD_{t_i}$ satisfying $\varphi_0(d_i)=\delta d_i$, 
and a group homomorphism $\alpha:\wt{G}(T,q,s,\ul{t})\to G$ with $\alpha|_T=\id_T$ such that 
${}^\alpha\cM(\cD,\varphi_0,q,s,\ul{d},\delta)$ is isomorphic to $\cR$ as a $G$-graded algebra with involution. 
\end{theorem}

\begin{proof}
Up to isomorphism of $G$-graded algebras, $\cR$ is of the form $\End_\cD(\cV)$, for a graded-division algebra $\cD$ and a graded right $\cD$-module of nonzero finite rank, and $\varphi$ is determined by a pair $(\varphi_0,B)$ as in Theorem \ref{th:involutions_from_B}. We show by induction on the rank $k$ of $\cV$ that there exists a graded $\cD$-basis $\{v_1,\ldots,v_k\}$ such that $B$ is represented by the matrix $\Phi$ in equation \eqref{eq:PhiMDpqsdd}, for some $q$, $s$ and $\ul{d}$ (with $q+2s=k$). 

If $k=1$, then we can take for $v_1$ any nonzero homogeneous element of $\cV$ and set $d_1\bydef B(v_1,v_1)$, which belongs to $\cD^\times_\gr$ because $B$ is homogeneous and nondegenerate. Now suppose $k>1$. If there exists a homogeneous element $v_1\in\cV$ with $B(v_1,v_1)\ne 0$, we include it in our graded basis and again take $d_1\bydef B(v_1,v_1)$. Let $\cV_1=v_1\cD$ and let $\cW=\cV_1^\perp\bydef\{w\in\cV\mid B(w,v)=0\text{ for all }v\in\cV_1\}$. Then $\cW$ is a $\cD$-submodule of $\cV$, graded because $B$ is homogeneous and of rank $k-1$ because $B$ is nondegenerate. Moreover, $\cV_1\cap\cW=0$ because $B$ restricts to a nondegenerate form on $\cV_1$. Therefore, $\cV=\cV_1\oplus\cW$. By induction hypothesis, $\cW$ has a graded $\cD$-basis $\{v_2,\ldots,v_k\}$ of the desired form, and we have $B(v_1,v_i)=\delta\varphi_0(B(v_i,v_1))=0$ for all $i>1$, which completes the proof in this case. Now if $B(v,v)=0$ for all homogeneous elements $v\in\cV$, pick some nonzero homogeneous $v_1\in\cV$. Since $B$ is nondegenerate, there exists a homogeneous $v_2\in\cV$ such that $B(v_1,v_2)\ne 0$ and hence $v_2\notin v_1\cD$. Replacing $v_2$ with $v_2 B(v_1,v_2)^{-1}$, we obtain $B(v_1,v_2)=1$ and hence $B(v_2,v_1)=\delta\varphi_0(1)=\delta$. The proof is then completed as above, using $\cV_1\bydef v_1\cD\oplus v_2\cD$.

Now let $g_i=\deg v_i$ and $g_0=\deg B$. Then the fact that $B(v_i,v_i)=d_i\ne 0$ for $1\le i\le q$ and $B(v_{q+2j-1},v_{q+2j})=1$ for $1\le j\le s$ implies the following relations:
\[
g_1^2 t_1^{-1}=\ldots=g_q^2 t_q^{-1}=g_{q+1}g_{q+2}=\ldots=g_{q+2s-1}g_{q+2s}=g_0^{-1}.
\]
Looking at the defining relations \eqref{eq:rel_Gtilde} of 
$\wt{G}=\wt{G}(T,q,s,\ul{t})$, we see that there exists a group homomorphism 
$\alpha:\wt{G}\to G$ such that $\alpha(t)=t$ for all $t\in T$ and 
$\alpha(\tilde{g}_i)=g_i$ for all 
$1\le i\le k$. Then ${}^\alpha\cD$ can be identified with $\cD$, and the graded 
$\cD$-module 
${}^\alpha\bigl(\cD^{[\tilde{g}_1]}\oplus\cdots\oplus\cD^{[\tilde{g}_k]}\bigr)$ is isomorphic to $\cV$ by sending the standard basis to $\{v_1,\ldots,v_k\}$. The result follows.
\end{proof}

\subsection{The graded algebras with involution $\cM^\ex(\cD,k)$}

To account for simple graded algebras with involution that are not simple as graded algebras, we introduce another family:

\begin{df}\label{df:Mex}
Let $\cD$ be a graded-division algebra with abelian support $T$ and let $k$ be a positive integer. 
Let $\wt{G}(T,k)\bydef F\times T$, where $F$ is the free abelian group generated by the symbols $\tilde{g}_1,\ldots,\tilde{g}_k$. 
\begin{enumerate}
\item[(i)]
The $\wt{G}(T,k)$-graded algebra $M_k(\cD)$ defined by equation \eqref{eq:gradingMDpqsdd} will be denoted by $\cM(\cD,k)$ and its grading by 
$\Gamma_\cM(\cD,k)$.
\item[(ii)]
Using the same grading on the opposite algebra, we obtain a $\wt{G}(T,k)$-graded algebra $\cM(\cD,k)\times\cM(\cD,k)^\op$ so that the exchange involution 
$\ex:(x,y)\mapsto(y,x)$ is degree-preserving. The resulting graded algebra with involution will be denoted by $\cM^\ex(\cD,k)$ and its grading by 
$\Gamma_{\cM^\ex}(\cD,k)$.
\end{enumerate} 
\end{df}

By \cite[Proposition 2.35]{EKmon}, the subgroup $\wt{G}^0(T,k)\simeq T\times\ZZ^{k-1}$ of $\wt{G}(T,k)$ generated by the support of $\Gamma_\cM(\cD,k)$ is the universal abelian group of this grading and hence the universal group of 
$\cM^\ex(\cD,k)$.

The proof of the following result is left as an exercise:

\begin{proposition}\label{pr:completeness}
Let $G$ be an abelian group, let $\cR$ be a $G$-graded algebra that satisfies DCC on graded left ideals, and let $\varphi$ be a degree-preserving involution on $\cR$. If $(\cR,\varphi)$ is simple as a graded algebra with involution, but $\cR$ is not graded-simple, then there exist 
a graded-division algebra $\cD$ whose support $T$ is a subgroup of $G$, an integer $k\ge 1$
and a group homomorphism $\alpha:\wt{G}(T,k)\to G$ with $\alpha|_T=\id_T$ such that ${}^\alpha\cM^\ex(\cD,k)$ is isomorphic to $\cR$ as a $G$-graded algebra with involution. \qed
\end{proposition}

Together with Theorem \ref{th:completeness}, this gives the following

\begin{corollary}\label{cor:completeness}
Let $\varphi$ be an involution on an artinian algebra $\cR$. If $(\cR,\varphi)$ is simple as an algebra with involution, then, for any $G$-grading $\Gamma$ on $(\cR,\varphi)$, exactly one of the following holds:
\begin{itemize}
\item $\Gamma$ is the image of some ${}^\alpha\Gamma_{\cM^\ex}(\cD,k)$ under an isomorphism of algebras with involution, where $\cD$ is a graded-division algebra whose support $T$ is a subgroup of $G$ and $\alpha:\wt{G}(T,k)\to G$ is a group homomorphism with $\alpha|_T=\id_T$; 
\item $\Gamma$ is the image of some ${}^\alpha\Gamma_\cM(\cD,\varphi_0,q,s,\ul{d},\delta)$ under an isomorphism of algebras with involution, where $\cD$ is a graded-division algebra whose support $T$ is a subgroup of $G$, $\varphi_0$ is a degree-preserving involution on $\cD$, and $\alpha:\wt{G}(T,q,s,\ul{t})\to G$ is a group homomorphism with $\alpha|_T=\id_T$.\qed 
\end{itemize}
\end{corollary}


\subsection{Fineness criteria}\label{ss:fineness}

It remains to check which of the gradings constructed in the previous subsections are 
indeed fine (as gradings on algebras with involution). 
First we will deal with the gradings $\Gamma_\cM(\cD,\varphi_0,q,s,\ul{d},\delta)$
in Definition \ref{df:M}.

\begin{theorem}\label{th:GMfine}
Let $\cD$ be a finite-dimensional graded-division algebra, with support $T$, 
and let $\varphi_0$ 
be a degree-preserving
involution on $\cD$. Let $q$ and $s$ be nonnegative integers (not both zero), let $\delta\in\{\pm 1\}$, and let $\ul{d}$ be a $q$-tuple of elements $d_i\in\cD^\times_\gr$ satisfying $\varphi_0(d_i)=\delta d_i$. Assume that $(q,s)\neq (2,0)$ and that the grading $\Gamma_\cD$ on $\cD$ is fine.
Then the grading $\Gamma_\cM(\cD,\varphi_0,q,s,\ul{d},\delta)$ is fine.
\end{theorem}

\begin{proof}
The homogeneous components of the grading $\Gamma\bydef\Gamma_\cM(\cD,\varphi_0,q,s,\ul{d},\delta)$ on $\cR\bydef\cM(\cD,\varphi_0,q,s,\ul{d},\delta)$ are given in equations \eqref{eq:kd}--\eqref{eq:2d}. As already noted, the involution $\varphi:X\mapsto\Phi^{-1}\varphi_0(X^T)\Phi$, with matrix $\Phi$ given by equation \eqref{eq:PhiMDpqsdd}, interchanges the two summands in \eqref{eq:2d}. Assume that a $G'$-grading $\Gamma'$ is a refinement of $\Gamma$ as a grading on the algebra with involution $(\cR,\varphi)$. We shall follow several steps:

\smallskip

\noindent$\bullet$\quad Assume that a nonzero element of the form $E_{ii}\otimes d$ is 
homogeneous for $\Gamma'$. Then the idempotent $\veps_i\bydef E_{ii}\otimes 1$ is also homogeneous 
for $\Gamma'$. 

Indeed, $\cD$ is isomorphic to the subalgebra $E_{ii}\otimes\cD\subset\cR$ (whose unit element is $\veps_i$). 
Since $\cD$ is finite-dimensional, $d$ is algebraic over $\FF$, and hence the order of $\deg_{\Gamma'}(E_{ii}\otimes d)$ is finite. Also note that $d$ is homogeneous for $\Gamma_\cD$ and hence invertible. Replacing $E_{ii}\otimes d$ by one of its powers, we can assume that its degree is the identity element $e'$ of $G'$. 
Since the constant term of the minimal polynomial of $d$ is nonzero, the element $E_{ii}\otimes 1$ is an $\FF$-linear combination of powers of $E_{ii}\otimes d$, and hence it is homogeneous (of degree $e'$).
 
\smallskip

\noindent$\bullet$\quad Now we repeat the arguments in \cite[Theorem 3.30]{EKmon} to check that the elements $E_{11}\otimes 1, \ldots, E_{q+2s,q+2s}\otimes 1$ are homogeneous for $\Gamma'$.
This is trivial if $q=1$ and $s=0$, so we may assume that $s\geq 1$ or $q\geq 3$.
In view of equations \eqref{eq:1d1} and \eqref{eq:1d2}, for $i=1,\ldots,s$, we can find nonzero $\Gamma'$-homogeneous
elements $E_{q+2i-1,q+2i}\otimes c_i$ and $E_{q+2i,q+2i-1}\otimes d_i$. Multiplying these, we get $\Gamma'$-homogeneous elements $E_{q+2i-1,q+2i-1}\otimes c_i d_i$ and $E_{q+2i,q+2i}\otimes d_i c_i$, to which we can apply the previous step, because $c_i$ and $d_i$ are $\Gamma_\cD$-homogeneous and hence invertible.

If $q\geq 3$, take three different indices $1\leq i,j,k\leq q$. We will prove that $E_{ii}\otimes 1$ is $\Gamma'$-homogeneous. In view of equation \eqref{eq:2d}, we can find a $\Gamma'$-homogeneous element of the form $E_{ij}\otimes d+E_{ji}\otimes d'$ where $d\ne 0$. If $d'=0$ then, applying $\varphi$ to $E_{ij}\otimes d$, we obtain another nonzero $\Gamma'$-homogeneous element $E_{ji}\otimes d''$. It follows that $E_{ii}\otimes dd''$ is $\Gamma'$-homogeneous, so the previous step applies. Now assume $d,d'\neq 0$. Then $\bigl(E_{ij}\otimes d+ E_{ji}\otimes d'\bigr)^2=
E_{ii}\otimes dd'+ E_{jj}\otimes d'd$ is $\Gamma'$-homogeneous, too. In the same vein, we pick a $\Gamma'$-homogeneous element of the form $E_{ii}\otimes a+E_{kk}\otimes b$ where $a\ne 0$. It follows that the  element 
$E_{ii}\otimes dd'a=\bigl(E_{ii}\otimes dd'+ E_{jj}\otimes d'd\bigr)
\bigl(E_{ii}\otimes a+ E_{kk}\otimes b\bigr)$ is $\Gamma'$-homogeneous, so the previous step applies.

Finally, if $1\leq i\leq q$ and $s\geq 1$, we can pick $\Gamma'$-homogeneous elements $x=E_{i,q+1}\otimes c+E_{q+2,i}\otimes c'$ and $y=E_{q+1,i}\otimes d+E_{i,q+2}\otimes d'$ where $c,d\neq 0$. Then 
$xy=E_{ii}\otimes a+ E_{q+2,q+2}\otimes a'$, with $a=cd\neq 0$ and $a'=c'd'$. If $a'=0$ then we can apply the previous step, so assume $c',d'\neq 0$. Then $yx=E_{ii}\otimes b+E_{q+1,q+1}\otimes b'$, with $b,b'\neq 0$. 
The nonzero product of $xy$ and $yx$ is of the form $E_{ii}\otimes f$, and the first step applies.

\smallskip

\noindent$\bullet$\quad For any $i=1,\ldots,q+2s$, recall the isomorphism $\cD\simeq E_{ii}\otimes\cD=\veps_i\cR\veps_i$. Since we have shown that $\veps_i$ is a homogeneous idempotent for $\Gamma'$, the subalgebra $\veps_i\cR\veps_i\subset\cR$ is $G'$-graded, so the isomorphism induces a $G'$-grading on $\cD$, which is a refinement of $\Gamma_\cD$.  
But $\Gamma_\cD$ is fine, so the restrictions of $\Gamma$ and $\Gamma'$ to $\veps_i\cR\veps_i$ have the same homogeneous components (although with different labels).

\smallskip

\noindent$\bullet$\quad The homogeneous components of $\Gamma$ in equations \eqref{eq:1d1} and \eqref{eq:1d2} do not split in $\Gamma'$, because we know from the previous step that all elements in $E_{q+2i-1,q+2i-1}\otimes\cD_e$ and $E_{q+2i,q+2i}\otimes\cD_e$ are $\Gamma'$-homogeneous of degree $e'$.

\smallskip

\noindent$\bullet$\quad Now consider any of the homogeneous components $E_{i,j}\otimes \cD_t\oplus E_{\sigma(j),\sigma(i)}\otimes \cD_{t'}$ of $\Gamma$ in equation \eqref{eq:2d}. 
Since all elements in $E_{i,i}\otimes\cD_e$ are $\Gamma'$-homogeneous of degree $e'$,  
it follows that all elements of $E_{i,j}\otimes \cD_t$ are $\Gamma'$-homogeneous of the same degree, and similarly for $E_{\sigma(j),\sigma(i)}\otimes \cD_{t'}$. Therefore, either $E_{i,j}\otimes \cD_t\oplus E_{\sigma(j),\sigma(i)}\otimes \cD_{t'}$ does not split in $\Gamma'$ or else it splits as the direct sum of $E_{i,j}\otimes \cD_t$ and $E_{\sigma(j),\sigma(i)}\otimes \cD_{t'}$. But $E_{i,j}\otimes\cD_t$ is not $\varphi$-invariant, so we conclude that the homogeneous components of $\Gamma$ in equation \eqref{eq:2d} do not split in $\Gamma'$.

\smallskip

\noindent$\bullet$\quad Finally, for any nonzero homogeneous element $d\in\cD$ and any $1\leq i<j\leq q+2s$, the elements 
$E_{i,i}\otimes d=(E_{i,j}\otimes 1)(E_{j,i}\otimes d)$ and
$E_{j,j}\otimes d=(E_{j,i}\otimes d)(E_{i,j}\otimes 1)$ are $\Gamma'$-homogeneous of the same
degree, so the homogeneous components of $\Gamma$ in equation \eqref{eq:kd} do not split in $\Gamma'$.

\smallskip

We conclude that none of the homogeneous components of $\Gamma$ split in $\Gamma'$ and, therefore, $\Gamma$ does not admit proper refinements.
\end{proof}

The case $(q,s)=(2,0)$ is an exception in Theorem \ref{th:GMfine}. The next two results deal with this case under some restrictions.

\begin{proposition}\label{pr:GMfineq2s0}
Let $\cD$ be a finite-dimensional graded-division algebra with $1$-dimensional homogeneous components and let $\varphi_0$ be a degree-preserving involution on $\cD$ such that $(\cD,\varphi_0)$ is central simple as an algebra with involution. Let $\delta\in\{\pm 1\}$ and let $\ul{d}=(d_1,d_2)$ where $d_i\in\cD^\times_\gr$ satisfy $\varphi_0(d_i)=\delta d_i$ for $i=1,2$. Then the grading $\Gamma_\cM(\cD,\varphi_0,2,0,\ul{d},\delta)$ is fine if and only if $\deg d_1\neq \deg d_2$.
\end{proposition}

\begin{proof}
We will use notation as in the proof of Theorem \ref{th:GMfine} and also  
$t_i\bydef\deg d_i$ and $G\bydef\wt{G}(T,2,0,\ul{t})$.
From equations \eqref{eq:kd} and \eqref{eq:2d}, $\Gamma$ has the following homogeneous components: 
$\cR_t=E_{11}\otimes \cD_t\oplus E_{22}\otimes \cD_t$ and 
$\cR_{zt}=E_{12}\otimes\cD_t\oplus E_{21}\otimes\cD_{t_1 t_2^{-1} t}$ where $t\in T$ and $z\bydef\tilde{g}_1\tilde{g}_2^{-1}\in G$, so $z^2=t_1 t_2^{-1}$ by equation \eqref{eq:rel_Gtilde}.
Recall from Lemma \ref{lm:2elementary} that, under our hypotheses on $\cD$, the support $T$ is an elementary $2$-group.

Suppose $t_1\neq t_2$. If $\veps_1$ and $\veps_2$ are $\Gamma'$-homogeneous, then the arguments at the end of the proof of Theorem \ref{th:GMfine} show that $\Gamma'$ is not a proper refinement of $\Gamma$. So, assume that $\cR_e=\lspan{\veps_1,\veps_2}$ splits in $\Gamma'$. It follows that $\dim\cR_{e'}=1$ and hence $(\cR,\Gamma')$ is a graded-division algebra. Applying Lemma \ref{lm:2elementary} again, we see that the support $T'$ of $\Gamma'$ is an elementary $2$-group. Since $T'$ is the universal group of $\Gamma'$, the coarsening $\Gamma$ is obtained by a group homomorphism $T'\to G$. This is a contradiction, since $z$ is in the support of $\Gamma$ and $z^2\neq e$.
 
Conversely, suppose $t_1=t_2$ and hence $z^2=e$ and $d_2=\mu d_1$ for some $\mu\in\FF^\times$. Then the subspaces
\[
\begin{aligned}
\cR_{(t,\bar 0)}&=\begin{bmatrix} 1&0\\ 0&1\end{bmatrix}\otimes\cD_t,&
\cR_{(t,\bar 1)}&=\begin{bmatrix} 1&0\\ 0&-1\end{bmatrix}\otimes\cD_t,\\
\cR_{(zt,\bar 0)}&=\begin{bmatrix} 0&\mu\\ 1&0\end{bmatrix}\otimes\cD_t,&
\cR_{(zt,\bar 1)}&=\begin{bmatrix} 0&-\mu\\ 1&0\end{bmatrix}\otimes\cD_t,
\end{aligned}
\]
are all invariant under $\varphi$ and constitute a $G\times \ZZ_2$-grading that is a proper refinement of $\Gamma$ (and turns $\cR$ into a graded-division algebra). 
\end{proof}

As we are interested in the real case, recall from Subsection \ref{sse:real_basics} that if $(\cD,\varphi_0)$ is central simple over $\RR$ then $\Gamma_\cD$ is fine if and only if either $\cD_e=\RR$ (and so Proposition \ref{pr:GMfineq2s0} applies) or $\cD_e=Z(\cD)\simeq\CC$ and $T$ is not $2$-elementary. The next result deals with the exceptional case $(q,s)=(2,0)$ in this latter situation, not covered in Proposition \ref{pr:GMfineq2s0}. 

\begin{proposition}\label{pr:GMfineq2s0_2}
Let $\cD$ be a simple finite-dimensional graded-division algebra over $\RR$ with $\cD_e=Z(\cD)\simeq\CC$ and let  
$\varphi_0$ be a degree-preserving second kind involution on $\cD$. 
Then the grading $\Gamma_\cM(\cD,\varphi_0,2,0,\ul{d},\delta)$ is never fine.
\end{proposition}

\begin{proof}
We continue using notation in the proof of Proposition \ref{pr:GMfineq2s0}. The involution $\varphi$ is given by
\[
\varphi:
\begin{bmatrix} a & b \\ c & d\end{bmatrix} \mapsto 
\begin{bmatrix} d_1^{-1} & 0 \\ 0 & d_2^{-1}\end{bmatrix}
\begin{bmatrix}\varphi_0(a) & \varphi_0(c) \\ \varphi_0(b) & \varphi_0(d)\end{bmatrix} 
\begin{bmatrix} d_1 & 0 \\ 0 & d_2\end{bmatrix}.
\]
Consider the inner automorphism $\psi:X\mapsto AXA^{-1}$ where 
$A=\left[\begin{smallmatrix} 0&d_2\\ d_1&0\end{smallmatrix}\right]$.
We observe the following properties of $A$ and $\psi$:

\begin{itemize}

\item $A$ is homogeneous of degree $z t_2=z^{-1} t_1$ and hence $\psi\in\Stab(\Gamma)$. 

\item $\psi|_{\cR_e}$ is the exchange automorphism on $\cR_e\simeq \cD_e\times\cD_e\simeq\CC\times\CC$.

\item $A^2=\begin{bmatrix} d_2d_1 & 0 \\ 0 & d_1d_2 \end{bmatrix}
=\begin{bmatrix} \beta(t_2,t_1) & 0 \\ 0 & 1 \end{bmatrix}
\begin{bmatrix} d_1d_2 & 0 \\ 0 & d_1d_2 \end{bmatrix}$. 
As $\beta(t_2,t_1)$ is a root of unity of degree dividing the order of $t_1$, we have $A^{2N}\in Z(\cR)$ where $N$ is the least common multiple of the orders of $t_1$ and $t_1t_2$. Therefore, the order of $\psi$ is finite, hence $\psi$ is diagonalizable as a $\CC$-linear operator and its eigenvalues are roots of unity.

\item $\varphi(A)=
\begin{bmatrix} d_1^{-1} & 0 \\ 0 & d_2^{-1} \end{bmatrix}
\begin{bmatrix} 0 & \varphi_0(d_1) \\ \varphi_0(d_2) & 0 \end{bmatrix}
\begin{bmatrix} d_1 & 0 \\ 0 & d_2 \end{bmatrix}=\delta A$, 
where we have used that $\varphi_0(d_i)=\delta d_i$, $i=1,2$.

\item For any $X\in \cR$, we have
\[
\varphi\bigl(\psi(X)\bigr)=\varphi(AXA^{-1})=\varphi(A)^{-1}\varphi(X)\varphi(A)=A^{-1}\varphi(X)A=\psi^{-1}\bigl(\varphi(X)\bigr),
\]
so $\varphi\psi=\psi^{-1}\varphi$ and hence also $\psi\varphi=\varphi\psi^{-1}$.
\end{itemize}

Now we claim that that $\varphi$ preserves the eigenspaces of $\psi$. 
Indeed, if $\psi(X)=\mu X$ for some $X\ne 0$, then $\bar \mu=\mu^{-1}$ (since $\mu$ is a root of unity), and hence 
\[
\psi(\varphi(X))=\varphi(\psi^{-1}(X))=\varphi(\mu^{-1}X)=\varphi(\bar\mu X)=\mu\varphi(X),
\]
where we have used that $\varphi$ is $\CC$-antilinear.

Therefore, the $\CC$-linear automorphism $\psi$ defines a proper refinement of $\Gamma$, which is a grading on $(\cR,\varphi)$ by the group $G\times\ZZ_n$ where $n$ is the order of $\psi$. Explicitly:
\[
\cR_{(g,\bar k)}=\{X\in\cR_g\mid \psi(X)=e^{2\pi\bi k/n}X\}\;\text{ for all }g\in G\text{ and }k\in\ZZ.
\]
(We note that this refinement is a grading on $\cR$ as an algebra over $\CC$, and any such proper refinement makes $\cR$ a graded-division algebra, because the identity component of $\cR$ becomes $\CC$.)
\end{proof}


Over $\RR$, we obtain a converse of Theorem \ref{th:GMfine} that shows that only fine gradings have to be considered on our graded-division algebras.

\begin{theorem}\label{th:fine_converse}
Let $\cD$ be a finite-dimensional graded-division algebra over $\RR$, 
let $\cV$ be a nonzero graded right $\cD$-module of finite rank, 
and let $\varphi$ be an involution of the graded algebra $\cR = \End_\cD(\cV)$. 
Assume $(\cR,\varphi)$ is central simple as an algebra with involution. 
Then if the grading on $(\cR,\varphi)$ is fine, so is the grading on 
$\cD$.
\end{theorem}

\begin{proof}
Denote by $\Gamma$ the grading on $(\cR,\varphi)$ and by $G$ its grading group. Assume that the grading on $\cD$ is not fine. As usual, denote it by $\Gamma_\cD$ and its support by $T$. The division algebra $\cD_e$ is isomorphic to  either $\CC$ or $\HH$, so all homogeneous components of $\Gamma_\cD$ have dimension $2$ or $4$, respectively. Recall
from Corollary \ref{cor:involutions_from_B} that $\cD$ admits a degree-preserving involution of the same kind as $\varphi$ and, for any such involution $\varphi_0$, we can express $\varphi$ as the adjunction with respect to a nondegenerate homogeneous $\varphi_0$-sesquilinear form $B:\cV\times\cV\to\cD$ satisfying $\wb{B}=\delta B$ with $\delta\in\{\pm 1\}$.  
Our goal is to construct a proper refinement $\Gamma'$ of $\Gamma$ as a grading on $(\cR,\varphi)$. We have the following three possibilities to consider.

\smallskip

\noindent$\bullet$ Case 1: $\cD_e$ is isomorphic to $\CC$, but does not coincide with $Z(\cD)$. 

\smallskip

Then $\cD_e=\RR 1\oplus\RR \bj$, with $\bj^2=-1$. Let $K$ be the support of the centralizer 
$\Cent_\cD(\cD_e)$ of $\cD_e$, which is a subgroup of $T$ of index $2$.
The proof of \cite[Proposition 20]{R16} works in this case, even though 
$\cD$ may fail to be simple, and shows that, for any $t_0\in T\smallsetminus K$ and $0\ne x\in\cD_{t_0}$,
we have $x\bj=-\bj x$ and $x^2\in Z(\cD)$, which implies that the inner automorphism 
$\tau=\Int(x)$ has order $2$ and refines $\Gamma_\cD$ to
a $T\times\ZZ_2$-grading $\Gamma'_\cD$ that has $1$-dimensional components. 
It will be convenient for our purposes to use $G\times\ZZ_4$ as the grading group, so we take 
$\cD_{(t,\bar 0)}\bydef\{ y\in\cD_t\mid \tau(y)=y\}$ and 
$\cD_{(t,\bar 2)}\bydef\{ y\in\cD_t\mid \tau(y)=-y\}$ for all $t\in T$. 

Now $\varphi_0$ restricts to an automorphism of $\cD_e$, which we may assume to be the complex conjugation 
$z\mapsto\bar{z}$. Indeed, if $\varphi_0|_{\cD_e}=\id_{\cD_e}$ then, for any $z\in\cD_e$, we have $\varphi_0(zx)=\varphi_0(x)z=\bar{z}\varphi_0(x)$. Since $\varphi_0(x)=ax$ for some $a\in\cD_e$, we have $x=\varphi_0^2(x)=\bar{a}ax$ 
and hence $\bar{a}a=1$. It follows that there exists $z\in\cD_e$ such that $zx$ is a symmetric element and hence  
we may replace the involution $\varphi_0$ by $\Int(zx)\circ\varphi_0$.

Any $\varphi_0$ that restricts to the complex conjugation on $\cD_e$ is compatible with $\Gamma'_\cD$. 
Indeed, for any $t\in T\smallsetminus K$, $y\in\cD_t$ and $z\in\cD_e$, we have $\varphi_0(zy)=\varphi_0(y)\bar{z}=z\varphi_0(y)$, which implies that $\varphi_0|_{\cD_t}$ is either $\id_{\cD_t}$ or $-\id_{\cD_t}$. 
In particular, the element $x$ is either symmetric or skew-symmetric, which implies that $\varphi_0\tau=\varphi_0\circ\Int(x)=\Int(\varphi_0(x)^{-1})\circ\varphi_0=\tau\varphi_0$. It follows that $\varphi_0$ preserves the components of $\Gamma'_\cD$.

Since the components of $\Gamma'_\cD$ are $1$-dimensional, Lemma \ref{lm:2elementary} now tells us that the support $T'\simeq T\times\ZZ_2$ must be $2$-elementary. In particular, for any $t\in K$ and $y\in\cD_t$, we have $y^2\in\cD_e$. Hence there exists $x_t\in\cD_t$, unique up to sign, such that $x_t^2=1$. It follows that 
$\RR x_t=\{x\in\cD_t\mid x^2\in\RR_{\geq 0}1\}$ and $\RR \bj x_t=\{x\in\cD_t\mid x^2\in\RR_{\leq 0}1\}$, 
hence $\RR x_t$ and $\RR \bj x_t$ must be preserved by $\varphi_0$ and they also must be the components into which $\Gamma'_\cD$ splits $\cD_t$. Note that $\varphi_0(x_t)=\pm x_t$ and $\varphi_0(\bj x_t)=\mp\bj x_t$, so we conclude that, for any $t\in K$, the homogeneous components of $\Gamma_\cD'$ contained in $\cD_t$ coincide with the eigenspaces of
$\varphi_0\vert_{\cD_t}$.

As in the proof of Theorem \ref{th:completeness}, pick a graded $\cD$-basis $\{v_1,\ldots,v_k\}$ of $\cV$ such that the matrix $\bigl(B(v_i,v_j)\bigr)$ is given by equation \eqref{eq:PhiMDpqsdd}. Let $g_0=\deg B$, $g_i=\deg v_i$ for $1\le i\le k$ and $t_i=\deg d_i$ for $1\le i\le q$. Then we have 
\begin{equation}\label{eq:B_homog}
g_1^2 t_1^{-1}=\ldots=g_q^2 t_q^{-1}=g_{q+1}g_{q+2}=\ldots=g_{q+2s-1}g_{q+2s}=g_0^{-1}.
\end{equation}
We are going to define a $G\times\ZZ_4$-grading on $\cV$ that will make it a graded module over $(\cD,\Gamma'_\cD)$ and therefore define a $G\times\ZZ_4$-grading $\Gamma'$ on $\cR$, but we have to make sure that $B$ remains homogeneous so that $\varphi$ will be compatible with $\Gamma'$. To this end, we will modify the 
graded $\cD$-basis of $\cV$ so that the elements $d_i$ will be replaced by $\Gamma'_\cD$-homogeneous elements.

For each $i>q$, let $v'_i=v_i$ and assign to this element degree $g'_i\bydef(g_i,\bar 0)\in G\times\ZZ_4$. Now consider $i\le q$. For any $z\in\cD_e\simeq\CC$, we have $B(v_i z,v_i z)=\varphi_0(z) d_i z=\bar{z} d_i z$. If $t_i\in T\smallsetminus K$ then we get $B(v_i z,v_i z)=d_i z^2$, so we can find $z_i\in\cD_e$ such that $d'_i\bydef d_i z_i^2$ is $\Gamma'_\cD$-homogeneous of degree $t'_i\bydef(t_i,\bar 0)$. So we let $v'_i=v_i z_i$ and assign to this element degree $g'_i\bydef(g_i,\bar 0)\in G\times\ZZ_4$. Finally, if $t_i\in K$ then $d_i$ is already $\Gamma'_\cD$-homogeneous, because $\varphi_0(d_i)=\delta d_i$ and we have seen that, in $\cD_{t_i}$, the homogeneous components of $\Gamma_\cD'$ coincide with the eigenspaces of $\varphi_0$. So we let $v'_i=v_i$ and assign to this element degree $g'_i$ as follows: $g'_i=(g_i,\bar 0)$ if $t'_i=(t_i,\bar 0)$ and $g'_i=(g_i,\bar 1)$ if $t'_i=(t_i,\bar 2)$ where $t'_i$ is the $\Gamma'_\cD$-degree of $d_i$. This assignment of degrees to the $\cD$-basis $\{v'_1,\ldots,v'_k\}$ defines a $G\times\ZZ_4$-grading $\Gamma'_\cV$ on $\cV$ as a module over $(\cD,\Gamma'_\cD)$. By construction, equation \eqref{eq:B_homog} remains valid if we replace $g_0$ by $(g_0,\bar 0)$, $g_i$ by $g'_i$ for $1\le i\le k$ and $t_i$ by $t'_i$ for $1\le i\le q$, so $B$ is homogeneous of degree $(g_0,\bar 0)$ with respect to the $G\times\ZZ_4$-gradings. Moreover, for the projection $\pi:G\times\ZZ_4\to G$, the coarsenings ${}^\pi\Gamma'_\cD$ and ${}^\pi\Gamma'_\cV$ are the original $G$-gradings on $\cD$ and $\cV$, respectively. We conclude that $\Gamma'$ is a grading on $(\cR,\varphi)$ and it is a proper refinement of the original grading $\Gamma$.

\smallskip

\noindent$\bullet$ Case 2: $\cD_e$ is isomorphic to $\CC$ and coincides with $Z(\cD)$. 

\smallskip

The grading $\Gamma_\cD$  not being fine implies that $T$ is $2$-elementary, so $\cD$ has a unique distinguished involution, which we will choose as $\varphi_0$. It is characterized by the property that  
$x\varphi_0(x)$ is a positive scalar multiple of $1$, for any nonzero homogeneous element $x\in \cD$. 
Each $\cD_t$ splits into the direct sum of the eigenspaces of $\varphi_0$, which are $\{x\in\cD_t\mid x^2\in\RR_{\geq 0}1\}$ and $\{x\in\cD_t\mid x^2\in\RR_{\leq 0}1\}$, and this gives a $T\times\ZZ_2$-grading $\Gamma'_\cD$ that has $1$-dimensional components. We see that this case is analogous to Case 1, except that now the whole $T$ plays the role of $K$. Therefore, we can proceed in the same manner (with the simplification of not having to change the graded $\cD$-basis of $\cV$, since the elements $d_i$ are automatically $\Gamma'_\cD$-homogeneous) to construct a $G\times\ZZ_4$-grading $\Gamma'$ on $(\cR,\varphi)$ that is a proper refinement of $\Gamma$. 

\smallskip

\noindent$\bullet$ Case 3: $\cD_e$ is isomorphic to $\HH$.

\smallskip 

Since $\cD_e$ is central simple, by the Double Centralizer Theorem (see e.g. \cite[Theorem 4.7]{Jacobson}), we can write $\cD=\cD_e\otimes\cC$ where $\cC=\Cent_\cD(\cD_e)$, which is a graded subalgebra of $\cD$. Since $\cD_t=\cD_e\otimes\cC_t$, $\cC$ has support $T$ and $1$-dimensional components. For each $t\in T$, we pick a nonzero element $x_t\in\cC_t$, so $\cD_t=\cD_e x_t$.

The degree-preserving involutions on $\cD$ are tensor products of such involutions on $\cD_e$ and $\cC$, so we may fix the involution $\varphi_0$ on $\cD$ so that $\varphi_0|_{\cD_e}$ is the standard involution of 
$\cD_e$, which we denote by bar. 

Note that the decomposition $\HH=\bigl(\RR 1+\RR\bi\bigr)\oplus(\RR \bj+\RR \bk)$ is a $\ZZ_2$-grading on $\HH$, so we can use the corresponding grading $\cD_e=(\cD_e)_{\bar 0}\oplus(\cD_e)_{\bar 1}$ to define a $G\times\ZZ_2$-grading $\Gamma_\cD'$ on $\cD$ as follows: $\cD_{(t,\bar 0)}=(\cD_e)_{\bar 0}\otimes\cC_t$ and
$\cD_{(t,\bar 1)}=(\cD_e)_{\bar 1}\otimes \cC_t$ for all $t\in T$. 
This is a refinement of $\Gamma_\cD$ with support $T'=T\times \ZZ_2$ and $2$-dimensional components. 
Since $\varphi_0|_{\cD_e}$ is the standard involution, $\varphi_0$ is compatible with $\Gamma'_\cD$.

We pick a graded $\cD$-basis $\{v_1,\ldots,v_k\}$ of $\cV$ as in the proof in Case 1 and then modify $v_i$ to make sure that the elements $d_i$ are replaced by $\Gamma'_\cD$-homogeneous elements, which will allow us to define a $G\times\ZZ_2$-grading $\Gamma'$ on $(\cR,\varphi)$. 

Consider first $i\le q$. We can write $d_i=h_i x_{t_i}$ for some $h_i\in\cD_e$. For any $y\in\cD_e\simeq\HH$, we have $B(v_i y,v_i y)=\varphi_0(y) d_i y=(\bar{y} h_i y)x_{t_i}$. Since any element of $\HH$ is conjugate to an element in the subalgebra $\RR 1+\RR\bi\simeq\CC$, we can find $y_i\in\cD_e$ such that $d'_i\bydef (\bar{y}_i h_i y_i)x_{t_i}$ belongs to $(\cD_e)_{\bar 0} x_{t_i}$. In other words, $d'_i$ is $\Gamma'_\cD$-homogeneous of degree $t'_i\bydef(t_i,\bar 0)$, so we let $v'_i=v_i y_i$. Now, for $i>q$, we let $v'_i=v_i$, as before. Finally, assigning to $v'_i$ degree $g'_i\bydef(g_i,\bar 0)\in G\times\ZZ_2$, for all $1\le i\le k$, we get a grading $\Gamma'_\cV$ on $\cV$ that yields the desired grading $\Gamma'$ on $(\cR,\varphi)$. Since the coarsenings ${}^\pi\Gamma'_\cD$ and ${}^\pi\Gamma'_\cV$ defined by the projection $\pi:G\times\ZZ_2\to G$ are the original $G$-gradings on $\cD$ and $\cV$, respectively, we conclude that $\Gamma'$ is a proper refinement of $\Gamma$.
\end{proof}


We now turn our attention to the gradings $\Gamma_{\cM^\ex}(\cD,k)$ in Definition \ref{df:Mex}.

\begin{theorem}\label{th:FineMex}
Let $\cD$ be a finite-dimensional graded-division algebra with abelian support $T$ over a field $\FF$ of characteristic different from $2$. Assume that $\cD$ is a central simple algebra and has $1$-dimensional homogeneous components. 
Let $k$ be a positive integer. Then the grading $\Gamma_{\cM^\ex}(\cD,k)$ is fine if and only if $k\ge 3$ or $T$ is not an elementary $2$-group. 
\end{theorem}

\begin{proof}
Since $\cD$ has $1$-dimensional components, the grading $\Gamma_\cM(\cD,k)$ on the algebra $\cS\bydef\cM(\cD,k)$ is fine (in the class of abelian group gradings) by \cite[Proposition 2.35]{EKmon}. (In fact, it is shown in \cite{R20} for any graded-division algebra $\cD$ with an abelian support that $\Gamma_\cM(\cD,k)$ is fine in the class of abelian group gradings if and only if $\Gamma_\cD$ is fine in the same class.) Therefore, the grading $\Gamma\bydef\Gamma_{\cM^\ex}(\cD,k)$ on $\cR\bydef\cS\times\cS^\op$ cannot be refined in the class of Type I gradings.

Assume that a $G'$-grading $\Gamma'$ is a Type II refinement of $\Gamma$ as a grading on the algebra with involution $(\cR,\ex)$. Since $\cD$ has $1$-dimensional components, it remains a graded-division algebra after an extension of scalars. Extending to an algebraic closure $\Falg$, $\Gamma$ gives the grading $\Gamma_{\cM^\ex}(\cD\otimes\Falg,k)$ on $(\cR\otimes\Falg,\ex)$ and $\Gamma'$ gives a Type II refinement of this grading. It follows that the graded algebra $\cS\otimes\Falg$ admits an antiautomorphism, namely, the restriction of the composition of $\ex$ and the action of any character $\chi:G'\rightarrow \Falg^\times$ that takes value $-1$ on the distinguished element $f\in G'$. (Such characters always exist as $\Falg^\times$ is a divisible group.) Now \cite[Proposition 3.27]{EKmon} shows that $T$ is $2$-elementary and $k\leq 2$.

Conversely, suppose $T$ is $2$-elementary. Since $\cD_e=\FF$ and $\cD$ is central simple, $\cD$ is a tensor product of graded subalgebras as in equation \eqref{eq:symbol_alg}, each of which is a symbol algebra of degree $2$, i.e., a quaternion algebra. It follows that $\cD$ admits a degree-preserving involution $\varphi_0$ (e.g., 
the tensor product of the standard involutions). For $k\le 2$, the $\wt{G}(T,k)$-graded algebra $\cS=M_k(\cD)$ admits an involution $\varphi$, namely, $\varphi=\varphi_0$ if $k=1$ and 
\[
\varphi:\begin{bmatrix} a & b \\ c & d \end{bmatrix}\mapsto
\begin{bmatrix} 0 & 1\\ 1 & 0 \end{bmatrix} 
\begin{bmatrix} \varphi_0(a) & \varphi_0(c) \\ \varphi_0(b) & \varphi_0(d) \end{bmatrix}
\begin{bmatrix} 0 & 1 \\ 1 & 0 \end{bmatrix}
=\begin{bmatrix} \varphi_0(d) & \varphi_0(b) \\ \varphi_0(c) & \varphi_0(a) \end{bmatrix}
\]
if $k=2$. The mapping $(\varphi,\varphi)\circ\ex=\ex\circ(\varphi,\varphi)$ is
an involutive automorphism of $(\cR,\ex)$, and gives a Type II refinement with grading group $G'=\wt{G}(T,k)\times\ZZ_2$.
\end{proof}

\subsection{Fine gradings in the real case}

Corollary \ref{cor:completeness} implies that if $(\cR,\varphi)$ is finite-dimensional and simple as an algebra with involution then any grading on $(\cR,\varphi)$ is isomorphic to a coarsening of a grading of the form  
$\Gamma_\cM(\cD,\varphi_0,q,s,\ul{d},\delta)$ or $\Gamma_{\cM^\ex}(\cD,k)$. It follows that, up to equivalence, all fine gradings on $(\cR,\varphi)$ are found among these gradings. Assuming $(\cR,\varphi)$ is central simple over $\RR$ (as an algebra with involution), the results of Subsection \ref{ss:fineness} give necessary and sufficient conditions for these gradings to be fine.

We now summarize all possibilities:

\begin{theorem}\label{th:fine_real}
Let $\cR$ be a finite-dimensional algebra over $\RR$ and $\varphi$ be an involution on $\cR$ such that $(\cR,\varphi)$ is central simple as an algebra with involution. If $(\cR,\varphi)$ is equipped with a group grading $\Gamma$, then $\Gamma$ is fine if and only if $\cR$ is equivalent as a graded algebra with involution to one of the following:  
\begin{enumerate}
\item[(1)] Central simple algebras over $\RR$ with a first kind involution, which is either orthogonal ($\delta=+1$) or symplectic ($\delta=-1$):
\begin{itemize}
\item $\cM(2m;\RR;q,s,\ul{d},\delta)\bydef\cM\bigl(\cD(2m;+1),*,q,s,\ul{d},\delta\bigr)$ where $m\ge 0$, \\ 
$X^*=X^T$ for all $X\in \cD(2m;+1)\simeq M_{2^m}(\RR)$ and, in the case \\
$(q,s)=(2,0)$, the pair $\ul{d}=(d_1,d_2)$ satisfies $\deg d_1\ne\deg d_2$;

\item $\cM(2m;\HH;q,s,\ul{d},\delta)\bydef\cM\bigl(\cD(2m;-1),*,q,s,\ul{d},-\delta\bigr)$ where $m\ge 1$, \\
$X^*=\wb{X}^T$ for all $X\in \cD(2m;-1)\simeq M_{2^{m-1}}(\HH)$ and, in the case \\
$(q,s)=(2,0)$, the pair $\ul{d}=(d_1,d_2)$ satisfies $\deg d_1\ne\deg d_2$; 
\end{itemize}

\item[(2)] Central simple algebras over $\CC$ with a second kind involution:
\begin{itemize}
\item $\cM^{\mathrm{(I)}}(\ell_1,\ldots,\ell_m;\CC;q,s,\ul{d})\bydef\cM\bigl(\cD(\ell_1,\ldots,\ell_m;\CC),\varphi_0,q,s,\ul{d},+1\bigr)$ with ${m\ge 1}$, prime powers $\ell_j\ge 2$ that are not all equal to $2$, \\
$\varphi_0=\varphi_{\ell_1}\otimes_\CC\cdots\otimes_\CC\varphi_{\ell_m}$ 
on $\cD(\ell_1,\ldots,\ell_m;\CC)\simeq M_{\ell_1\cdots \ell_m}(\CC)$ defined by equation \eqref{eq:inv_bad_D}, 
and $(q,s)\ne(2,0)$;

\item $\cM^{\mathrm{(II)}}(2m+1;\CC;q,s,\ul{d})\bydef\cM\bigl(\cD(2m+1;\RR),*,q,s,\ul{d},+1\bigr)$ where $m\ge 0$, \\
$X^*=\wb{X}^T$ for all $X\in \cD(2m+1;\RR)\simeq M_{2^m}(\CC)$ and, in the case \\
$(q,s)=(2,0)$, the pair $\ul{d}=(d_1,d_2)$ satisfies $\deg d_1\ne\deg d_2$;
\end{itemize}

\item[(3)] Nonsimple algebras with an exchange involution:
\begin{itemize}
\item $\cM^{\mathrm{(I)}}(2m;\RR;k)\bydef\cM^{\ex}\bigl(\cD(2m;+1),k\bigr)$ where $m\ge 0$ and $k\ge 3$;

\item $\cM^{\mathrm{(I)}}(2m;\HH;k)\bydef\cM^{\ex}\bigl(\cD(2m;-1),k\bigr)$ where $m\ge 1$ and $k\ge 3$;

\item $\cM^{\mathrm{(II)}}(2m+1;\RR;q,s,\ul{d})\bydef\cM\bigl(\cD(2m+1;+1),\varphi_0,q,s,\ul{d},+1\bigr)$ \\
where $m\ge 0$, $\varphi_0(X,Y)=(Y^T,X^T)$ for all $(X,Y)\in\cD(2m+1;+1)
\simeq M_{2^m}(\RR)\times M_{2^m}(\RR)$ and, if $(q,s)=(2,0)$, then 
$\deg d_1\ne\deg d_2$;

\item $\cM^{\mathrm{(II)}}(2m+1;\HH;q,s,\ul{d})\bydef\cM\bigl(\cD(2m+1;-1),\varphi_0,q,s,\ul{d},+1\bigr)$ \\
where $m\ge 1$, $\varphi_0(X,Y)=(\wb{Y}^T,\wb{X}^T)$ for all $(X,Y)\in\cD(2m+1;-1)
\simeq {M_{2^{m-1}}(\HH)\times M_{2^{m-1}}(\HH)}$ and, if $(q,s)=(2,0)$, then 
$\deg d_1\ne\deg d_2$.\qed
\end{itemize}
\end{enumerate}
\end{theorem}

\begin{remark}
Theorem \ref{th:GMfine}, Proposition \ref{pr:GMfineq2s0} and Theorem \ref{th:FineMex} show that the complexifications of all graded algebras with involution in Theorem \ref{th:fine_real} except $\cM^{\mathrm{(I)}}(\ell_1,\ldots,\ell_m;\CC;q,s,\ul{d})$ have fine gradings.
\end{remark}

\section{The equivalence problem}\label{se:equivalence}

\subsection{Equivalence of graded-simple algebras with involution}

The isomorphisms between $G$-graded algebras of the form $\cR=\End_\cD(\cV)$ and $\cR'=\End_{\cD'}(\cV')$ are described in \cite[Theorem 2.10]{EKmon}. By an isomorphism $(\cD,\cV)\to(\cD',\cV')$ we mean a pair $(\psi_0,\psi_1)$ where
\begin{enumerate}
\item[(i)] $\psi_0:\cD\to\cD'$ is an isomorphism of $G$-graded algebras,
\item[(ii)] $\psi_1:\cV\to\cV'$ is an isomorphism of $G$-graded vector spaces, and
\item[(iii)] $\psi_1$ is $\psi_0$-semilinear: $\psi_1(vd)=\psi_1(v)\psi_0(d)$ for all $v\in\cV$ and $d\in\cD$.
\end{enumerate}
If $\psi:\cR\to\cR'$ is an isomorphism, then there exist an element $g\in G$ and an isomorphism $(\psi_0,\psi_1)$ from $({}^{[g^{-1}]}\cD^{[g]},\cV^{[g]})$ to $(\cD',\cV')$ such that $\psi(r)=\psi_1 r\psi_1^{-1}$ for all $r\in\cR$. Conversely, any isomorphism $(\psi_0,\psi_1)$ from $({}^{[g^{-1}]}\cD^{[g]},\cV^{[g]})$ to $(\cD',\cV')$ gives rise to an isomorphism $\psi:\cR\to\cR'$ in this way.

When we study equivalence of graded algebras, it is convenient to use universal groups, because an equivalence then becomes a weak isomorphism. So, we start with the question what happens to a $G$-graded algebra of the form $\End_\cD(\cV)$ if we consider it as graded by its universal group $U$.

\begin{proposition}\label{pr:EndDV_over_U}
Let $G$ be a group and let $\cR=\End_\cD(\cV)$ where $\cD$ is a graded-division algebra with support $T\subset G$ and $\cV$ is a $G$-graded right $\cD$-module of finite rank. Assume that $G$ is generated by the support of $\cR$. Denote by $\Gamma_\cD$ and $\Gamma_\cV$ the $G$-gradings on $\cD$ and $\cV$, respectively, and by $\Gamma$ the $G$-grading on $\cR$ (induced by $\Gamma_\cV$). Let $U$ be the universal group of $\Gamma$, let $\wt{\Gamma}:\cR=\bigoplus_{u\in U}\wt{\cR}_u$ be the realization of $\Gamma$ as a $U$-grading, and let $\alpha:U\rightarrow G$ be the associated group homomorphism such that $\Gamma={}^\alpha\wt{\Gamma}$. Then there are $U$-gradings $\wt{\Gamma}_\cD$ on $\cD$ and $\wt{\Gamma}_\cV$ on $\cV$ such that 
\begin{enumerate}
\item[(i)] $(\cD,\wt{\Gamma}_\cD)$ is a graded algebra, ${}^\alpha\wt{\Gamma}_\cD=\Gamma_\cD$, and $\alpha$ maps $\supp\wt{\Gamma}_\cD$ bijectively onto $T$ (hence $\Gamma_\cD$ and $\wt{\Gamma}_\cD$ differ only by the labeling of the components);
\item[(ii)] $(\cV,\wt{\Gamma}_\cV)$ is a graded right $(\cD,\wt{\Gamma}_\cD)$-module, ${}^\alpha\wt{\Gamma}_\cV=\Gamma_\cV$, and $\alpha$ maps $\supp\wt{\Gamma}_\cV$ bijectively onto $\supp\Gamma_\cV$ (hence $\Gamma_\cV$ and 
$\wt{\Gamma}_\cV$ differ only by the labeling of the components);
\item[(iii)] $\wt{\Gamma}$ is the grading induced by $\wt{\Gamma}_\cV$.
\end{enumerate}
\end{proposition}

\begin{proof}
Let $\cI$ be a minimal graded left ideal of $(\cR,\Gamma)$. Then $\cI=\cR\veps$ for an idempotent $\veps\in\cR_e$ and, as a $G$-graded left $\cR$-module, $\cI$ is isomorphic to a shift of $\cV$ (see e.g. \cite[Lemma 2.7]{EKmon}), 
so there is an element $g\in G$ and an isomorphism $\psi_1:\cV^{[g]}\to\cI$. 
Hence we get a chain of isomorphisms of $G$-graded algebras (see e.g. \cite[Lemma 2.8]{EKmon}):
\[
{}^{[g^{-1}]}\cD^{[g]}\simeq\End_\cR(\cV^{[g]})\simeq\End_\cR(\cI)\simeq\veps\cR\veps\defby\cD',
\]
whose composition will be denoted by $\psi_0:{}^{[g^{-1}]}\cD^{[g]}\to\cD'$.

Passing to the universal group $U$, the idempotent $\veps$ remains homogeneous of degree $e$, as the degree of any nonzero homogeneous idempotent is necessarily the neutral element, and hence $\wt{\Gamma}$ restricts to $U$-gradings $\wt{\Gamma}_{\cD'}:\cD'=\bigoplus_{u\in U}\wt{\cD}'_u$ and $\wt{\Gamma}_\cI:\cI=\bigoplus_{u\in U}\tilde{\cI}_u$ where $\wt{\cD}'_u=\veps\wt{\cR}_u\veps$ and $\tilde{\cI}_u=\wt{\cR}_u\veps$ for all $u\in U$. Thus $\cD'$ becomes a $U$-graded algebra and $\cI$ becomes a $U$-graded right $\cD'$-module. By construction, the original $G$-gradings $\Gamma_{\cD'}$ and $\Gamma_\cI$ are obtained from these $U$-gradings $\wt{\Gamma}_{\cD'}$ and $\wt{\Gamma}_\cI$ through the homomorphism $\alpha$. Since $\alpha$ maps $\supp\wt{\Gamma}$ bijectively onto $\supp\Gamma$ (more precisely, $\wt{\cR}_u=\cR_{\alpha(u)}$ for all $u\in\supp\wt{\Gamma}$), it maps $\supp\wt{\Gamma}_{\cD'}$ and $\supp\wt{\Gamma}_\cI$  bijectively onto $\supp\Gamma_{\cD'}$ and $\supp\Gamma_\cI$, respectively. Moreover, if we identify $\cR$ with $\End_{\cD'}(\cI)$, $\wt{\Gamma}$ is induced by the $U$-grading on $\cI$, i.e., $\wt{\cR}_u=\{r\in\cR\mid r\tilde{\cI}_{u'}\subset\tilde{\cI}_{uu'}\;\forall u'\in U\}$.

Since $\supp\Gamma$ generates $G$, the homomorphism $\alpha:U\to G$ is surjective, so we can find $h\in U$ such that $\alpha(h)=g$. Define the $U$-gradings $\wt{\Gamma}_\cD:\cD=\bigoplus_{u\in U}\wt{\cD}_u$ and $\wt{\Gamma}_\cV:\cV=\bigoplus_{u\in U}\wt{\cV}_u$ by $\wt{\cD}_u=\psi_0^{-1}(\wt{\cD}'_{h^{-1}uh})$ and $\wt{\cV}_u=\psi_1^{-1}(\tilde{\cI}_{uh})$ for all $u\in U$. 
By construction, $\wt{\Gamma}_\cV$ corresponds under $\psi_1$ to $\wt{\Gamma}_{\cI}^{[h^{-1}]}$, so ${}^\alpha\wt{\Gamma}_\cV$ corresponds to $\Gamma_{\cI}^{[g^{-1}]}$, hence we have ${}^\alpha\wt{\Gamma}_\cV=\Gamma_\cV$. 
Since right translations in any group are bijective, $\alpha$ maps $\supp\wt{\Gamma}_\cV$ bijectively onto $\supp\Gamma_\cV$, because this holds for the supports of the corresponding gradings on $\cI$. 
Similarly for $\cD$. Since $\psi_1:(\cV,\wt{\Gamma}_\cV^{[h]})\to(\cI,\wt{\Gamma}_\cI)$ is an isomorphism of $U$-graded left $\cR$-modules, $\wt{\Gamma}$ is induced by $\wt{\Gamma}_\cV$.
\end{proof}

\begin{proposition}\label{pr:B_U_homogeneous}
Under the conditions of Proposition \ref{pr:EndDV_over_U}, suppose $\cR$ admits an antiautomorphism $\varphi$ that leaves the components of $\Gamma$ invariant (hence $G$ and $U$ are abelian). Let $(\varphi_0,B)$ be a pair associated to $\varphi$ on the $G$-graded algebra $(\cR,\Gamma)$ as in Theorem \ref{th:involutions_from_B}. Then 
$\varphi_0$ is an antiautomorphism of the $U$-graded algebra $(\cD,\wt{\Gamma}_\cD)$ and $B:\cV\times\cV\to\cD$ is homogeneous with respect to the gradings $\wt{\Gamma}_\cV$ and $\wt{\Gamma}_\cD$.
\end{proposition}

\begin{proof}
Since $\wt{\Gamma}_\cD$ has, up to relabeling, the same components as $\Gamma_\cD$, the claim about $\varphi_0$ is clear.
We know that $B$ is homogeneous of degree, say, $g_0\in G$ with respect to the gradings $\Gamma_\cV$ and $\Gamma_\cD$. 
Since $\varphi$ is an antiautomorphism of the $U$-graded algebra $(\cR,\wt{\Gamma})$, there exists a pair $(\tilde{\varphi}_0,\wt{B})$ as in Theorem \ref{th:involutions_from_B}, where $\wt{B}:\cV\times\cV\to\cD$ is homogeneous of degree, say, $u_0\in U$ with respect to $\wt{\Gamma}_\cV$ and $\wt{\Gamma}_\cD$. Applying $\alpha$, it follows that $\wt{B}$ is homogeneous of degree $\alpha(u_0)$ with respect to $\Gamma_\cV$ and $\Gamma_\cD$. Hence $\wt{B}=dB$ for some nonzero element $d\in\cD$ of degree $\alpha(u_0)g_0^{-1}$ with respect to $\Gamma_\cD$. It follows that $d$ is homogeneous with respect to $\wt{\Gamma}_\cD$ and hence $B=d^{-1}\wt{B}$ is homogeneous with respect to $\wt{\Gamma}_\cV$ and $\wt{\Gamma}_\cD$.
\end{proof}

As an application, we will now find the universal groups of the graded algebras with involution $\cM(\cD,\varphi_0,q,s,\ul{d},\delta)$ defined in Subsection \ref{sse:def_MDpqsdd}. Recall that these are algebras of the form  $M_{q+2s}(\cD)\cong M_{q+2s}(\FF)\otimes\cD$, equipped with a grading by the abelian group $\wt{G}(T,q,s,\ul{t})$, defined by equation \eqref{eq:gradingMDpqsdd}, and a certain involution $\varphi$, where $\cD$ is a graded-division algebra with support $T$, $\varphi_0$ is an involution on $\cD$, and $\ul{d}$ is a $q$-tuple of elements $0\ne d_i\in\cD_{t_i}$. 

\begin{proposition}\label{pr:G0_is_universal}
Let $\cR=\cM(\cD,\varphi_0,q,s,\ul{d},\delta)$ and let $\wt{G}^0(T,q,s,\ul{t})$ be the subgroup of $\wt{G}(T,q,s,\ul{t})$ generated by the support of the grading $\Gamma$ on $\cR$. Then:
\begin{enumerate}
\item[(i)] $\wt{G}^0(T,q,s,\ul{t})$ is the universal group of $\Gamma$. 
\item[(ii)] This group can alternatively be presented, for any fixed $i_0$, as the abelian group generated by $T$ and the symbols $u_1,\ldots,u_{q+2s}$ with defining relations 
\begin{equation}\label{eq:relGtilde0}
u_1^2 t_1^{-1}=\ldots=u_q^2 t_q^{-1}=u_{q+1}u_{q+2}=\ldots=u_{q+2s-1}u_{q+2s}\;
\text{ and }\;u_{i_0}=e,
\end{equation}
and the corresponding grading $\wt{\Gamma}$ on $\cR$ is then defined by $\deg(E_{ij}\ot d)=u_i u_j^{-1}t$ for any $0\ne t\in\cD_t$, $t\in T$.
\end{enumerate}
\end{proposition}

\begin{proof}
Denote $G=\wt{G}(T,q,s,\ul{d})$, $G_0=\wt{G}^0(T,q,s,\ul{d})$, and $k=q+2s$. In more abstract terms, $\cR=\End_\cD(\cV)$ where $\cV=\cD^k$ is given a $G$-grading by assigning the standard $\cD$-basis elements $v_1,\ldots,v_k$ degrees $\tilde{g}_1,\ldots,\tilde{g}_k$, respectively. The involution $\varphi$ on $\cR$ is the adjunction with respect to the $\varphi_0$-sesquilinear form $B$ represented by matrix \eqref{eq:PhiMDpqsdd}. Now fix some $i_0$, say $i_0=1$. Replacing $\cV$ by $\cV^{[\tilde{g}_1^{-1}]}$ makes $\deg v_i=g_i\bydef\tilde{g}_i\tilde{g}_1^{-1}$, so we may suppose that $\cV$, $\cD$ and $\cR$ are graded by the group $G_0$, which is generated by $T$ and the elements $g_2,\ldots,g_k$ (note that $g_1=e$). 

Now we can use Proposition \ref{pr:EndDV_over_U} to describe the realization of the $G$-grading $\Gamma$ as a grading $\wt{\Gamma}$ by the universal group $U$. Since the homomorphism $\alpha:U\to G_0$ restricts to an isomorphism $\supp\wt{\Gamma}_\cD\to T$, we can identify $T$ with a subgroup of $U$ and assume $\alpha|_T=\id_T$. Note that $U$ is generated by $T$ and the elements $u_1,\ldots,u_k$, where $u_i\in U$ is the degree of $v_i$ with respect to the grading $\wt{\Gamma}_\cV$, and also that $\alpha$ maps $u_i\mapsto g_i$. Replacing $\wt{\Gamma}_\cV$ by $\wt{\Gamma}_\cV^{[u_1^{-1}]}$ and keeping the same $\wt{\Gamma}_\cD$ (since $U$ is abelian), we may suppose that $u_1=e$. By Proposition \ref{pr:B_U_homogeneous}, $B$ is homogeneous with respect to the $U$-gradings, so the fact that $B(v_i,v_i)=d_i\ne 0$ for $1\le i\le q$ and $B(v_{q+2j-1},v_{q+2j})=1$ for $1\le j\le s$ implies that $u_1,\ldots,u_k$ satisfy relations \eqref{eq:relGtilde0}.
Comparing these with the defining relations \eqref{eq:rel_Gtilde} of $G$, we see that there exists a homomorphism $\gamma:G\to U$ defined by $\gamma|_T=\id_T$ and $\gamma(\tilde{g}_i)=u_i$ for $1\le i\le k$. Since $\gamma(g_i)=\gamma(\tilde{g}_i)\gamma(\tilde{g}_1)^{-1}=u_i u_1^{-1}=u_i$ (recall that $u_1=e$), we see that $\gamma\alpha=\id_U$, which implies that $\alpha:U\to G_0$ is an isomorphism. Moreover, $\pi\bydef\alpha\gamma$ is a retraction of $G$ onto $ G_0$, which allows us to obtain a presentation for $G_0$ in terms of the generators $g_i=\pi(\tilde{g}_i)$. Since $g_1=e$, $\pi$ factors through a homomorphism $\bar{\pi}:G/K\to G_0$ that sends $\tilde{g}_i K\mapsto g_i$ where $K\bydef\langle\tilde{g}_1\rangle$. On the other hand, the composition of the inclusion $G_0\to G$ and the quotient map $G\to G/K$ is a homomorphism $G_0\to G/K$ that sends $g_i\mapsto\tilde{g}_i\tilde{g}_1^{-1}K=\tilde{g}_i K$ and, therefore, is the inverse of $\bar{\pi}$. 
\end{proof}

The following result extends \cite[Proposition 2.33]{EKmon} to take into account antiautomorphisms.

\begin{theorem}\label{th:equivalence_with_anti}
Let $G$ and $G'$ be abelian groups and consider the $G$-graded algebra $\cR=\End_\cD(\cV)$ and $G'$-graded algebra $\cR'=\End_{\cD'}(\cV')$ where $\cD$ and $\cD'$ are graded-division algebras (graded by $G$ and $G'$, respectively) and $\cV$ and $\cV'$ are nonzero graded right modules of finite rank over $\cD$ and $\cD'$, respectively. 
\begin{enumerate}
\item[(1)]
If $\psi:\cR\to\cR'$ is an equivalence of graded algebras, then there exists an equivalence $(\psi_0,\psi_1)$ from $(\cD,\cV)$ to $(\cD',\cV')$, by which we mean that
\begin{enumerate}
\item[(i)] $\psi_0:\cD\to\cD'$ is an equivalence of graded algebras,
\item[(ii)] $\psi_1:\cV\to\cV'$ is an equivalence of graded vector spaces, and
\item[(iii)] $\psi_1$ is $\psi_0$-semilinear: $\psi_1(vd)=\psi_1(v)\psi_0(d)$ for all $v\in\cV$ and $d\in\cD$,
\end{enumerate}
such that $\psi(r)=\psi_1 r\psi_1^{-1}$ for all $r\in\cR$. 
\item[(2)] 
Another pair $(\psi'_0,\psi'_1)$ satisfies these conditions if and only if there exists
$d'\in\cD'^\times_\gr$ such that $\psi'_0=\Int(d')^{-1}\circ\psi_0$ and $\psi'_1(v)=\psi_1(v)d'$ for all $v\in\cV$.
\item[(3)]
Suppose $\varphi$ and $\varphi'$ are antiautomorphism of $\cR$ and $\cR'$, respectively, and let $(\varphi_0,B)$ and $(\varphi'_0,B')$ be corresponding pairs as in Theorem \ref{th:involutions_from_B}. Then $\psi$ satisfies $\psi\varphi\psi^{-1}=\varphi'$ if and only if there exists $d_0\in\cD^\times_\gr$ such that 
\begin{equation}\label{eq:B'B}
B'\bigl(\psi_1(v),\psi_1(w)\bigr)=\psi_0\bigl(d_0B(v,w)\bigr)\;\text{ for all }v,w\in\cV\text{ and }d\in\cD.
\end{equation}
If this is the case, we also have
\begin{equation}\label{eq:d0psi}
\Int(d_0)\circ\varphi_0=\psi_0^{-1}\varphi_0'\psi_0.
\end{equation} 
\end{enumerate}
\end{theorem}

\begin{proof}
Parts (1) and (2) are proved in \cite{EKmon}, but we will obtain (1) here as an application of Proposition \ref{pr:EndDV_over_U}, because this approach will also allow us to prove (3).
 
Since $\End_\cD(\cV)=\End_\cD(\cV^{[g]})$ for any $g\in G$, nothing in part (1) is affected by a shift of grading. Let $G_0$ be the subgroup of $G$ generated by $\supp\cR$. Applying a suitable shift, we may assume without loss of generality that $\supp\cV\subset G_0$. (One way to see it is to represent $\End_\cD(\cV)$ by matrices, with the $G$-grading given by \eqref{eq:degEijd}, and then replace $\cV$ by $\cV^{[g_1^{-1}]}$.) Considering $\cR$, $\cV$ and $\cD$ as $G_0$-graded, we can apply Proposition \ref{pr:EndDV_over_U} and relabel these gradings with the elements of the universal group $U$. Since nothing in part (1) is affected by such a relabeling, we may assume without loss of generality that $G$ is the universal group of the grading on $\cR$, and similarly that $G'$ is the universal group for $\cR'$. 

Now let $\gamma:\supp\cR\to\supp\cR'$ be the bijection defined by the equivalence $\psi$, i.e., $\psi(\cR_g)=\cR_{\gamma(g)}$ for all $g\in\supp\cR$. Then $\gamma$ extends to a group isomorphism $G\to G'$, which we also denote by $\gamma$. By construction, $\psi$ then becomes an isomorphism of $G$-graded algebras $\cR\to{}^{\gamma^{-1}}\cR'$. Denote $\cR''={}^{\gamma^{-1}}\cR'$, $\cD''={}^{\gamma^{-1}}\cD'$ and $\cV''={}^{\gamma^{-1}}\cV'$. Then $\cR''=\End_{\cD''}(\cV'')$ and, applying \cite[Theorem 2.10]{EKmon} to the isomorphism $\psi:\cR\to\cR''$, we can find an element 
$g\in G$ and an isomorphism $(\psi_0,\psi_1)$ from $(\cD,\cV^{[g]})$ to $(\cD'',\cV'')$ such that $\psi(r)=\psi_1 r\psi_1^{-1}$ for all $r\in\cR$. Since $\psi_0$ is an equivalence of graded algebras $\cR\to\cR'$ and $\psi_1$ is an equivalence of graded vector spaces $\cV\to\cV'$, we are done with part (1). Note that any $(\psi'_0,\psi'_1)$ as in part (2) is an isomorphism $(\cD,\cV^{[g']})\to(\cD'',\cV'')$ where $g'=g\gamma^{-1}(t')$ for $d'\in\cD_{t'}$.

For part (3), we note that, after the manipulations above to replace $G$ and $G'$ by the universal groups, the antiautomorphisms $\varphi_0$ and $\varphi'_0$ remain degree-preserving and the sesquilinear forms $B$ and $B'$ remain homogeneous, due to Proposition \ref{pr:B_U_homogeneous}. The set $\cD^\times_\gr$ is also unaffected. The proof is then completed by applying \cite[Lemma 3.32]{EKmon} to $\psi$ regarded as an isomorphism $\cR\to\cR''$ and $(\psi_0,\psi_1)$ regarded as an isomorphism 
$(\cD,\cV^{[g]})\to(\cD'',\cV'')$. 
\end{proof}

Since a necessary condition for the equivalence of $\cR$ and $\cR'$ is that of $\cD$ and $\cD'$, the equivalence problem reduces to the case $\cD=\cD'$. The next result solves it for the graded algebras with involution constructed in Subsection \ref{sse:def_MDpqsdd}.

\begin{theorem}\label{th:equivalenceMM'}
The graded algebras with involution $\cR=\cM(\cD,\varphi_0,q,s,\ul{d},\delta)$ and $\cR'=\cM(\cD,\varphi'_0,q',s',\ul{d}',\delta')$ are equivalent if and only if $q'=q$, $s'=s$, and there exist an equivalence $\psi_0:\cD\to\cD$ and an element $c_0\in\cD^\times_\gr$ such that 
\begin{enumerate}
\item[(i)] $\varphi'_0=\Int(c_0)\circ\psi_0\varphi_0\psi_0^{-1}$, 
\item[(ii)] $\varphi'_0(c_0)=\delta\delta' c_0$, and 
\item[(iii)] there exist $c_1,\ldots,c_q\in\cD^\times_\gr$ such that the $q$-tuple $(d'_1,\ldots,d'_q)$ is a permutation of $(\tilde{d}_1,\ldots,\tilde{d}_q)$ where $\tilde{d}_i\bydef c_0\psi_0\bigl(c_i d_i \varphi_0(c_i)\bigr)$.
\end{enumerate}
\end{theorem}

\begin{proof}
By definition, $\cR=\End_\cD(\cV)$ and $\cR'=\End_\cD(\cV')$, graded by $G=\wt{G}(T,q,s,\ul{t})$ and $G'=\wt{G}(T,q',s',\ul{t}')$, respectively, where $T$ is the support of $\cD$, $t_i=\deg d_i$, $t'_i=\deg d'_i$, $\cV$ has a $\cD$-basis $\{v_1,\ldots,v_k\}$ with $\deg v_i=\tilde{g}_i$ and $\cV'$ has a $\cD$-basis $\{v'_1,\ldots,v'_{k'}\}$ with $\deg v'_i=\tilde{g}'_i$, $k=q+2s$ and $k'=q'+2s'$. The involution $\varphi$ of $\cR$ is the adjunction with respect to the hermitian or skew-hermitian form $B$ represented by the matrix in equation \eqref{eq:PhiMDpqsdd}, and similarly for the involution $\varphi'$ of $\cR'$.
  
Suppose $\psi:\cR\to\cR'$ is an equivalence of graded algebras with involution and let $(\psi_0,\psi_1)$ be an equivalence $(\cD,\cV)\to(\cD,\cV')$ as in Theorem \ref{th:equivalence_with_anti}. First observe that $\psi_1$ must map $\{v_1,\ldots,v_k\}$ onto a graded basis of $\cV'$ (hence $k'=k$), which is necessarily a set of the form $\{v'_1 c'_1,\ldots,v'_k c'_k\}$ with $c'_i\in\cD^\times_\gr$, because the cosets $\tilde{g}'_1 T,\ldots,\tilde{g}'_k T$ are distinct. Since $\psi\varphi\psi^{-1}=\varphi'$, there exists $d_0\in\cD^\times_\gr$ such that \eqref{eq:B'B} and \eqref{eq:d0psi} hold. But $B(v_i,v_i)\ne 0$ if and only if $i\le q$, and similarly for $B'$, hence \eqref{eq:B'B} implies $q'=q$, $s'=s$, and $\psi_1$ must map the set $\{v_1,\ldots,v_q\}$ onto $\{v'_1 c'_1,\ldots,v'_q c'_q\}$. Conjugating both sides of \eqref{eq:d0psi} by $\psi_0$, we obtain (i) with $c_0\bydef\psi_0(d_0)$. The $(\psi_0^{-1}\varphi'_0\psi_0)$-sesquilinear form
$
B_0\bydef\psi_0^{-1}\circ B'\circ(\psi_1\times\psi_1)
$
satisfies $\wb{B_0}=\delta' B_0$ (since $B'$ has the same property), and \eqref{eq:B'B} reads $B_0=d_0 B$, so formula \eqref{eq:delta'} gives $\varphi_0(d_0)=\delta\delta' d_0$. 
Applying $\psi_0$ to both sides of this equation gives $(\psi_0\varphi_0\psi_0^{-1})(c_0)=\delta\delta' c_0$. On the other hand, inverting both sides of (i), we get $\varphi'_0=\psi_0\varphi_0\psi_0^{-1}\circ\Int(c_0^{-1})$, 
hence $(\psi_0\varphi_0\psi_0^{-1})(c_0)=\bigl(\varphi'_0\circ\Int(c_0)\bigr)(c_0)=\varphi'_0(c_0)$. This proves (ii). 
Finally, if $\psi_1^{-1}(v')=v c$ for some $v\in\cV$, $v'\in\cV'$ and $c\in\cD^\times_\gr$, then \eqref{eq:B'B} gives
\begin{equation}\label{eq:compute_diag}
B'(v',v')=\psi_0\bigl(d_0 B(vc,vc)\bigr)=c_0\psi_0\bigl(\varphi_0(c)B(v,v)c\bigr).
\end{equation}
Applying this for $v'\in\{v'_1,\ldots,v'_q\}$ and the corresponding $v\in\{v_1,\ldots,v_q\}$ and $c\in\cD^\times_\gr$, we obtain (iii).

To prove the converse, first observe that if we permute the entries of the $q$-tuple $\ul{d}$, we obtain a graded algebra $\cR''=\End_\cD(\cV'')$ with involution $\varphi''$ (coming from a $\varphi_0$-sesquilinear form $B''$ with $\overline{B''}=\delta B''$) that is weakly isomorphic to $\cR$. Indeed, let $\pi\in S_q$ and consider $\pi\cdot\ul{d}\bydef\bigl(d_{\pi^{-1}(1)},\ldots,d_{\pi^{-1}(1)}\bigr)$. From the defining relations \eqref{eq:rel_Gtilde}, we see that there exists an isomorphism $\alpha:G\to G''\bydef\wt{G}(T,q,s,\pi\cdot\ul{d})$ defined by $\alpha|_T=\id_T$ and $\alpha(\tilde{g}_i)=\tilde{g}''_{\pi(i)}$. Then $\psi_0\bydef\id_\cD$ can be regarded as an isomorphism of $G''$-graded algebras ${}^\alpha\cD\to\cD$ and the mapping $v_i\mapsto v''_{\pi(i)}$ extends to an isomorphism of $\cD$-modules $\psi_1:\cV\to\cV''$ that can be regarded as an isomorphism of $G''$-graded spaces ${}^\alpha\cV\to\cV''$. Hence the pair $(\psi_0,\psi_1)$ gives rise to an isomorphism of $G''$-graded algebras ${}^\alpha\cR\to\cR''$. As $B$ and $B''$ are $\varphi_0$-sesquilinear and $B''(\psi_1(v_i),\psi_1(v_j))=B(v_i,v_j)$ for all $1\le i,j\le k$, condition \eqref{eq:B'B} is satisfied for $B$ and $B''$ (with $\psi_0=\id_\cD$ and $d_0=1$), so $\psi\varphi\psi^{-1}=\varphi''$. 

Now suppose $q'=q$, $s'=s$, and conditions (i), (ii) and (iii) hold, where in the latter we may assume without loss of generality that $d'_i=c_0\psi_0\bigl(c_i d_i \varphi_0(c_i)\bigr)$ for all $1\le i\le q$. In particular, this implies the following relation between $t_i\bydef\deg d_i$ and $t'_i\bydef\deg d'_i$:
\begin{equation}\label{eq:degrees_agree}
t'_i=t_0\alpha\bigl(t_i (\deg c_i)^2\bigr)\;\text{ for all }1\le i\le q,
\end{equation}
where $t_0\bydef\deg c_0$ and $\alpha$ is the automorphism of $T$ defined by $\psi_0(\cD_t)=\cD_{\alpha(t)}$ for all $t\in T$. In view of the defining relations \eqref{eq:rel_Gtilde} and equation \eqref{eq:degrees_agree}, we have an isomorphism $G\to G'$ that extends $\alpha$ and maps $\tilde{g}_i\mapsto\tilde{g}'_i\alpha(\deg c_i)^{-1}$ for $1\le i\le q$, $\tilde{g}_i\mapsto\tilde{g}'_i$ for $i=q+2j-1$, and $\tilde{g}_i\mapsto\tilde{g}'_i t_0$ for $i=q+2j$ ($1\le j\le s$). We will denote this extension also by $\alpha$. Then $\psi_0$ can be regarded as an isomorphism of $G'$-graded algebras ${}^\alpha\cD\to\cD$, and the mapping $v_i\varphi_0(c_i)\mapsto v'_i$ for $1\le i\le q$, $v_i\mapsto v'_i$ for $i=q+2j-1$, and $v_i\mapsto v'_i c_0$ for $i=q+2j$ ($1\le j\le s$) extends to a $\psi_0$-semilinear isomorphism $\psi_1:\cV\to\cV'$ that can be regarded as an isomorphism of $G'$-graded spaces ${}^\alpha\cV\to\cV'$. Hence the pair $(\psi_0,\psi_1)$ gives rise to an isomorphism of $G'$-graded algebras ${}^\alpha\cR\to\cR'$. It remains to prove that condition \eqref{eq:B'B} is satisfied. Recall that it can be rewritten as $B_0=d_0 B$ where $d_0\bydef\psi_0^{-1}(c_0)$. Since $B_0$ is $(\psi_0^{-1}\varphi'_0\psi_0)$-sesquilinear, $d_0B$ is $\Int(d_0)\circ\varphi_0$-sesquilinear, and $\psi_0^{-1}\varphi'_0\psi_0=\Int(d_0)\circ\varphi_0$ by condition (i), it suffices to verify the equality $B_0=d_0B$ on some $\cD$-basis $\{w_1,\ldots,w_k\}$ of $\cV$. Moreover, we have $\wb{B_0}=\delta' B_0$, and condition (ii) implies that $\wb{d_0 B}=\delta' d_0 B$, so it is sufficient to verify that $B_0(w_i,w_j)=d_0 B(w_i,w_j)$ for all $1\le i\le j\le k$. This is straightforward using the basis $\{v_1\varphi_0(c_1),\ldots,v_q\varphi_0(c_q),v_{q+1},\ldots,v_{k}\}$ and a calculation similar to equation \eqref{eq:compute_diag}. 
\end{proof}

Condition (i) can be interpreted in terms of group actions as follows. Let $\Aut(\Gamma_\cD)$ be the group of equivalences from the graded algebra $(\cD,\Gamma_\cD)$ to itself or, in other words, automorphisms of the algebra $\cD$ that permute the components of $\Gamma_\cD$. This group acts on the set of degree-preserving antiautomorphisms of $\cD$ by conjugation, leaving the subset of involutions invariant. On the other hand, the group $\cD^\times_\gr$ acts on the same set as follows: an element $d\in\cD^\times_\gr$ moves an antiautomorphism $\varphi_0$ to $\Int(d)\circ\varphi_0$. These two actions are compatible in a way that they combine into an action of the semidirect product $\cD^\times_\gr\rtimes\Aut(\Gamma_\cD)$ with respect to the natural action of $\Aut(\Gamma_\cD)$ on $\cD^\times_\gr$. Then (i) implies that $\varphi_0$ and $\varphi'_0$ are in the same orbit of this action.

\subsection{The central simple case}

Now suppose we are in the setting of Subsection~\ref{sse:central_simple_basics}: $\cD$ is finite-dimensional, admits a degree-preserving involution $\varphi_0$, and is central simple as an algebra with involution, so  the center $\KK\bydef Z(\cD)$ is either $\FF$ or a quadratic \'etale algebra over $\FF$ with $\varphi_0|_\KK\ne\id_\KK$. Then the above action of $\cD^\times_\gr$ is transitive on the set of degree-preserving antiautomorphisms if $\KK=\FF$ and on the subset of those antiautomorphisms that act nontrivially on $\KK$ if $\KK\ne\FF$. In particular, it can move $\varphi_0$ to any other degree-preserving involution $\varphi'_0$ of the same kind. In fact, by Corollary \ref{cor:involutions_from_B}, $\varphi$ can be obtained from any such involution by choosing an appropriate $B$, so we may fix $\varphi_0$ for each such $\cD$ (as we did, for example, in the real case --- see Subsection \ref{sse:real_basics}). Moreover, $B$ is hermitian if $\chr{\FF}=2$ and can be chosen hermitian if $\KK\ne\FF$, which we will always do.

If there is a special involution on $\cD$ that commutes with all elements of $\Aut(\Gamma_\cD)$ and is of the right kind, as happens over $\RR$ for the distinguished involution if $\cD_e=\RR$ and $\KK\in\{\RR,\CC\}$ (see Subsection \ref{sse:real_basics}), then choosing this involution as $\varphi_0$ will considerably simplify the analysis: for example, condition (i) will say that $c_0$ is central. However, since we do not have such an involution in all cases, we continue in the general setting. 

With a fixed involution $\varphi_0$, we can express the remaining conditions (ii) and (iii) of Theorem \ref{th:equivalenceMM'} in terms of group actions as follows. 
Define a left action of the group $\cD^\times_\gr$ on the set $\cD^\times_\gr$ by $c*d\bydef cd\varphi_0(c)$. We denote the set of orbits by $X_{\varphi_0}(\cD)$ and the orbit of an element $d\in\cD^\times_\gr$ by $\cO(d)$.
Clearly, if $d$ is symmetric or skew-symmetric, then all elements in $\cO(d)$ have the same property. 
We also denote by $A_{\varphi_0}(\cD)$ the stabilizer of $\varphi_0$ in $\cD^\times_\gr\rtimes\Aut(\Gamma_\cD)$:
\[
A_{\varphi_0}(\cD)\bydef\{(c,\psi_0)\in\cD^\times_\gr\rtimes\Aut(\Gamma_\cD)\mid
\Int(c)\circ\psi_0\varphi_0\psi_0^{-1}=\varphi_0\}.
\]
The group $\cD^\times_\gr\rtimes\Aut(\Gamma_\cD)$ acts naturally on the set $\cD^\times_\gr$ via 
$(c,\psi_0)\cdot d\bydef c\psi_0(d)$, and the action of the subgroup $A_{\varphi_0}(\cD)$ 
passes to the quotient set $X_{\varphi_0}(\cD)$. 
Indeed, for $c',d\in\cD^\times_\gr$ and $(c,\psi_0)\in A_{\varphi_0}(\cD)$, we have 
$\varphi_0=\varphi_0^{-1}=\psi_0\varphi_0\psi_0^{-1}\circ\Int(c^{-1})$, so 
\[
\begin{split}
(c,\psi_0)\cdot\Bigl(\bigl(\psi_0^{-1}(c^{-1}c'c)\bigr)*d\Bigr)
&=c'c\psi_0(d)\psi_0\Bigl(\varphi_0\bigl(\psi_0^{-1}(c^{-1}c'c)\bigr)\Bigr)\\
&=c'c\psi_0(d)\varphi_0(c')=c'*\bigl((c,\psi_0)\cdot d\bigr),
\end{split}
\]
which proves that $(c,\psi_0)\cdot\cO(d)=\cO\bigl((c,\psi_0)\cdot d\bigr)$.

\begin{lemma}\label{lm:Aphi+Dsurj}
For any $\psi_0\in\Aut(\Gamma_\cD)$, there exists $c\in\cD^\times_\gr$ such that $(c,\psi_0)\in A_{\varphi_0}(\cD)$ and $\varphi_0(c)=c$. If $\KK=\FF$, then the former condition implies the latter.
\end{lemma}

\begin{proof}
We will apply Corollary \ref{cor:involutions_from_B} to $\cD$ as a graded right $\cD$-module of rank $1$. Since $\varphi\bydef\psi_0\varphi_0\psi_0^{-1}$ is a degree-preserving involution on $\cD=\End_\cD(\cD)$ of the same kind as $\varphi_0$, it is the adjunction with respect to a nondegenerate homogeneous $\varphi_0$-sesquilinear form $B$, which is automatically hermitian if $\KK=\FF$ (because $\varphi$ has the same type as $\varphi_0$ if $\chr{\FF}\ne 2$) and can be chosen hermitian otherwise. Looking at the ``matrix'' $c\bydef B(1,1)$, we get $\varphi_0(c)=c$ and $\varphi_0(d)c=B(d,1)=B(1,\varphi(d))=c\varphi(d)$ for all $d\in\cD$, which means $\varphi_0=\Int(c)\circ\varphi$. 
\end{proof}

It is easy to verify that if $(c,\psi_0)\in A_{\varphi_0}(\cD)$, $\varphi_0(c)=c$, and $d\in\cD^\times_\gr$ satisfies $\varphi_0(d)=\delta d$ for $\delta\in\{\pm 1\}$, then $d'\bydef c\psi_0(d)$ also satisfies $\varphi_0(d')=\delta d'$. It follows that
\begin{equation}\label{eq:Aphi+D}
A_{\varphi_0}^+(\cD)\bydef\{(c,\psi_0)\in\cD^\times_\gr\rtimes\Aut(\Gamma_\cD)\mid
\Int(c)\circ\psi_0\varphi_0\psi_0^{-1}=\varphi_0\text{ and }\varphi_0(c)=c\}
\end{equation}
is a subgroup of $A_{\varphi_0}(\cD)$, which leaves invariant the following subsets of $X_{\varphi_0}(\cD)$:
\begin{equation}\label{eq:Xp0d}
X_{\varphi_0}(\cD,\delta)\bydef\{\cO(d)\mid d\in\cD^\times_\gr,\,\varphi_0(d)=\delta d\}.
\end{equation}
By Lemma \ref{lm:Aphi+Dsurj}, the projection $A_{\varphi_0}^+(\cD)\to\Aut(\Gamma_\cD)$ is surjective. 
For example, if $\varphi_0$ commutes with all elements of $\Aut(\Gamma_\cD)$, then $A_{\varphi_0}^+(\cD)=\FF^\times\times\Aut(\Gamma_\cD)$.

For a given $q$-tuple $\ul{d}=(d_1,\ldots,d_q)$ of elements of $\cD^\times_\gr$, denote by $\wt{\Sigma}(\ul{d})$ the multiset $\{\cO(d_1),\ldots,\cO(d_q)\}$ of the corresponding orbits (of size $q$ if we count with multiplicity). The group $A_{\varphi_0}(\cD)$ acts on such multisets via its action on $X_{\varphi_0}(\cD)$.
 
Now Theorem \ref{th:equivalenceMM'} tells us that $\cM(\cD,\varphi_0,q,s,\ul{d},\delta)$ and $\cM(\cD,\varphi_0,q',s',\ul{d}',\delta')$ are equivalent as graded algebras with involution if and only if $q'=q$, $s'=s$, and there exists $(c_0,\psi_0)\in A_{\varphi_0}(\cD)$ that moves the multiset $\wt{\Sigma}(\ul{d})$ to $\wt{\Sigma}(\ul{d}')$ and satisfies $\varphi_0(c_0)=\delta\delta' c_0$. If $\KK\ne\FF$, the latter condition reads $\varphi_0(c_0)=c_0$, since in this case we use $\delta'=\delta=1$. If $\KK=\FF$, then we also have $\varphi_0(c_0)=c_0$ by Lemma \ref{lm:Aphi+Dsurj}. 

To summarize: 

\begin{corollary}\label{cor:equivalenceMM'cs}
If $(\cD,\varphi_0)$ is central simple as an algebra with involution, then the graded algebras with involution of the form $\cM(\cD,\varphi_0,q,s,\ul{d},\delta)$, subject to the convention that $\delta=1$ if $\KK\ne\FF$, are classified up to equivalence by $q$, $s$, $\delta$, and the orbit of the multiset $\wt{\Sigma}(\ul{d})$ in $X_{\varphi_0}(\cD,\delta)$ under the action of the group $A_{\varphi_0}^+(\cD)$.\qed
\end{corollary}

\subsection{The real case}

We now specialize to the setting of Subsection \ref{sse:real_basics}: the ground field $\FF$ is $\RR$ (or any real closed field) and hence $\KK\in\{\RR,\CC,\wt{\CC}\}$. We will also assume that $\cD_e\subset\KK$. Recall that, unless $\KK=\wt{\CC}$, we may choose $\varphi_0$ to be the distinguished involution or a member of the distinguished class (if $\cD_e=\KK=\CC$). 

\begin{proposition}\label{pr:realXp0}
Assume that $(\cD,\varphi_0)$ is central simple as an algebra with involution over $\RR$ and that $\cD_e\subset\KK$. If $\cD_e=\CC$, denote $\cH_t\bydef\{d\in\cD_t\mid\varphi_0(d)=d\}$, which is an $\RR$-subspace of dimension $1$, for any $t\in T\bydef\supp\cD$. Then, for the left action of $\cD^\times_\gr$ on itself via $c*d=cd\varphi_0(c)$, we have the following orbits:
\begin{enumerate}
\item[(1)] If $\KK=\RR$ and $\varphi_0$ is the distinguished involution, then the set $\cD_t\smallsetminus\{0\}$, for any $e\ne t\in T$, is an orbit, and there are two orbits in $\cD_e\smallsetminus\{0\}$, namely, $\cO(1)=\RR_{>0}$ and $\cO(-1)=\RR_{<0}$;
\item[(2)] If $\KK=\CC$, we have two cases:
\begin{enumerate}
\item[(a)] If $\cD_e=\RR$ and $\varphi_0$ is the distinguished involution, then the set $\cD_t\smallsetminus\{0\}$, for any $t\notin\{e,f\}\bydef\supp\KK$, is an orbit, and there are two orbits in each of the sets $\cD_e\smallsetminus\{0\}$ and $\cD_f\smallsetminus\{0\}$, namely, $\cO(1)$, $\cO(-1)$, $\cO(\bi)$ and $\cO(-\bi)$;
\item[(b)] If $\cD_e=\CC$ and $\varphi_0$ belongs to the distinguished class, then each set $\bigcup_{t\in S}\cH_t\smallsetminus\{0\}$, for a nontrivial coset $S$ of $T^{[2]}$ in $T$, is an orbit, and there are two orbits in $\bigcup_{t\in T^{[2]}}\cH_t\smallsetminus\{0\}$, namely, $\cO(1)$ and $\cO(-1)$;
\end{enumerate}
\item[(3)] If $\KK=\wt{\CC}$, then each set $\cD_t\smallsetminus\{0\}$, for any $t\in T$, is an orbit.
\end{enumerate}  
\end{proposition}

\begin{proof}
Recall the alternating bicharacter $\beta$ defined by equation \eqref{eq:def_beta}. For any $s,t\in T$, we have 
\begin{equation}\label{eq:c*d}
c*d=cd\varphi_0(c)=\beta(s,t)dc\varphi_0(c)\in\cD_{ts^2}\;\text{ for all }0\ne c\in\cD_s\text{ and }0\ne d\in\cD_t.
\end{equation}
Assume we are not in case (2b). Then $T$ is $2$-elementary, so all elements in an orbit have the same degree. It is clear that there can be at most two orbits in $\cD_t\smallsetminus\{0\}$. In case (3), we have $c^2\in\RR_{>0}$ for any $0\ne c\in\KK_f$, hence $c*d=dc(-c)\in\RR_{<0}d$, so there is actually only one orbit. In cases (1) and (2a), since $\varphi_0$ is the distinguished involution, we have $c\varphi_0(c)\in\RR_{>0}$ for all $c\in\cD^\times_\gr$. Hence, if $d$ is central, we get $c*d=dc\varphi_0(c)\in\RR_{>0}d$ for all $c\in\cD^\times_\gr$, which means that $d$ and $-d$ are not in the same orbit. On the other hand, if $0\ne d\in\cD_t$ for $t\notin\supp\KK=\rad\beta$, then we can find $s\in T$ with $\beta(s,t)=-1$, and then equation \eqref{eq:c*d} gives $c*d=-dc\varphi_0(c)\in\RR_{<0}d$, so there is only one orbit in $\cD_t\smallsetminus\{0\}$.

Now consider case (2b). From equation \eqref{eq:c*d}, we see that the orbit of an element $0\ne d\in\cH_t$ intersects $\cH_{t'}$  for any $t'$ in the coset $S\bydef tT^{[2]}$. Therefore, there can be at most two orbits in $\bigcup_{t'\in S}\cH_{t'}\smallsetminus\{0\}$. Since $\varphi_0$ belongs to the distinguished class, we have $c\varphi_0(c)\in\RR_{>0}$ for any $0\ne c\in\cD_s$ with $s\in T_{[2]}$. If $S$ is a nontrivial coset, i.e., $t\notin T^{[2]}$, then, by Lemma \ref{lm:perp}, we can find $s\in T_{[2]}$ such that $\beta(s,t)=-1$. But then, as above, equation \eqref{eq:c*d} gives $c*d\in\RR_{<0}d$, so there is only one orbit. On the other hand, the element $c*1=c\varphi_0(c)$ cannot be $-1$ for any $c\in\cD^\times_\gr$, because either its degree is not $e$ or else it belongs to $\RR_{>0}$.
\end{proof}

\begin{remark}
It is clear from the proof that the distinguished involution in cases (1) and (2a) is characterized by the property that $\cO(1)\ne\cO(-1)$. The same is true for the distinguished class of involutions in case (2b). Moreover, given any involution $\varphi_0$ in this class, the normalized elements of $\cO(1)$ form a basis of the graded subalgebra $\bigoplus_{t\in T^{[2]}}\cD_t$, which is called the \emph{distinguished basis} (associated to $\varphi_0$) in \cite{BKR_inv}.
\end{remark}

Proposition \ref{pr:realXp0} gives, in the real case, an explicit description of the orbit sets $X_{\varphi_0}(\cD,\delta)\subset X_{\varphi_0}(\cD)$, defined by equation \eqref{eq:Xp0d}. 
It is convenient to introduce the following map:
\[
p:X_{\varphi_0}(\cD)\to T/T^{[2]},\;\cO(d)\mapsto (\deg d)T^{[2]}.
\]
It is clear from the definition of the action $*$ that this map is well defined. In all cases except (2b), the codomain of $p$ is just $T$, and it is partitioned into subsets $T_{\varphi_0}(+1)$ and $T_{\varphi_0}(-1)$ defined by
\[
T_{\varphi_0}(\delta)\bydef\{t\in T\mid\varphi_0(d)=\delta d\text{ for all }d\in\cD_t\}.
\]
When $\varphi_0$ is the distinguished involution (cases (1a) and (2a)), we will abbreviate these sets as $T_+$ and $T_-$.
Then Proposition \ref{pr:realXp0} tells us this: 
\begin{enumerate}
\item[(1)] If $\KK=\RR$, then $p$ maps $X_{\varphi_0}(\cD,-1)$ bijectively onto $T_-$ and $X_{\varphi_0}(\cD,+1)$ surjectively onto $T_+$, with all elements of $T_+$ except $e$ having a unique preimage and $e$ having two preimages: $e_+\bydef\cO(1)$ and $e_-\bydef\cO(-1)$.
\item[(2)] If $\KK=\CC$, then $p$ maps $X_{\varphi_0}(\cD,+1)$ surjectively onto $T_+$ if $\cD_e=\RR$ and onto $T/T^{[2]}$ if $\cD_e=\CC$, with all elements except one, namely, $e$ and $T^{[2]}$, respectively, having a unique preimage and this element having two preimages: $e_+\bydef\cO(1)$ and $e_-\bydef\cO(-1)$.
\item[(3)] If $\KK=\wt{\CC}$, then $p$ maps $X_{\varphi_0}(\cD,+1)$ bijectively onto $T_{\varphi_0}(+1)$.
\end{enumerate}

Now we take into account the action of the subgroup $A_{\varphi_0}^+(\cD)\subset\cD^\times_\gr\rtimes\Aut(\Gamma_\cD)$, defined by equation \eqref{eq:Aphi+D}, on the set $X_{\varphi_0}(\cD)$ via $(c,\psi_0)\cdot\cO(d)=\cO(c\psi_0(d))$. This subgroup also acts on the set $T/T^{[2]}$ via the following action of $\cD^\times_\gr\rtimes\Aut(\Gamma_\cD)$ on $T$: 
$(c,\psi_0)\cdot t=\deg\bigl(c\psi_0(d)\bigr)$ for any $0\ne d\in\cD_t$. Our map $p$ is $A_{\varphi_0}^+(\cD)$-equivariant:
\[
\begin{split}
&p\bigl((c,\psi_0)\cdot\cO(d)\bigr)=p\bigl(\cO(c\psi_0(d))\bigr)=\deg\bigl(c\psi_0(d)\bigr)T^{[2]}\\
&=\bigl((c,\psi_0)\cdot(\deg d)\bigr)T^{[2]}=(c,\psi_0)\cdot\bigl((\deg d)T^{[2]}\bigr)=(c,\psi_0)\cdot p(\cO(d)).
\end{split}
\]
It follows that $A_{\varphi_0}^+(\cD)$ must fix the trivial coset in $T/T^{[2]}$ in all cases except (3), and permute its preimages $e_+$ and $e_-$. (This is obvious in cases (1) and (2a), because $A_{\varphi_0}^+(\cD)=\RR^\times\times\Aut(\Gamma_\cD)$ in these cases.)
Note that the element $(-1,\id_\cD)$ interchanges $e_+$ and $e_-$ and acts trivially on $T$. 
Therefore, if $\KK\ne\wt{\CC}$, we have the following invariant: 

\begin{df}\label{df:sign_multiset}
The \emph{signature} of a multiset in $X_{\varphi_0}(\cD,+1)$ is the nonnegative integer $|n^0_+-n^0_-|$ where $n^0_+$ and $n^0_-$ are the multiplicities of the elements $e_+$ and $e_-$, respectively. 
Explicitly, for the multiset $\wt{\Sigma}(\ul{d})$ associated to a $q$-tuple $\ul{d}=(d_1,\ldots,d_q)$ of symmetric elements of $\cD^\times_\gr$, the numbers $n^0_+$ and $n^0_-$ are as follows. In cases (1) and (2a), $n^0_+$ (respectively, $n^0_-$) is the number of positive (respectively, negative) real numbers among the elements $d_i$, whereas in case (2b), we have to consider all $d_i$ with degrees in $T^{[2]}$: if $t_i\bydef\deg d_i\in T^{[2]}$ then $d_i$ counts for $n^0_+$ (respectively, $n^0_-$) if the real number $cd_i\varphi_0(c)$ is positive (respectively, negative) for some (and hence any) $0\ne c\in\cD_s$ with $s^2=t_i^{-1}$. We will use the notation: $|n^0_+-n^0_-|=\mathrm{signature}(\ul{d})$.
\end{df}

Note that the three cases in the above definition are precisely those in which the (ungraded) algebra with involution $\cM(\cD,\varphi_0,q,s,\ul{d},\delta)$ can have a signature. Recall that if $\varphi$ is an orthogonal involution on $M_n(\RR)$, a symplectic involution on $M_n(\HH)$, or a second kind involution on $M_n(\CC)$, then the signature of $\varphi$ is defined as the absolute value of the signature of any hermitian matrix $\Psi$ such that $\varphi$ is given by $X\mapsto\Psi^{-1}\wb{X}^T\Psi$. Hence, the signature of the involution $\varphi$ on our $\cM(\cD,\varphi_0,q,s,\ul{d},\delta)$ is always defined in cases (2a) and (2b), while in case (1) it is defined if and only if $\delta=1$, because we chose $\varphi_0$ to be the distinguished involution of $\cD$, which is of the correct type to have a signature (see Subsection \ref{sse:real_basics}). The relationship between the signatures of $\ul{d}$ and $\varphi$ follows from formulas (19) and (20) in \cite{BKR_Lie}. If $\cD$ is isomorphic to $M_\ell(\Delta)$ as an ungraded algebra, where $\Delta\in\{\RR,\CC,\HH\}$, then
\begin{equation}\label{eq:signatures}
\mathrm{signature}(\varphi)=\frac{\ell}{\sqrt{|T^{[2]}|}}\,\mathrm{signature}(\ul{d}).
\end{equation}

Recall that the image of a finite multiset in a set $X$ under a mapping $p:X\to Y$ is the multiset in $Y$ obtained by counting all preimages of a given point with multiplicity. Formally, if $\wt{\kappa}:X\to\ZZ_{\ge 0}$ is the multiplicity function defining a multiset in $X$, then the multiplicity function $\kappa:Y\to\ZZ_{\ge 0}$ defining its image in $Y$ is given by $\kappa(y)=\sum_{x\in p^{-1}(y)}\wt{\kappa}(x)$ for all $y\in Y$. 

\begin{df}\label{df:simpler_multiset}
For a $q$-tuple $\ul{d}=(d_1,\ldots,d_q)$ of elements of $\cD^\times_\gr$ satisfying $\varphi_0(d_i)=\delta d_i$, denote by $\Sigma(\ul{d})$ the image of the multiset $\wt{\Sigma}(\ul{d})$ under $p:X_{\varphi_0}(\cD,\delta)\to T/T^{[2]}$. In other words, $\Sigma(\ul{d})$ is the multiset $\{t_1 T^{[2]},\ldots,t_q T^{[2]}\}$ where $t_i=\deg d_i$.   
\end{df}

Since the unordered pair $\{n^0_+,n^0_-\}$ as in in Definition \ref{df:sign_multiset} can be recovered from $|n^0_+-n^0_-|=\mathrm{signature}(\ul{d})$ and $n^0_++n^0_-$, which is the multiplicity of the trivial coset $T^{[2]}$ in $\Sigma(\ul{d})$, we can now  forget about $X_{\varphi_0}(\cD,\delta)$ and consider only $T/T^{[2]}$. The action of the group $\cD^\times_\gr\rtimes\Aut(\Gamma_\cD)$ on this set is defined via its action on $T$, which factors through the natural action of its quotient group $T\rtimes W(\Gamma_\cD)$, namely, $(s,\alpha)\cdot t=s\alpha(t)$ for all $s,t\in T$ and $\alpha\in W(\Gamma_\cD)$. Therefore, in all cases it is sufficient to consider the image of the subgroup $A_{\varphi_0}^+(\cD)$ in $T\rtimes W(\Gamma_\cD)$. Note that, by Lemma~\ref{lm:Aphi+Dsurj}, the projection of this image to $W(\Gamma_\cD)$ is surjective.

Recall that, except in case (2b), we have a quadratic form $\mu$ on $T$, regarded as a vector space over $GF(2)$, which is defined by condition \eqref{eq:def_mu} and defines the distinguished involution by $x\mapsto\mu(t)x$ for all $x\in\cD_t$. In cases (1) and (2a), the element $e$ is fixed, so the image of $A_{\varphi_0}^+(\cD)$ is contained in the subgroup $W(\Gamma_\cD)$ and hence must be equal to it, because the projection is surjective. By Proposition \ref{pr:Weyl_group_R}, we have $W(\Gamma_\cD)=\Aut(T,\mu)$.

In case (2b), the element $T^{[2]}$ of $T/T^{[2]}$ is fixed, which means that, for any $(c,\psi_0)\in A_{\varphi_0}^+(\cD)$, we have $c\in T^{[2]}$. Therefore, the action of $A_{\varphi_0}^+(\cD)$ on $T/T^{[2]}$ goes through the projection $A_{\varphi_0}^+(\cD)\to W(\Gamma_\cD)$ and the natural action of $W(\Gamma_\cD)$. This projection is surjective, and $W(\Gamma_\cD)$ is computed in  Proposition \ref{pr:Weyl_group_C_as_R}: $\Aut(T,\beta)$ if $T$ is $2$-elementary and $\Aut(T,\beta)\rtimes\langle\tau\rangle$ otherwise. However, the action of $\tau$ is trivial modulo $T^{[2]}$, so we only need to consider the action of $\Aut(T,\beta)$ on $T/T^{[2]}$.

In case (3), the distinguished involution is trivial on $\KK$, so we have to take $\varphi_0$ defined by $x\mapsto\eta(t)x$ for all $x\in\cD_t$, where $\eta\bydef\chi_0\mu$ and $\chi_0:T\to\{\pm 1\}$ is any fixed character satisfying $\chi_0(f)=-1$, so that $\eta$ is a quadratic form with the same polarization as $\mu$ but satisfies $\eta(f)=-1$. An element $\psi_0\in\Aut(\Gamma_\cD)$ commutes with the distinguished involution, but not necessarily with our $\varphi_0$:
\[
(\psi_0\varphi_0\psi_0^{-1})(x)=\eta(\alpha^{-1}(t))x=\chi_0(\alpha^{-1}(t))\chi_0(t)^{-1}\varphi_0(x)\;\text{ for all }x\in\cD_t\text{ and }t\in T,
\]
where $\alpha$ is the image of $\psi_0$ in $W(\Gamma_\cD)=\Aut(T,\mu)$. Since the projection $A_{\varphi_0}^+(\cD)\to\Aut(\Gamma_\cD)$ is surjective and $\cD$ is central simple as an algebra with involution, there must be a symmetric element $c\in\cD^\times_\gr$, determined up to a factor in $\RR^\times$, such that $(c,\psi_0)\in A_{\varphi_0}^+(\cD)$. 
Indeed, $c$ is any nonzero element in $\cD_{t_\alpha}$ where $t_\alpha$ is the unique element of $T$ determined by the conditions $\eta(t_\alpha)=1$ and $\beta(t_\alpha,t)=\chi_0(\alpha^{-1}(t))^{-1}\chi_0(t)$ for all $t\in T$. Therefore, the action of $A_{\varphi_0}^+(\cD)$ on $T$ goes through the projection $A_{\varphi_0}^+(\cD)\to W(\Gamma_\cD)$ and the following ``twisted action'' of $W(\Gamma_\cD)$: 
\begin{equation}\label{eq:twisted_case3}
\alpha\cdot t\bydef\alpha(t)t_\alpha\;\text{ for all } \alpha\in W(\Gamma_\cD)\text{ and }t\in T,
\end{equation} 
which is the composition of the group monomorphism $\iota:W(\Gamma_\cD)\to T\rtimes W(\Gamma_\cD)$, $\alpha\mapsto(t_\alpha,\alpha)$, and the natural action of $T\rtimes W(\Gamma_\cD)$.

Thus, in all cases, we have reduced the problem of equivalence to affine actions of certain groups on vector spaces over $GF(2)$. Case (1) is especially easy: we have the natural action of the orthogonal group $\Aut(T,\mu)$ on $T$. The following result will allow us to compress information to the quotient $\wb{T}\bydef T/\rad\beta$ in cases (2a) and (3), and also simplify the action in case (3). Since we regard $T$ as a vector space, \emph{we will now switch to additive notation for $T$}. If a vector space $V$ over $GF(2)$ has a quadratic form $Q$ (possibly singular), we will denote the stabilizer of $Q$ in $\GL(V)$ by $\Ort(V,Q)$ or simply $\Ort(V)$, and the stabilizer of the polarization $\langle\cdot,\cdot\rangle$ of $Q$ (an alternating form, possibly degenerate) by $\SP(V)$. 

\begin{lemma}\label{lm:actions}
Let $Q$ be a quadratic form on a finite-dimensional vector space $V$ over $GF(2)$ whose polarization $\langle\cdot,\cdot\rangle$ is nondegenerate. Let $T=V\oplus\langle f\rangle$ and extend $Q$ to a quadratic form on $T$ in two ways: $Q_0(f)=0$ and $Q_1(f)=1$, both having the same polarization, also denoted by $\langle\cdot,\cdot\rangle$, whose radical is spanned by $f$ (so $Q_0$ is singular and $Q_1$ is nonsingular). Then:
\begin{enumerate}
\item[(i)] There exists a group isomorphism $\iota$ from $V\rtimes\SP(V)\simeq\SP(T)$ onto the stabilizer of $Q_1$ in the affine group $T\rtimes\GL(T)$ such that $\pi\bigl(\iota(z)\cdot t\bigr)=z\cdot\pi(t)$ for all $z\in V\rtimes\SP(V)$ and $t\in T$, where $\cdot$ denotes the natural action of affine groups on $T$ and $V$, and $\pi:T\to V$ is the projection. 
\item[(ii)] This isomorphism maps the subgroup $V\rtimes\Ort(V)\simeq\Ort(T,Q_0)$ onto the stabilizer of $Q_1$ in $T\rtimes\Ort(T,Q_0)$.
\item[(iii)] This isomorphism maps the subgroup $\SP(V)$ onto $\Ort(T,Q_1)$.
\end{enumerate}
\end{lemma}

\begin{proof}
For an element $\bar{\alpha}\in\SP(V)$, the quadratic form $Q\circ\bar{\alpha}^{-1}$ has the same polarization as $Q$, so we can measure how far $\bar{\alpha}$ is from being in the orthogonal group $\Ort(V)$ by looking at the \emph{linear form} $Q\circ\bar\alpha^{-1}-Q$. Since $\langle.,.\rangle$ is nondegenerate, there is a unique $v_{\bar{\alpha}}\in V$ such that $Q(\bar{\alpha}^{-1}(u))-Q(u)=\langle v_{\bar{\alpha}},u\rangle$ for all $u\in V$ or, equivalently, $Q(u)-Q(\bar{\alpha}(u))=\langle v_{\bar{\alpha}},\bar\alpha(u)\rangle$ for all $u\in V$. As a function of $\bar{\alpha}$, the expression $Q\circ\bar\alpha^{-1}-Q$ is a $1$-coboundary of $\SP(V)$ with values in the space of quadratic forms on $V$. Since the isomorphism $V\to V^*$ sending $v\mapsto\langle v,\cdot\rangle$ is $\SP(V)$-equivariant, it follows that the map $\bar{\alpha}\mapsto v_{\bar{\alpha}}$ is a $1$-cocycle of $\SP(V)$ with values in $V$, and hence the map $(v,\bar{\alpha})\to (v+v_{\bar{\alpha}},\bar \alpha)$ is an automorphism of $V\rtimes \SP(V)$.

In the same vein, given any $\alpha\in\SP(T)$, $Q_1\circ\alpha^{-1}-Q_1$ is a linear form that annihilates $f$, and 
hence there is an element $t_\alpha\in T$ such that $Q_1(x)-Q_1(\alpha(x))=\langle t_\alpha,\alpha(x)\rangle$ for all 
$x\in T$. This element $t_\alpha$ is unique if we impose $Q_1(t_\alpha)=0$, and we will do it in what follows. 
(In fact, one can check that the map $\alpha\mapsto t_{\alpha}$ is a $1$-cocycle.) 

Now, the stabilizer $S$ of $Q_1$ in $\AGL(T)=T\rtimes\GL(T)$ consists of all pairs $(t,\alpha)$ such that 
$Q_1(t+\alpha(x))=Q_1(x)$ for all $x\in T$. Putting $x=0$, we get $Q_1(t)=0$, and our condition becomes   
$\langle t,\alpha(x)\rangle=Q_1(x)-Q_1(\alpha(x))$ for all $x\in T$. Linearizing in $x$, we get 
$0=\langle x,y\rangle-\langle \alpha(x),\alpha(y)\rangle$ for all $x,y\in T$. 
We conclude that $(t,\alpha)$ lies in $S$ if and only if $\alpha\in\SP(T)$ and $t=t_\alpha$.
Therefore, the projection $\AGL(T)\to\GL(T)$ maps $S$ isomorphically onto $\SP(T)$, and the inverse isomorphism $\SP(T)\to S$ is given by $\alpha\mapsto(t_\alpha,\alpha)$.

On the other hand, any $\alpha\in\SP(T)$ fixes $f$ and, identifying $V$ with $T/\langle f\rangle$, induces an element $\bar{\alpha}\in\SP(V)$. It follows that there is a unique element $v_\alpha\in V$ such that $\alpha(u)=\bar{\alpha}(u)+\langle v_\alpha,\bar{\alpha}(u)\rangle f$ for all $u\in V$, and this gives a group isomorphism $\SP(T)\to V\rtimes\SP(V)$, $\alpha\mapsto (v_\alpha,\bar\alpha)$. Also, for any $u\in V$, we have
\[ 
\begin{split}
Q(u)&=Q_1(u)=\langle t_\alpha,\alpha(u)\rangle+Q_1(\alpha(u)) 
=\langle \pi(t_\alpha),\bar{\alpha}(u)\rangle+Q(\bar{\alpha}(u))+\langle v_\alpha,\bar{\alpha}(u)\rangle \\
&=\langle \pi(t_\alpha)+v_\alpha,\bar{\alpha}(u)\rangle+Q(\bar{\alpha}(u)),
\end{split}
\]
and this gives $v_{\bar{\alpha}}=\pi(t_\alpha)+v_\alpha$, or $\pi(t_\alpha)=v_\alpha+v_{\bar{\alpha}}$.

We have obtained a chain of group isomorphisms:
\[
\begin{array}{ccccccl}
S&\simeq& \SP(T)&\simeq&V\rtimes\SP(V)&\simeq&V\rtimes\SP(V)\\
(t_\alpha,\alpha)&\leftrightarrow&\alpha&\leftrightarrow
&(v_\alpha,\bar{\alpha})&\leftrightarrow
&(v_\alpha+v_{\bar{\alpha}},\bar{\alpha})=(\pi(t_\alpha),\bar{\alpha})
\end{array}
\]
Composing them from right to left, we define the desired group isomorphism 
$\iota:V\rtimes \SP(V)\to S$, $(v,\bar\alpha)\mapsto (t_\alpha,\alpha)$ where $t_\alpha=v+Q(v)f$ and
$\alpha\in\SP(V)$ is defined by $f\mapsto f$ and $u\mapsto\bar{\alpha}(u)+\langle v+v_{\bar{\alpha}},\bar{\alpha}(u)\rangle f$ for all $u\in V$.
Indeed, for any $t\in T$ and $\alpha\in\SP(V)$, we have $\pi((t_\alpha,\alpha)\cdot t)=\pi(t_\alpha+\alpha(t))=(\pi(t_\alpha),\bar{\alpha})\cdot \pi(t)$, as desired. This proves part (i).

Part (ii) follows since $\alpha$ lies in $\Ort(T,Q_0)$ if and only if $\bar\alpha$ lies in $\Ort(V)$, if and only if $v_{\bar{\alpha}}=0$. As to part (iii), $\alpha$ lies in $\Ort(T,Q_1)$ if and only if $t_\alpha=0$, if and only if
$\pi(t_\alpha)=0$.
\end{proof}

Part (iii) implies the well-known fact $\Ort(2m+1,2)\simeq\SP(2m,2)$ and allows us to replace the action of $\Aut(T,\mu)$ on $T$ in case (2a) by the action of $\Aut(\wb{T},\bar{\beta})$ on $\wb{T}$, since the quotient map $\pi:T\to\wb{T}$ gives a bijection of $T_+$ onto $\wb{T}$. Similarly, in case (3), we identify $\wb{T}$ with $V\bydef\ker\chi_0$, which is a complement for $\rad\beta$ in $T$, and apply part (ii) to replace the twisted action of $\Aut(T,\mu)$ on $T$, given by equation \eqref{eq:twisted_case3}, by the natural action of $\wb{T}\rtimes\Aut(\wb{T},\bar{\mu})$ on $\wb{T}$. So, in these cases we will use the image of the multiset $\Sigma(\ul{d})$  (see Definition \ref{df:simpler_multiset}) in $\wb{T}$.

\begin{remark}\label{rm:compare_with_C}
Part (i) can be applied to classify graded algebras with involution like in our case (3), but over an algebraically closed field. This gives an alternative approach to the classification of fine gradings on simple Lie algebras of series $A$, where appears the natural action of $\wb{T}\rtimes\Aut(\wb{T},\bar{\beta})$ on $\wb{T}$ (compare with \cite[Theorem 3.35]{EKmon}, but note that $T$ there corresponds to our $\wb{T}$). The twisted action of $\Aut(T,\beta)$ on $T$, which is used to classify graded algebras with involution like in our case (1), but over an algebraically closed field (see \cite[Theorem 3.42]{EKmon}), and, consequently, appears in the classification of fine gradings on simple Lie algebras of series $C$ and $D$ (as $T$ is trivial in series $B$), corresponds to the action $\bar{\alpha}\cdot v\bydef\bar{\alpha}(v)+v_{\bar{\alpha}}$ for $v\in V$ and $\bar{\alpha}\in\SP(V)$, where $v_{\bar{\alpha}}$ is defined as in the proof of Lemma \ref{lm:actions}. Of course, our cases (2a) and (2b) have no analogs over an algebraically closed field.
\end{remark}

Finally, the group $\Aut(T,\beta)$ can be rather unpleasant in case (2b), but the following result describes its image in $\Aut(T/T^{[2]})$, i.e., in the general linear group of the vector space $V\bydef T/T^{[2]}$ over $GF(2)$. The images of the elements of even order in any symplectic basis $\{a_1,b_1,\ldots,a_m,b_m\}$ of $(T,\beta)$ constitute a basis of $V$. It is convenient to choose the symplectic basis separately for each primary component of $T$, so the orders are prime powers, and only the powers of $2$ contribute to $V$. For each integer $i\ge 0$, let $V_i$ be the subspace of $V$ spanned by the images of $a_j$ and $b_j$ whose order is a power of $2$ that does not exceed $2^i$. These subspaces form a flag $\cF$ in $V$ whose \emph{length}, by which we mean the number of nontrivial quotients $W_i\bydef V_i/V_{i-1}$, is at most $\log_2\bigl(\exp(M)\bigr)$, where $M$ is the $2$-primary component of $T$ and $\exp(M)$ is its exponent (the maximum order of elements). Moreover, each nontrivial quotient $W_i$ carries a nondegenerate alternating form $\langle\cdot,\cdot\rangle_i$ for which the images of the $a_j$ and $b_j$ whose order is exactly $2^i$ constitute a symplectic basis. These objects can be defined in a basis-free manner: 
\begin{align}
&\cF=\{V_i\}_{i=0,1,\ldots}\;\text{ with }V_i\bydef T_{[2^i]}+T^{[2]}/T^{[2]},\label{eq:def_F}\\
&W_i\bydef V_i/V_{i-1}\simeq T_{[2^i]}+T^{[2]}/T_{[2^{i-1}]}+T^{[2]},\nonumber
\end{align}
where the form $\langle\cdot,\cdot\rangle_i$ on $W_i=V_i/V_{i-1}\simeq T_{[2^i]}/\bigl(T_{2^{[i-1]}}+(T_{[2^i]}\cap T^{[2]})\bigr)$ is induced by $\beta^{2^{i-1}}$. It is clear from this definition that the action of $\Aut(T,\beta)$ on $V$ stabilizes the flag $\cF$ and each symplectic form $\langle\cdot,\cdot\rangle_i$.

\begin{proposition}\label{pr:flag}
Let $\beta$ be a nondegenerate alternating bicharacter on a finite abelian group $T$ and consider the natural action of $\Aut(T,\beta)$ on the vector space $V\bydef T/T^{[2]}$. Let $\cF$ be the flag in $V$ defined by equation \eqref{eq:def_F}. Then the image of $\Aut(T,\beta)$ in $\GL(V)$ is equal to the stabilizer $\SP(V,\cF)$ of the flag $\cF$ and the symplectic forms on each of its nontrivial quotients.
\end{proposition}

We postpone the proof of Proposition \ref{pr:flag} until Subsection \ref{sse:flag} and summarize what we have proved here:

\begin{theorem}\label{th:equivalenceMM'real}
Let $\cD$ be a finite-dimensional graded-division algebra over $\RR$, with support $T$ and center $\KK$, that satisfies $\cD_e\subset\KK$ and admits a degree-preserving involution that makes it central simple as an algebra with involution over $\RR$. Then, subject to the convention that, if $\KK\ne\wt{\CC}$, we choose the involution $\varphi_0$ to satisfy $x\varphi_0(x)\in\RR_{>0}$ for all $0\ne x\in \cD_t$ with $t^2=e$ and that, if $\KK\ne\RR$, we use only hermitian forms (i.e., choose $\delta=1$), the graded algebras with involution of the form $\cM(\cD,\varphi_0,q,s,\ul{d},\delta)$ are classified up to equivalence by $q$, $s$, $\delta$, $\mathrm{signature}(\ul{d})$ (see Definition \ref{df:sign_multiset}) and the orbit of the multiset $\{\deg(d_1) H,\ldots,\deg(d_q)H\}$ in the elementary abelian $2$-group $V\bydef T/H$ under the action of $A$, where the subgroup $H\subset T$ and the group $A$ that acts on $V$ are as follows: 
\begin{enumerate}
\item[(1)] If $\KK=\RR$, then $H$ is trivial and $A=\Aut(T,\mu)$ where $\mu$ is the quadratic form on $T$ defined by condition \eqref{eq:def_mu};
\item[(2)] If $\KK=\CC$, we have two cases:
\begin{enumerate}
\item[(a)] If $\cD_e=\RR$, then $H=\supp\KK$ and $A=\Aut(V,\bar{\beta})$ where $\bar{\beta}$ is the nondegenerate alternating bicharacter on $V$ induced by the alternating bicharacter $\beta$ on $T$ defined by equation \eqref{eq:def_beta};
\item[(b)] If $\cD_e=\CC$, then $H=T^{[2]}$ and $A=\Aut(T,\beta)$, whose image in $\GL(V)$ is the stabilizer of the flag $\cF$ defined by equation \eqref{eq:def_F} and the symplectic forms on each of its nontrivial quotients (see Proposition \ref{pr:flag}); 
\end{enumerate}
\item[(3)] If $\KK=\wt{\CC}$, then $H=\supp\KK$ and $A=V\rtimes\Aut(V,\bar{\mu})$ where $\bar{\mu}$ is the quadratic form on $V$ induced by $\mu$. 
\end{enumerate}  
\end{theorem}

Since the alternating forms in cases (2a) and (2b) and the polarizations of the quadratic forms in cases (1) and (3) are nondegenerate, Witt's Extension Theorem (which is valid for quadratic forms even in characteristic $2$, see e.g. \cite[Theorem 8.3]{EKM}) tells us that any isometry between subspaces extends to an isometry of the entire space: $T$ in case (1), $V=T/\langle f\rangle$ in cases (2a) and (3), and each nontrivial quotient $V_i/V_{i-1}$ in case (2b). 

The case $(q,s)=(1,0)$ gives the classification of graded algebras with involution 
$(\cD,\varphi)$ up to equivalence. It follows from Witt's Extension Theorem that $A$ acts transitively on $T_+\smallsetminus\{e\}$ and $T_-$ in case (1), on $V\smallsetminus\{0\}$ in case (2a), and on each nonempty $V_i\smallsetminus V_{i-1}$ in case (2b). It is obvious that $A$ acts transitively on the entire $V$ in case (3). To summarize:
\begin{corollary}
Under the conditions of Theorem \ref{th:equivalenceMM'real}, assume $\cD\ne\KK$. Then the number of equivalence classes of graded algebras with involution
$(\cD,\varphi)$ per each equivalence class of $\cD$ is as follows:
\begin{enumerate}
\item[(1)] If $\KK=\RR$, there are $3$ classes: the distinguished involution and all other involutions of each of the two types (orthogonal and symplectic);
\item[(2a)] If $\KK=\CC$ with nontrivial grading, there are $2$ classes: the distinguished involution and all other involutions (of the second kind);
\item[(2b)] If $\KK=\CC$ with trivial grading, there are $\mathrm{length}(\cF)+1$ classes: the distinguished class and one class for each nontrivial quotient in $\cF$;  
\item[(3)] If $\KK=\wt{\CC}$, there is only $1$ class.\qed
\end{enumerate}
\end{corollary}

Theorem \ref{th:equivalenceMM'real} completes the classification of fine gradings on finite-dimensional central simple algebras with involution over $\RR$, given in Theorem \ref{th:fine_real}, by determining which of the graded algebras with involution therein are equivalent to each other:

\begin{corollary}\label{cor:fine_real}
The graded algebras with involution appearing in different items of Theorem \ref{th:fine_real} are not equivalent to each other. Within each item, they are classified up to equivalence by the following invariants, where $T=\supp\cD$ and $H=\supp Z(\cD)$: 
\begin{enumerate}
\item[(1)] Central simple algebras over $\RR$ with a first kind involution:
\begin{itemize}
\item $\cM(2m;\RR;q,s,\ul{d},\delta)$ 
by $m$, $q$, $s$, $\delta$, $\mathrm{signature}(\ul{d})$ in the case $\delta=1$, and the orbit of the multiset $\{\deg d_1,\ldots,\deg d_q\}$ in $T_\delta$ under the action of the orthogonal group $\Ort_+(2m,2)$ with trivial Arf invariant;

\item $\cM(2m;\HH;q,s,\ul{d},\delta)$ 
by $m$, $q$, $s$, $\delta$, $\mathrm{signature}(\ul{d})$ in the case $\delta=-1$, and the orbit of the multiset $\{\deg d_1,\ldots,\deg d_q\}$ in $T_{-\delta}$ under the action of the orthogonal group $\Ort_-(2m,2)$ with nontrivial Arf invariant;
\end{itemize}

\item[(2)] Central simple algebras over $\CC$ with a second kind involution:
\begin{itemize}
\item $\cM^{\mathrm{(I)}}(\ell_1,\ldots,\ell_m;\CC;q,s,\ul{d})$ 
by the multiset $\{\ell_1,\ldots,\ell_m\}$, $q$, $s$, $\mathrm{signature}(\ul{d})$, and the orbit of the multiset $\{\deg(d_1)T^{[2]},\ldots,\deg(d_q)T^{[2]}\}$ in the vector space $V\bydef T/T^{[2]}$ over $GF(2)$ under the action of the group $\SP(V,\cF)$ defined in Proposition \ref{pr:flag};

\item $\cM^{\mathrm{(II)}}(2m+1;\CC;q,s,\ul{d})$ 
by $m$, $q$, $s$, $\mathrm{signature}(\ul{d})$, and the orbit of the multiset 
$\{\deg(d_1)H,\ldots,\deg (d_q)H\}$ in $T/H$ under the action of the symplectic group $\SP(2m,2)$;
\end{itemize}

\item[(3)] Nonsimple algebras with an exchange involution:
\begin{itemize}
\item $\cM^{\mathrm{(I)}}(2m;\RR;k)$ by $m$ and $k$;

\item $\cM^{\mathrm{(I)}}(2m;\HH;k)$ by $m$ and $k$;

\item $\cM^{\mathrm{(II)}}(2m+1;\RR;q,s,\ul{d})$ by $m$, $q$, $s$, and the orbit of the multiset \\
$\{\deg(d_1)H,\ldots,\deg (d_q)H\}$ in $T/H$ under the action of the affine orthogonal group $\AO_+(2m,2)$ 
with trivial Arf invariant;

\item $\cM^{\mathrm{(II)}}(2m+1;\HH;q,s,\ul{d})$ by $m$, $q$, $s$, and the orbit of the multiset \\
$\{\deg(d_1)H,\ldots,\deg (d_q)H\}$ in $T/H$ under the action of the affine orthogonal group $\AO_-(2m,2)$ 
with nontrivial Arf invariant.\qed
\end{itemize}
\end{enumerate}
\end{corollary}

Similarly to Definitions \ref{df:M} and \ref{df:Mex}, we will denote the grading on the algebra with involution $\cM(2m;\RR;q,s,\ul{d},\delta)$ by $\Gamma_\cM(2m;\RR;q,s,\ul{d},\delta)$, etc. In other words, $\Gamma_\cM(2m;\RR;q,s,\ul{d},\delta)$ is $\Gamma_\cM(\cD,\varphi_0,q,s,\ul{d},\delta)$ where $\cD=\cD(2m;+1)$ and $\varphi_0$ is the matrix transposition. We will denote in the same way the corresponding isomorphism class of gradings on any algebra with involution $(\cR,\varphi)$ that is isomorphic to $\cM(2m;\RR;q,s,\ul{d},\delta)$, etc.

To conclude this subsection, we show how the flag $\cF$ in $V\bydef T/T^{[2]}$ can be used to determine the universal group $U=\wt{G}^0(T,q,s,\ul{t})$ of the grading $\Gamma_\cM(\cD,\varphi_0,q,s,\ul{d},\delta)$, where $t_j\bydef\deg d_j$, $1\le j\le q$. Note that the flag $\cF$ is defined in all cases, but it has length $1$ if $T$ is a nontrivial elementary $2$-group. 

Recall the presentation of $U$ from Proposition \ref{pr:G0_is_universal}, where we will take $i_0=1$, i.e., set $u_1=e$. 
If $s>0$, the relation involving $u_{q+2j-1}$ and $u_{q+2j}$ reads $u_{q+2j-1}=u_{q+2j}^{-1}u_2$ if $q=0$ and $u_{q+2j-1}=u_{q+2j}^{-1}t_1^{-1}$ if $q>0$, so in any case we get
\[
U\simeq U'\times\ZZ^s
\]
where $\ZZ^s$ is freely generated by $u_{q+2},u_{q+4},\ldots,u_{q+2s}$ and $U'$ is the subgroup of $U$ generated by $T$ and the symbols $u_2,\ldots,u_q$ if $q>1$ and simply $T$ if $q\le 1$. There is nothing further to do in the case $q\le 1$, so assume $q>1$. Let $M$ be the $2$-primary component of $T$. The isomorphism class of $M$ is determined by the partition of the integer $\log_2|M|$ that records the logarithms of the orders of the factors (in descending order) of  any decomposition of $M$ as a direct product of cyclic groups. It will be more convenient to use the dual partition $\nu_1\ge\nu_2\ge\ldots$ (obtained by reflecting the Young diagram with respect to the diagonal), which has the following meaning: $\nu_i$ is the number of cyclic groups in the decomposition of $M$ that have order at least $2^i$. Hence this dual partition is directly related to the dimensions of the quotients of $\cF$: $\dim W_i=\nu_i-\nu_{i+1}$.

Now, the defining relations of $U'$ read $u_j^2=t_j t_1^{-1}$, so $U'$ is obtained from $T$ by a sequence of extensions of index $2$. In particular, $U'$ is the torsion part of $U$. To determine its isomorphism class, consider the group homomorphism $\tau:\ZZ_2^{q-1}\to T/T^{[2]}$ defined by sending the $j$-th element of the standard basis to $t_{j+1}t_1^{-1}T^{[2]}$. (In fact, $\tau$ can be identified with the element of $\Ext_\ZZ(\ZZ_2^{q-1},T)$ associated to $U'$ as an extension of $\ZZ_2^{q-1}$ by $T$.) 
To compute the isomorphism class of $U'$ explicitly, we define a flag of subspaces of $\ZZ_2^{q-1}$ by pulling $\cF$ back along $\tau$:
\[
U_i\bydef\tau^{-1}(V_{i-1})\text{ for }i\ge 1\text{ and }U_0\bydef 0.
\]
Then, choosing a basis of $\ZZ_2^{q-1}$ adapted to this flag, we can present $U'$ with new generators $x_2,\ldots,x_q$ such that the defining relations are $x_j^2=t'_j$ where the elements $t'_j\in T$ satisfy the following: the first $\dim U_1$ of them are equal to $e$, the next $\dim(U_2/U_1)$ of them have order $2$ and their $T^{[2]}$-cosets are linearly independent, and so on. Adding one generator $x_j$ at a time results in adding a cyclic factor of order $2$ or replacing a cyclic factor of order $o(t'_j)$ by one of order $2o(t'_j)$. Therefore, the partition $\nu'_1\ge\nu'_2\ge\ldots$ that determines the isomorphism class of the $2$-primary component of $U'$ is given by $\nu'_i=\nu_i+\dim(U_i/U_{i-1})$. The odd primary components of $U'$ are the same as those of $T$, so we get the following generalization of equation (3.33) from \cite{EKmon}:
\begin{equation}\label{eq:iso_type_U}
n'_i(p)=\begin{cases}
n_i(2)+\dim(U_i/U_{i-1})-\dim(U_{i+1}/U_i) & \text{if $p=2$ and $q>1$}; \\
n_i(p) & \text{otherwise};
\end{cases}
\end{equation}
where $n_i(p)$ and $n'_i(p)$ denote the number of cyclic factors of prime power order $p^i$ in $T$ and in $U'$, respectively.

\subsection{Examples}\label{sse:examples}

The first example shows that the complexification of a fine grading on a central simple real algebra (associative with involution or Lie) may fail to be fine and the second example will be used in Section \ref{se:Lie} to classify fine gradings on the real forms of the simple complex Lie algebra of type $D_4$. The universal groups can be computed by formula \eqref{eq:iso_type_U} and the signatures are given by formula \eqref{eq:signatures}.

\begin{example}\label{ex:complexification_not_fine}
Up to equivalence, the fine gradings on $M_9(\CC)$ with a second kind involution of signature $1$ or $3$ are the following:
\begin{enumerate}
\item[$\boxed{T=\ZZ_2}$] $\Gamma_{\cM}^{\mathrm{(II)}}(1;\CC;q,s,\ul{d})$, where $q\in\{1,3,5,7,9\}$, $s=\frac{9-q}{2}$, and the $q$-tuple $\ul{d}$ consists of $\pm 1$, with the number of 1's being $\frac{q+1}{2}$ for signature $1$ and $\frac{q+3}{2}$ for signature $3$ (in the latter case $q\ne 1$); the universal group is $\ZZ_2^q\times\ZZ^s$. 

\item[$\boxed{T=\ZZ_3^2}$] $\Gamma_{\cM}^{\mathrm{(I)}}(3;\CC;q,s,\ul{d})$, where $q\in\{1,3\}$, $s=\frac{3-q}{2}$, and the $q$-tuple $\ul{d}$ consists of $\pm 1$, with the number of 1's $\frac{q+1}{2}$ for signature $1$ and $3$ for signature $3$ (in the latter case $q=3$); the universal group is $\ZZ_3^2\times\ZZ_2^{q-1}\times\ZZ^s$. 

\item[$\boxed{T=\ZZ_3^4}$] $\Gamma_{\cM}^{\mathrm{(I)}}(3,3;\CC;1,0,(1))$ for signature $1$; the universal group is $\ZZ_3^4$.

\item[$\boxed{T=\ZZ_9^2}$] $\Gamma_{\cM}^{\mathrm{(I)}}(9;\CC;1,0,(1))$ for signature $1$; the universal group is $\ZZ_9^2$.
\end{enumerate}
The complexification of the graded algebra with involution $\cM^{\mathrm{(I)}}(3;\CC;q,s,\ul{d})$ is isomorphic to $(\cS\times\cS^\op,\ex)$ where the complex algebra $\cS=M_3(\CC)\otimes_\CC\cD(3;\CC)$ has a grading by $\ZZ_2^2\times\ZZ_3^2$ or $\ZZ\times\ZZ_3^2$, which can be refined by splitting its $2$-dimensional components to obtain the fine grading $\Gamma_{\cM}\bigl(\cD(3;\CC),3\bigr)$ by $\ZZ^2\times\ZZ_3^2$. It follows from the transfer results discussed in Section \ref{se:Lie} that the same remarks apply to the restrictions of these gradings to the real Lie algebras $\frsu(5,4)$ and $\frsu(6,3)$ and their complexification $\frsl_9(\CC)$.
\end{example}

\begin{example}\label{ex:M8}
Up to equivalence, the fine gradings on the real forms of $M_8(\CC)$ with an orthogonal involution are the gradings 
$\Gamma_\cM(2m;\Delta;q,s,\ul{d},+1)$ with $\Delta\in\{\RR,\HH\}$, $2^m(q+2s)=8$, and the degrees of $d_1,\ldots,d_q$ forming a multiset $\Sigma$ (as in Definition \ref{df:simpler_multiset} and part (1) of Corollary \ref{cor:fine_real}) indicated below, where in each case we take representatives of the orbits of multisets of size $q$ in $T_+$ for $\Delta=\RR$ and in $T_-$ for $\Delta=\HH$.
The signature of the involution is defined only for $\Delta=\RR$.
\begin{enumerate}

\item[$\boxed{T=\{e\}}$] $m=0$, $\Delta=\RR$, $q\in\{0,2,4,6,8\}$, $s=\frac{8-q}{2}$, and the $q$-tuple $\ul{d}$ consists of real numbers, so the signature can be $0,2,\ldots,q$; the universal group is $\ZZ^4$ for $q=0$ and $\ZZ_2^{q-1}\times\ZZ^s$ for $q>0$. 

\item[$\boxed{T=\ZZ_2^2}$] $m=1$, $\Delta\in\{\RR,\HH\}$, $q\in\{0,2,4\}$, $s=\frac{4-q}{2}$, and we have the following representatives of the orbits of multisets, where for the elements of $T$ we use the notation $a=(\bar{1},\bar{0})$, $b=(\bar{0},\bar{1})$ and $c=(\bar{1},\bar{1})$:
\begin{itemize}
\item $\cD\simeq M_2(\RR)$, $T_+=\{e,a,b\}$, $\Ort_+(2,2)\simeq S_2$ permutes $a$ and $b$, so 
\begin{itemize}
\item $\Sigma=\{\}$: signature $0$; universal group $\ZZ_2^2\times\ZZ^2$;
\item $\Sigma=\{a,a\}$: signature $0$; universal group $\ZZ_2^3\times\ZZ$;
\item $\Sigma=\{a,b\}$: signature $0$; universal group $\ZZ_2\times\ZZ_4\times\ZZ$;
\item $\Sigma=\{e,a\}$: signature $2$; universal group $\ZZ_2\times\ZZ_4\times\ZZ$;
\item $\Sigma=\{e,e\}$: signature $0$ or $4$ (depending on whether the real numbers $d_1$ and $d_2$ have the same or opposite signs); universal group $\ZZ_2^3\times\ZZ$;
\item $\Sigma=\{a,a,a,a\}$: signature $0$; universal group $\ZZ_2^5$;
\item $\Sigma=\{a,a,a,b\}$: signature $0$; universal group $\ZZ_2^3\times\ZZ_4$;
\item $\Sigma=\{a,a,b,b\}$: signature $0$; universal group $\ZZ_2^3\times\ZZ_4$;
\item $\Sigma=\{e,a,a,a\}$: signature $2$; universal group $\ZZ_2^3\times\ZZ_4$;
\item $\Sigma=\{e,a,a,b\}$: signature $2$; universal group $\ZZ_2\times\ZZ_4^2$;
\item $\Sigma=\{e,e,a,a\}$: signature $0$ or $4$; universal group $\ZZ_2^3\times\ZZ_4$;
\item $\Sigma=\{e,e,a,b\}$: signature $0$ or $4$; universal group $\ZZ_2\times\ZZ_4^2$;
\item $\Sigma=\{e,e,e,a\}$: signature $2$ or $6$; universal group $\ZZ_2^3\times\ZZ_4$;
\item $\Sigma=\{e,e,e,e\}$: signature $0$, $4$ or $8$; universal group $\ZZ_2^5$;
\end{itemize} 

\item $\cD\simeq \HH$, $T_-=\{a,b,c\}$, $\Ort_-(2,2)\simeq S_3$ permutes $a$, $b$ and $c$, so 
\begin{itemize}
\item $\Sigma=\{\}$: universal group $\ZZ_2^2\times\ZZ^2$;
\item $\Sigma=\{a,a\}$: universal group $\ZZ_2^3\times\ZZ$;
\item $\Sigma=\{a,b\}$: universal group $\ZZ_2\times\ZZ_4\times\ZZ$;
\item $\Sigma=\{a,a,a,a\}$: universal group $\ZZ_2^5$;
\item $\Sigma=\{a,a,a,b\}$: universal group $\ZZ_2^3\times\ZZ_4$;
\item $\Sigma=\{a,a,b,b\}$: universal group $\ZZ_2^3\times\ZZ_4$;
\item $\Sigma=\{a,a,b,c\}$: universal group $\ZZ_2\times\ZZ_4^2$.
\end{itemize}
\end{itemize}

\item[$\boxed{T=\ZZ_2^4}$] $m=2$, $\Delta\in\{\RR,\HH\}$, $q\in\{0,2\}$, $s=\frac{2-q}{2}$, and in the case $q=2$ we must have $\deg d_1\ne\deg d_2$ (otherwise the grading is not fine), so we have the following possibilities:
\begin{itemize}
\item $\cD\simeq M_4(\RR)$, $\Ort_+(4,2)$ acts transitively on the set $T_+\smallsetminus\{e\}$ and on pairs $(t_1,t_2)$ of distinct elements in this set that have a fixed value of $\beta(t_1,t_2)$ (by Witt's Extension Theorem), so 
\begin{itemize}
\item $\Sigma=\{\}$: signature $0$; universal group $\ZZ_2^4\times\ZZ$;
\item $\Sigma=\{t_1,t_2\}$ with $t_1\ne t_2$, $t_i\in T_+\smallsetminus\{e\}$, $\beta(t_1,t_2)=1$ or $\beta(t_1,t_2)=-1$: signature $0$; universal group $\ZZ_2^3\times\ZZ_4$;
\item $\Sigma=\{e,t\}$ with $t\in T_+\smallsetminus\{e\}$: signature $4$, universal group $\ZZ_2^3\times\ZZ_4$;  
\end{itemize}

\item $\cD\simeq M_2(\HH)$, $\Ort_-(4,2)$ acts transitively on $T_-$ and on pairs $(t_1,t_2)$ of distinct elements in $T_-$ that have a fixed value of $\beta(t_1,t_2)$, so
\begin{itemize}
\item $\Sigma=\{\}$: universal group $\ZZ_2^4\times\ZZ$;
\item $\Sigma=\{t_1,t_2\}$ with $t_1\ne t_2$, $t_i\in T_-$, $\beta(t_1,t_2)=1$ or $\beta(t_1,t_2)=-1$: universal group $\ZZ_2^3\times\ZZ_4$.
\end{itemize}
\end{itemize}

\item[$\boxed{T=\ZZ_2^6}$] $m=3$, $\Delta\in\{\RR,\HH\}$, $q=1$, $s=0$, so we have the following possibilities:
\begin{itemize}
\item $\cD\simeq M_8(\RR)$, $\Ort_+(6,2)$ acts transitively on $T_+\smallsetminus\{e\}$, so
\begin{itemize}
\item $\Sigma=\{t\}$ with $t\in T_+\smallsetminus\{e\}$: signature $0$, universal group $\ZZ_2^6$;
\item $\Sigma=\{e\}$: signature $8$, universal group $\ZZ_2^6$;
\end{itemize}

\item $\cD\simeq M_4(\HH)$, $\Ort_-(6,2)$ acts transitively on $T_-$, so 
\begin{itemize}
\item $\Sigma=\{t\}$ with $t\in T_-$: universal group $\ZZ_2^6$.
\end{itemize}

\end{itemize}

\end{enumerate}
Over $\CC$, there is no signature and fine gradings on simple algebras with an orthogonal or symplectic involution are classified (see Remark \ref{rm:compare_with_C}) by the orbits of multisets under a twisted action of $\Aut(T,\beta)$. This action depends on the arbitrary choice of $\varphi_0$ (as there is no distinguished involution), and if we choose the transposition on $\cD\simeq M_2(\CC)$, the group $\SP(2,2)\simeq S_3$ permutes $e$, $a$ and $b$ (see \cite[Example 3.44]{EKmon}). Therefore, in the list above, the gradings, for example, with multisets $\{a,a,b,b\}$ and $\{e,e,a,a\}$, have equivalent complexifications (regardless of signature).
\end{example}

Example \ref{ex:M8} is summarized in Table \ref{tb:M8}, where all gradings with equivalent complexifications appear in one row. Because of our intended application to the real forms of the Lie algebra $\frso_8(\CC)$, there is one more piece of information added: each grading is marked as inner or outer according to the corresponding action on $\frso_8(\CC)$ of the group of complex characters of the universal group. 

To see this distinction in terms of associative algebras with involution, one can look at the corresponding Clifford algebra. If $V$ is a finite-dimensional vector space over a field $\FF$ of characteristic different from $2$ and $b$ is a nondegenerate symmetric bilinear form on $V$, then $b$ allows us to define, on the one hand, an involution $\sigma$ on the algebra $\cA=\End_\FF(V)$ and, on the other hand, the Clifford algebra $\Cl(V,b)$, which is actually a superalgebra because it has a canonical $\ZZ_2$-grading. It turns out (see \cite[\S 8.B]{KMRT}) that one can go directly from $(\cA,\sigma)$ to the even part of the Clifford algebra. This construction of $\Cl(\cA,\sigma)$ directly from $(\cA,\sigma)$ has two important advantages: it works for any central simple algebra $\cA$ and any orthogonal involution $\sigma$, and it is functorial with respect to isomorphisms of algebras with involution. As a result, group actions and gradings on $(\cA,\sigma)$ pass to $\Cl(\cA,\sigma)$, which itself has a canonical involution induced by $\sigma$. 

The algebra $\Cl(\cA,\sigma)$ is always semisimple, and its center is $\FF$ if the degree of $\cA$ is odd and an \'etale algebra of dimension $2$ if the degree is even. In the latter case, we can sort $G$-gradings on $(\cA,\sigma)$ into Type I (or inner) and Type II (or outer) depending on whether the induced grading on the center of $\Cl(\cA,\sigma)$ is trivial or nontrivial. The support of this latter grading is generated by an element $h\in G$ of order at most $2$, which is not affected by extending scalars to the algebraic closure and, in the case of algebraically closed $\FF$, was computed in \cite[Lemma 33]{EK_Israel}. For our graded algebras with involution $\cM(\cD,\varphi_0,q,s,\ul{d},\delta)$, with even dimension and $\cD_e=\FF=Z(\cD)$, the center of the Clifford algebra is trivially graded if and only if one of the following holds: $|T|>4$, or $|T|=1$ and $q=0$, or $|T|=4$ and all three permissible values in $T$ (to make the involution orthogonal) have multiplicities  of the same parity in the multiset $\Sigma(\ul{d})$.

\begin{table}
\centering
\caption{Fine gradings and their universal groups for the real forms of $M_8(\CC)$ with an orthogonal involution, grouped in rows according to the equivalence class of their complexification. Notation: $M_{5+3}(\RR)$ indicates that the involution on $M_8(\RR)$ has signature $5-3=2$, etc., and the gray color marks inner gradings.}
\label{tb:M8}
\begin{tabular}{|l|c|c|c|c|c|c|}
\hline
& 
$M_{4+4}(\RR)$ & $M_{5+3}(\RR)$ & $M_{6+2}(\RR)$ & $M_{7+1}(\RR)$ & $M_{8+0}(\RR)$ & $M_{4}(\HH)$ \\
\hline
\rowcolor{lightgray}
$\ZZ^4$ &
$1$ &&&&& \\
\hline
$\ZZ_2\times\ZZ^3$ &
$1$ & $1$ &&&& \\
\hline
$\ZZ_2^3\times\ZZ^2$ &
$1$ & $1$ & $1$ &&& \\
\hline
$\ZZ_2^5\times\ZZ$ &
$1$ & $1$ & $1$ & $1$ && \\
\hline
$\ZZ_2^7$ &
$1$ & $1$ & $1$ & $1$ & $1$ & \\
\hline
\rowcolor{lightgray}
$\ZZ_2^2\times\ZZ^2$ &
$1$ &     &     &     &     & $1$ \\
\hline
\rowcolor{lightgray}
$\ZZ_2^3\times\ZZ$ &
$2$ &     & $1$ &     &     & $1$ \\
\hline
$\ZZ_2\times\ZZ_4\times\ZZ$ &
$1$ & $1$ &     &     &     & $1$ \\
\hline
\rowcolor{lightgray}
$\ZZ_2^5$ &
$2$ &     & $1$ &     & $1$ & $1$ \\
\hline 
\rowcolor{lightgray}
$\ZZ_2^3\times\ZZ_4$ &
$2$ &     & $1$ &     &     & $1$ \\
\hline 
$\ZZ_2^3\times\ZZ_4$ &
$1$ & $2$ &     & $1$ &     & $1$ \\
\hline 
$\ZZ_2\times\ZZ_4^2$ &
$1$ & $1$ & $1$ &     &     & $1$ \\
\hline 
\rowcolor{lightgray}
$\ZZ_2^4\times\ZZ$ &
$1$ &     &     &     &     & $1$ \\
\hline
\rowcolor{lightgray}
$\ZZ_2^3\times\ZZ_4$ &
$2$ &     & $1$ &     &     & $2$ \\
\hline
\rowcolor{lightgray}
$\ZZ_2^6$ &
$1$ &     &     &     & $1$ & $1$ \\
\hline
\end{tabular}

\end{table}

\subsection{Proof of Proposition \ref{pr:flag}}\label{sse:flag}

Given a finite abelian $2$-group $M$ (the $2$-primary component of $T$) and a nondegenerate alternating bicharacter 
$\beta:M\times M\to\CC^\times$, the group $\Aut(M,\beta)$ acts on the $GF(2)$-vector space $V\bydef M/M^{[2]}$, and we are interested in the image of $\Aut(M,\beta)$ in $\GL(V)$. 

We will view $M$ as a $\ZZ$-module, so $M^{[2]}=2M$ is the radical of $M$. Consider the ring homomorphism $\pi:\End(M)\to\End(V)$ that sends $r\in R\bydef\End(M)$ to $\bar{r}\in\End(V)$ defined by $\bar{r}(x+2M)=r(x)+2M$. Then 
\[
K\bydef\ker\pi=\{r\in R\mid rM\subset 2M\}
\] 
is a nilpotent ideal of $R$, as $K^i M\subset 2^i M$. Recall the flag $\cF=\{V_i\}_{i=0,1,\ldots}$ in $V$ from equation \eqref{eq:def_F}: $V_i\bydef M_{[2^i]}+2M/2M$; it is induced by the socle filtration $M_i=M_{[2^i]}$ of $M$. 
We claim that the image of $\pi$ is
\[
\End(V,\cF)\bydef\{s\in\End(V)\mid sV_i\subset V_i\text{ for all }i\}.
\]
Indeed, since each $M_i$ is an $R$-submodule of $M$, we get $\pi(R)\subset\End(V,\cF)$. To show the other inclusion, write $M$ as the direct product of cyclic subgroups $\langle c_j\rangle$ and note that the elements $v_j\bydef c_j+2M$ constitute a basis of $V$ adapted to the flag $\cF$: $V_i=\lspan{v_j\mid o(c_j)\le 2^i}$. Given $s\in\End(V,\cF)$, we have $s(v_j)=\sum_i s_{ij}v_i$ where $s_{ij}\in\{0,1\}$ and $s_{ij}=0$ whenever $o(c_i)>o(c_j)$. It follows that the mapping $c_j\mapsto\sum_i s_{ij}c_i$ extends to an endomorphism $r$ of $M$, which satisfies $\pi(r)=s$.

To make computations, we will represent the elements of $M$ by integer vectors and elements of $R=\End(M)$ by integer matrices. It will be convenient to divide them into blocks as follows. Let $2^{m_1}<2^{m_2}<\ldots<2^{m_\ell}$ be the orders of the cyclic subgroups appearing in the decomposition of $M$, so $\ell=\mathrm{length}(\cF)$ and $2^{m_\ell}$ is the exponent of $M$. We will label the generators $c_j$ so that the first $n_1$ of them have order $2^{m_1}$, the next $n_2$ have order $2^{m_2}$, and so on. Then, with respect to these generators, the elements of $M$ are represented by block vectors of the form $x=(x_i)_{1\le i\le\ell}$ where $x_i\in\ZZ^{n_i}$, and the elements of $R$ are represented by block matrices of the form $A=(a_{ij})_{1\le i,j\le\ell}$ where $a_{ij}\in M_{n_i\times n_j}(\ZZ)$ and $a_{ij}\equiv 0\pmod{2^{m_i-m_j}}$ for $i>j$. Note that $x$ and $x'$ represent the same element of $M$ if and only if $x_i\equiv x'_i\pmod{2^{m_i}}$ for all $i$ and, consequently, $A$ and $A'$ represent the same element of $R$ if and only if $a_{ij}\equiv a'_{ij}\pmod{2^{m_i}}$ for all $i,j$. 

The corresponding basis $\{v_j\}_j$ of $V$ is then also partitioned into blocks, each giving a basis of one of the nontrivial quotients $W_{m_i}=V_{m_i}/V_{m_{i-1}}$ of $\cF$ ($1\le i\le\ell$, with $m_0\bydef 0$). If we represent the elements of $\End(V)$ by matrices with respect to $\{v_j\}_j$, then the homomorphism $\pi$ is simply the reduction of entries mod $2$. Note that, for any $r\in R$, $\pi(r)$ is represented by an upper block-triangular matrix, whose diagonal blocks represent the operators induced by $\pi(r)$ on $W_{m_i}$. In particular, $\pi$ maps the multiplicative group $\Aut(M)$ of $R$ onto the stabilizer of $\cF$ in $\GL(V)$. 

To take into account $\beta$, consider the involution $\sigma$ that $\beta$ defines on the ring $R$: $\beta(ru,v)=\beta(u,\sigma(r)v)$ for all $u,v\in M$ and $r\in R$. Then 
\[
\Aut(M,\beta)=\{r\in\Aut(M)\mid \sigma(r)=r^{-1}\}=\{r\in R\mid \sigma(r)r=1\}.
\]
We will use a symplectic basis of $(M,\beta)$ as the generating set $\{c_j\}_j$. Ordering these elements appropriately and recalling Remark \ref{rm:change_roots}, we may assume that $\beta(u,v)=e^{2\pi\bi\langle x,y\rangle/2^{m_\ell}}$ where $x,y$ are vectors representing $u,v\in M$ and 
\[
\langle x,y \rangle=x^T B y\;\text{ where }B=\diag\bigl(2^{m_\ell-m_1}J_{n_1},\ldots,
2^{m_\ell-m_{\ell-1}} J_{n_{\ell-1}},J_{n_\ell}\bigr)
\]
and $J_n\bydef\left[\begin{smallmatrix} 0 & I_{n/2} \\ -I_{n/2} & 0\end{smallmatrix}\right]$ for any even $n$ (so $J_n^{-1}=J_n^T=-J_n$). In particular, if $u,v\in M_{[2^{m_i}]}$ then $\langle x,y\rangle\equiv 2^{m_\ell-m_i}x_i^T J_{n_i} y_i\pmod{2^{m_\ell-m_i+1}}$. So, as expected, the $i$-th block of the basis $\{v_j\}_j$ of $V$ gives a symplectic basis of $W_{m_i}$ with respect to the form $\langle\cdot,\cdot\rangle_{m_i}$ induced by $\beta^{2^{m_i-1}}$ on $W_{m_i}$.

Hence, if $r\in R$ is represented by matrix $A$ and $\sigma(r)$ by $A'$ then $A^T B\equiv BA'\pmod{2^{m_\ell}}$. For the $(i,j)$-th block, this gives $a'_{ij}\equiv 2^{m_i-m_j}a_{ji}^*\pmod{2^{m_i}}$, where 
\begin{equation}\label{eq:standard_symplectic}
C^*\bydef J_n^{-1}C^T J_m\;\text{ for any $m\times n$ matrix $C$ with even $m$ and $n$}.
\end{equation}
Note that, for $i<j$, we have $a_{ji}\equiv 0\pmod{2^{m_j-m_i}}$, hence 
$a_{ji}^*\equiv 0\pmod{2^{m_j-m_i}}$ as well, so $2^{m_i-m_j}a_{ji}^*\in\ZZ$.

As $a_{ij}\equiv 0\pmod{2^{m_i-m_j}}$ for $i>j$, any $r\in R$ is a sum of elements of the form 
\[
\tilde{e}_{ij}(a)\bydef\begin{cases}
e_{ij}(a) & \text{if }i\le j, \\
e_{ij}\bigl(2^{m_i-m_j}a\bigr) & \text{if }i>j,
\end{cases}
\]
where $a\in M_{n_i\times n_j}(\ZZ)$ and $e_{ij}(b)$ is the element represented by the matrix whose $(i,j)$-th block is $b$ and all other blocks are $0$. Using this notation, we get
\begin{equation}\label{eq:weird_symplectic}
\sigma(\tilde{e}_{ij}(a))=\tilde{e}_{ji}(a^*)\;\text{ for any }a\in M_{n_i\times n_j}(\ZZ).
\end{equation}
In particular, if $r=\sum_{i,j}\tilde{e}_{ij}(a_{ij})$ belongs to $\Aut(M,\beta)$ then $a_{ii}^*a_{ii}\equiv 1\pmod{2}$, i.e., $a_{ii}$ mod $2$ belongs to $\SP(n_i,2)$ for all $i$, which means that $\pi(r)$ belongs to the stabilizer $\SP(V,\cF)$ of $\cF$ and all symplectic forms $\langle\cdot,\cdot\rangle_{m_i}$, as expected. We are now going to prove that $\pi$ maps $\Aut(M,\beta)$ onto $\SP(V,\cF)$.

First observe that $\SP(V,\cF)=U\rtimes S$ where $U$ is the subgroup of all operators of the form $1+z$ with $z\in\End(V)$ represented by a strictly upper block-triangular matrix and $S$ is the subgroup of all operators represented by block-diagonal matrices consisting of symplectic blocks. To prove that both $S$ and $U$ are contained in the image of $\Aut(M,\beta)$, we will need the following lemma. But first we introduce some notation: if $\sigma$ is an involution on a ring $A$, we write
\[
\operatorname{Symd}(A)=\operatorname{Symd}(A,\sigma)\bydef(\id_A+\sigma)(A)=\{a+\sigma(a)\mid a\in A\}.
\] 
This is clearly a subgroup of the additive group $\operatorname{Sym}(A,\sigma)$ of symmetric elements. It is also easy to check that if $b\in\operatorname{Symd}(A)$ then $\sigma(a)ba\in\operatorname{Symd}(A)$ for any $a\in A$.

\begin{lemma}\label{lm:symd}
Let $\KK$ be a commutative unital ring and $\sigma$ be the standard symplectic involution on $A=M_n(\KK)$, i.e., $\sigma(a)=a^*$ where $a^*$ is given by equation \eqref{eq:standard_symplectic}. 
\begin{enumerate}
\item[(i)] $1\in\operatorname{Symd}(A)$.
\item[(ii)] The $\KK$-module $\operatorname{Symd}(A)$ is free and has a free complement in $A$.
\item[(iii)] For any even $m$ and $a\in M_{m\times n}(\KK)$, we have $a^*a\in\operatorname{Symd}(A)$.
\item[(iv)] For any $I\triangleleft\KK$, we have $\operatorname{Symd}(A)\cap IA=\operatorname{Symd}(IA)$.
\item[(v)] For any $I\triangleleft\KK$ and $c\in 1+\operatorname{Symd}(IA)$, there exists $u\in 1+IA$ such that $u^*cu\in 1+\operatorname{Symd}(I^2 A)$.
\end{enumerate}
\end{lemma}

\begin{proof}
For $a=\left[\begin{smallmatrix}\alpha & \beta \\ \gamma & \delta\end{smallmatrix}\right]$ with $\alpha,\beta,\gamma,\delta\in M_{n/2}(\KK)$, we have $a^*=\left[\begin{smallmatrix}\delta^T & -\beta^T \\ -\gamma^T & \alpha^T\end{smallmatrix}\right]$, so
\[
\operatorname{Symd}(A)=\left\{\left[\begin{smallmatrix}\alpha & \beta \\ \gamma & \alpha^T\end{smallmatrix}\right]
\mid\text{$\beta$ and $\gamma$ are skew-symmetric with zero diagonal}\right\}.
\]
Therefore, $\operatorname{Symd}(A)$ contains the identity matrix and has a $\KK$-basis consisting of certain sums and differences of matrix units. Also, the $\KK$-submodule 
\[
\left\{\left[\begin{smallmatrix}\alpha & \beta \\ \gamma & 0\end{smallmatrix}\right]
\mid\text{$\beta$ and $\gamma$ are upper triangular}\right\}
\]
is a complement of $\operatorname{Symd}(A)$ and has a basis consisting of matrix units. 
This completes parts (i) and (ii). Since for any $b\in\operatorname{Symd}\bigl(M_m(\KK)\bigr)$ and $a\in M_{m\times n}(\KK)$, we have $a^*ba\in\operatorname{Symd}(A)$, part (iii) follows from (i).

To prove part (iv), observe that, since $\sigma$ is $\KK$-linear, we have $\operatorname{Symd}(IA)=I\operatorname{Symd}(A)$. On the other hand, $\operatorname{Symd}(A)\cap IA=I\operatorname{Symd}(A)$, too, because the $\KK$-module $\operatorname{Symd}(A)$ is a direct summand of $A$.

Finally, in part (v), there exists $x\in IA$ such that $c=1+x+x^*$. Letting $u=1-x$, we get
\[
u^*cu=(1-x^*)(1+x+x^*)(1-x)=1-x^*x-x^2-(x^*)^2+x^*(x+x^*)x,
\]
hence $u^*cu-1\in\operatorname{Symd}(A)\cap I^2 A=\operatorname{Symd}(I^2 A)$.
\end{proof}

\begin{corollary}\label{cor:symd1}
Let $I$ be a nilpotent ideal in a commutative unital ring $\KK$. Then, for any even $n$, the ring homomorphism $M_n(\KK)\to M_n(\KK/I)$ of entry-wise reduction mod $I$ maps $\SP(n,\KK)$ onto $\SP(n,\KK/I)$.
\end{corollary}

\begin{proof}
As in Lemma \ref{lm:symd}, consider the ring $A=M_n(\KK)$ with the standard symplectic involution $*$. The kernel of the reduction homomorphism is $M_n(I)=IA$.  
Given an element $\bar{a}\in\SP(n,\KK/I)$, pick $a_0\in A$ such that $\bar{a}=a_0+IA$. 
Then $a_0^* a_0-1\in\operatorname{Symd}(A)\cap IA=\operatorname{Symd}(IA)$ by Lemma \ref{lm:symd}. 
Iterating part (v) of the said lemma, we can construct a sequence of elements $\{a_k\}_{k=0,1,\ldots}$ in $A$ such that $a_k^* a_k\in 1+\operatorname{Symd}\bigl(I^{2^k}A\bigr)$ and $a_k-a_0\in IA$. 
Indeed, if we already found $a_k$, then there exists $u_k\in 1+I^{2^k}A$ such that $u_k^* (a_k^* a_k) u_k\in 1+\operatorname{Symd}\bigl((I^{2^k})^2 A\bigr)$, so we can take $a_{k+1}=a_k u_k$. Since $I$ is nilpotent, we get $a_k^* a_k=1$ for sufficiently large $k$, so $a_k$ is a preimage of $\bar{a}$ in $\SP(n,\KK)$. 
\end{proof}

In particular, the reduction mod $2$ maps $\SP(n,\ZZ/2^{m}\ZZ)$ onto $\SP(n,2)$. This implies that the subgroup $S$ of $\SP(V,\cF)$ is contained in the image of $\Aut(M,\beta)$. Indeed, if $s=\sum_i e_{ii}(\bar{a}_i)$ with $\bar{a}_i\in\SP(n_i,2)$ then, for each $i$, there exists $a_i\in M_{n_i}(\ZZ)$ such that $a_i^* a_i\equiv 1\pmod{2^{m_i}}$ and $a_i\mapsto \bar{a}_i$ under the reduction mod $2$. By equation \eqref{eq:weird_symplectic}, the element $r=\sum_i\tilde{e}_{ii}(a_i)=\sum_i e_{ii}(a_i)\in R$ satisfies $\sigma(r)r=1$, so $r\in\Aut(M,\beta)$. By construction, $\pi(r)=s$.

To prove that $U$ is contained in the image of $\Aut(M,\beta)$, we will use the following: 
\begin{corollary}\label{cor:symd2}
Let $I$ be a nilpotent ideal in a commutative unital ring $\KK$. Then, for any even $n$, the mapping $M_n(\KK)\to M_n(\KK)$, $a\mapsto a^*a$, with $*$ being the standard symplectic involution, maps $1+I M_n(\KK)$ onto $1+\operatorname{Symd}\bigl(I M_n(\KK)\bigr)$.
\end{corollary}

\begin{proof}
Let $A=M_n(\KK)$. For any $x\in IA$, we have $(1+x)^*(1+x)=1+x+x^*+x^*x\in 1+\operatorname{Symd}(IA)$ by Lemma \ref{lm:symd}. Given $c\in 1+\operatorname{Symd}(IA)$, we iterate part (v) of the said lemma to construct a sequence of elements $\{u_k\}_{k=1,2,\ldots}$ in $1+IA$ such that $u_k^* c u_k\in 1+\operatorname{Symd}(I^{2^k}A)$. For sufficiently large $k$, we get $u_k^* c u_k=1$ and hence $c=a^*a$ for $a=u_k^{-1}$.
\end{proof}

Since $U$ is generated by the elements of the form $u_{ij}(\bar{a})\bydef 1+e_{ij}(\bar{a})$ where $i<j$ and $\bar{a}\in M_{n_i\times n_j}(\ZZ/2\ZZ)$, it is sufficient to find a preimage for each such element. As the initial approximation, we pick any $a\in M_{n_i\times n_j}(\ZZ)$ such that $a\mapsto\bar{a}$ under the reduction mod $2$ and take $r_0=1+\tilde{e}_{ij}(a)-\tilde{e}_{ji}(a^*)$. Since $\tilde{e}_{ji}(a^*)=e_{ji}(2^m a^*)$ with $m\bydef m_j-m_i>0$, we have $\pi(r_0)=u_{ij}(\bar{a})$. By equation \eqref{eq:weird_symplectic}, we get:
\[
\sigma(r_0)r_0=\bigl(1+\tilde{e}_{ij}(a)-\tilde{e}_{ji}(a^*)\bigr)\bigl(1-\tilde{e}_{ij}(a)+\tilde{e}_{ji}(a^*)\bigr)
=1+e_{ii}(b)+e_{jj}(c),
\] 
where $b=2^m aa^*\in\operatorname{Symd}\bigl(2^m M_{n_i}(\ZZ)\bigr)$ and $c=2^m a^*a\in\operatorname{Symd}\bigl(2^m M_{n_j}(\ZZ)\bigr)$ by Lemma \ref{lm:symd}. Applying Corollary \ref{cor:symd2} to $M_{n_i}(\ZZ/2^{m_i}\ZZ)$ and $M_{n_j}(\ZZ/2^{m_j}\ZZ)$, we can find $x\in 2^{m} M_{n_i}(\ZZ)$ and $y\in 2^{m} M_{n_j}(\ZZ)$ such that $(1+x)^*(1+b)(1+x)\equiv 1\pmod{2^{m_i}}$ and $(1+y)^*(1+c)(1+y)\equiv 1\pmod{2^{m_j}}$.
Let $r=r_0\bigl(1+e_{ii}(x)+e_{jj}(y)\bigr)$. Then $\pi(r)=\pi(r_0)=u_{ij}(\bar{a})$ and, by equation \eqref{eq:weird_symplectic}, we get: 
\[
\begin{split}
\sigma(r)r
&=\bigl(1+e_{ii}(x^*)+e_{jj}(y^*)\bigr)\bigl(1+e_{ii}(b)+e_{jj}(c)\bigr)\bigl(1+e_{ii}(x)+e_{jj}(y)\bigr) \\
&=e_{ii}\bigl((1+x)^*(1+b)(1+x)\bigr)+e_{jj}\bigl((1+y)^*(1+c)(1+y)\bigr)+\sum_{k\ne i,j}e_{kk}(1)=1.
\end{split}
\]
The proof of Proposition \ref{pr:flag} is complete.\qed

\section{Gradings on classical simple Lie algebras}\label{se:Lie}

Since the support of a grading by any group on a simple Lie algebra generates an abelian subgroup, \emph{we will assume in this section that all grading groups are abelian}. Under this assumption, gradings can be transferred from one finite-dimensional algebra (with any number of multilinear operations) to another (possibly with a different set of operations) via any homomorphism of their automorphism group schemes. First we will review this technique in general terms and also how it applies to transfer the classification of gradings (all gradings up to isomorphism or fine gradings up to equivalence) between a finite-dimensional associative algebra with involution $(\cR,\varphi)$ that is central simple as an algebra with involution and the corresponding central simple Lie algebra $\cL$ coming from $\Skew(\cR,\varphi)$, namely, the quotient of the commutator subalgebra by the center. This was used in \cite{BKR_Lie} to classify all gradings on real forms of classical Lie algebras (except $D_4$) up to isomorphism, and here we will use it to obtain a classification of fine gradings up to equivalence as an immediate consequence of our results in the previous sections. 

\subsection{Transfer of gradings by abelian groups}\label{sse:transfer}

We follow the functorial approach to affine group schemes as presented in 
\cite[Appendix A]{EKmon}, \cite[Chapter VI]{KMRT}, or \cite{Waterhouse}. 

An affine group scheme over a field $\FF$ is a representable group-valued functor 
$\Gs$ defined on 
the category $\Alg_\FF$ of unital, commutative, associative algebras over $\FF$. 
For example, given a finite-dimensional algebra $\cA$ over $\FF$, which has any number of 
multilinear operations, the \emph{automorphism group scheme} $\AAut(\cA)$ assigns to any unital,
commutative, associative $\FF$-algebra $R$ the group of automorphisms of the $R$-algebra 
$\cA_R\bydef \cA\otimes_\FF R$. To any homomorphism of $\FF$-algebras 
$f:R\rightarrow S$, the functor $\AAut(\cA)$ assigns the group homomorphism
given by the extension of scalars from $R$ to $S$ via $f$. Picking a basis of $\cA$ and thinking of $R$-linear endomorphisms of $\cA_R$ as matrices, one can express the condition of being an $R$-algebra automorphism by polynomial equations, which implies that the functor $\AAut(\cA)$ is representable.

By Yoneda's Lemma, the representing object $\cH$ of an affine group scheme $\Gs$ over $\FF$ becomes a Hopf algebra, so there is a natural isomorphism $\Gs\equiv \Hom_{\Alg_\FF}(\cH,.)$, and the multiplication in $\Gs(R)$, which is called \emph{the group of $R$-points of $\Gs$}, is determined by the comultiplication $\Delta$ of $\cH$. In this way, the categories of affine group schemes and of commutative Hopf algebras over $\FF$ are antiequivalent.

\smallskip

Given an abelian group $G$, any $G$-grading $\Gamma:\cA=\bigoplus_{g\in G}\cA_g$  
determines a homomorphism of affine group schemes
\[
\eta_\Gamma:G^D\longrightarrow \AAut(\cA)
\]
where $G^D$ is the \emph{diagonalizable group scheme} represented by the group
algebra $\FF G$, with its natural structure of Hopf algebra defined by  
$\Delta(g)=g\otimes g$ for all $g\in G$. The homomorphism $\eta_\Gamma$ is 
determined as follows. For any $R$ in $\Alg_\FF$, the corresponding group homomorphism  
$(\eta_\Gamma)_R:\Hom_{\Alg_\FF}(\FF G,R)\rightarrow \Aut_R(\cA\otimes_\FF R)$
is given by
\[
(\eta_\Gamma)_R(f)(x\otimes r)=x\otimes f(g)r
\]
for all $g\in G$, $x\in\cA_g$, $r\in R$, and $f\in\Hom_{\Alg_\FF}(\FF G,R)$.

Conversely, a homomorphism $\eta:G^D\rightarrow \AAut(\cA)$ defines 
a \emph{generic automorphism} $\rho$ of $\cA\otimes_\FF\FF G$, as an algebra over $\FF G$,
given by the image under $\eta$ of the identity map 
$\id_{\FF G}\in G^D(\FF G)=\Hom_{\Alg_\FF}(\FF G,\FF G)$.  
Then $\cA$ becomes $G$-graded:  $\cA=\bigoplus_{g\in G}\cA_g$ where
\[
\cA_g=\{a\in\cA\mid  \rho(a\otimes 1)=a\otimes g\}.
\]
(The set of `eigenvectors' in $\cA$ with `eigenvalue' $g$ for the generic automorphism.)

\smallskip

Now consider the $H$-grading ${}^\alpha\Gamma$ on $\cA$ induced by a homomorphism of abelian groups $\alpha: G\rightarrow H$. We have a Hopf algebra homomorphism $\FF G\rightarrow \FF H$ defined by $g\mapsto\alpha(g)$ for all $g\in G$, which we also denote by $\alpha$, and hence a homomorphism of affine group schemes $\alpha^D:H^D\rightarrow G^D$. 
The composition $\eta_\Gamma\circ\alpha^D:H^D\rightarrow \AAut(\cA)$ is precisely 
$\eta_{{}^\alpha\Gamma}$.

The \emph{diagonal group scheme} $\Diags(\Gamma)$ is the diagonalizable group scheme whose group of $R$-points, 
for any $R$ in $\Alg_\FF$, consists of those automorphisms of the scalar extension $\cA_R=\cA\otimes_\FF R$ that act by a scalar on each homogeneous component:
\[
\Diags(\Gamma)(R)\bydef\{f\in\Aut_R(\cA_R)\mid
f|_{\cA_g\otimes_\FF R}\in R^\times\,\id_{\cA_g\otimes_\FF R}\}.
\]
Up to isomorphism, $\Diags(\Gamma)$ is $U^D$ where $U$ is the universal group of 
$\Gamma$ (see Subsection \ref{sse:universal}).
The morphism $\eta_\Gamma$ factors through $\Diags(\Gamma)$ and hence induces a homomorphism of Hopf algebras in the
reverse direction: $\FF U\rightarrow \FF G$, or equivalently a group homomorphism 
$\alpha:U\rightarrow G$, already considered in Subsection \ref{sse:universal}. 

As a consequence, the grading $\Gamma$ is fine if and only if $\Diags(\Gamma)$
is a maximal diagonalizable subgroupscheme of $\AAut(\cA)$.

\smallskip

The automorphism group of $\cA$ is the group of $\FF$-points 
$\Aut(\cA)=\AAut(\cA)(\FF)$, and it acts by conjugation on $\AAut(\cA)$. Namely,
$\psi\in\Aut(\cA)$ defines a homomorphism $\Ad_\psi:\AAut(\cA)\to\AAut(\cA)$
as follows:
$
f\mapsto \psi_R\circ f\circ\psi_R^{-1}
$
for all $R$ in $\Alg_\FF$ and $f\in \Aut_R(\cA_R)$, where $\psi_R\bydef\psi\otimes_\FF\id_R$.

\begin{proposition}[{\cite[Propositions 1.36 and 1.37]{EKmon}}]\label{pr:Isom_Equiv}
Let $\cA$ be a finite-dimensional algebra over a field $\FF$.
\begin{enumerate} 
\item Let $G$ be an abelian group and let $\Gamma$ and $\Gamma'$ be two $G$-gradings on $\cA$, with associated homomorphisms $\eta_\Gamma,\eta_{\Gamma'}:G^D\to \AAut(\cA)$. Then $\Gamma$ and $\Gamma'$ are isomorphic if and only if 
there is an automorphism $\psi\in\Aut(\cA)$ such that $\Ad_\psi\circ\eta_\Gamma=\eta_{\Gamma'}$.

\item Let $\Gamma$ and $\Gamma'$ be two fine abelian group gradings on $\cA$. 
Then $\Gamma$ and $\Gamma'$ are equivalent if and only if there is an automorphism
$\psi\in\Aut(\cA)$ such that $\Ad_\psi\bigl(\Diags(\Gamma)\bigr)=\Diags(\Gamma')$.
\end{enumerate}
\end{proposition}

\smallskip

Now, if $\cB$ is another finite-dimensional algebra, 
any homomorphism of affine group schemes
$\theta:\AAut(\cA)\to \AAut(\cB)$ sends gradings on $\cA$ to gradings on $\cB$.
More precisely, if $G$ is an abelian group and $\Gamma$ is a $G$-grading on
$\cA$ with associated homomorphism $\eta_\Gamma:G^D\to \AAut(\cA)$, then
the composition $\theta\circ\eta_\Gamma:G^D\to\AAut(\cB)$ defines a $G$-grading on $\cB$,
denoted by $\theta(\Gamma)$.

For any group homomorphism $\alpha:G\to H$,
it follows that $\theta({}^\alpha\Gamma)={}^\alpha\bigl(\theta(\Gamma)\bigr)$,
since $\theta\circ(\eta_\Gamma\circ\alpha^D)=(\theta\circ\eta_\Gamma)\circ\alpha^D$.
Moreover, \cite[Theorem 1.38]{EKmon} shows that if two $G$-gradings $\Gamma$ and
$\Gamma'$ on $\cA$ are isomorphic, or weakly isomorphic, so are the gradings
$\theta(\Gamma)$ and $\theta(\Gamma')$ on $\cB$. In particular, we obtain the
first part of the next result. The second part is proved in 
\cite[Theorem 1.39]{EKmon}.

\begin{theorem}\label{th:transfer}
Let $\cA$ and $\cB$ be two finite-dimensional algebras over a field $\FF$, and let
$\theta:\AAut(\cA)\to\AAut(\cB)$ be an isomorphism of affine group schemes.
\begin{enumerate} 
\item Let $G$ be an abelian group and let $\Gamma$ and
$\Gamma'$ be two $G$-gradings on $\cA$. Then $\Gamma$ and $\Gamma'$ 
are isomorphic if and only if so are $\theta(\Gamma)$ and $\theta(\Gamma')$. 

\item Let $\Gamma$ and $\Gamma'$ be two fine  abelian group gradings on $\cA$. 
Then $\Gamma$ and $\Gamma'$ are equivalent if and only so are $\theta(\Gamma)$ and $\theta(\Gamma')$.
\end{enumerate}
\end{theorem}

Therefore, the problems of classification of $G$-gradings up to isomorphism, or the 
classification of fine gradings up to equivalence, on $\cA$ and on $\cB$ are 
equivalent. This `transfer' property has been crucial in the works cited in the Introduction to
classify gradings on the classical simple Lie algebras of series $A$, $B$, $C$ and $D$ (except $D_4$) over an 
algebraically closed field of characteristic different from $2$ or over a real closed field, 
by transferring the problem to associative algebras with involution, as we will outline in the next subsection. 
It has been used, too, for types $G_2$ or $F_4$ (see \cite[Chapters 4 and 5]{EKmon}), 
as the automorphism group scheme of a simple Lie algebra of type $G_2$ (resp., $F_4$) 
over a field of characteristic different from $2$ and $3$ (resp., different from $2$)
is isomorphic to the automorphism group scheme of a Cayley (or octonion) algebra (resp., Albert algebra).

There is an analogous transfer property for \emph{twisted forms} of finite-dimensional algebras (with any number of multilinear operations), since they are also controlled by the automorphism group schemes, via Galois cohomology (see e.g. \cite{Waterhouse}). Recall that, given an algebra $\cA$, an algebra $\cB$ of the same kind (i.e., with the same number and arity of multilinear operations) is called a \emph{twisted form} of $\cA$ if $\cA$ and $\cB$ become isomorphic after extending scalars to an algebraic closure $\overline{\FF}$ of the ground field: $\cA_{\overline{\FF}}\simeq \cB_{\overline{\FF}}$, where $\cA_{\overline{\FF}}\bydef \cA\otimes_\FF \overline{\FF}$.

\subsection{Automorphism group schemes of classical simple Lie algebras}

Assume in this subsection that the characteristic of the ground field 
$\FF$ is not $2$.

Let $\cR$ be a finite-dimensional central simple associative algebra
over $\FF$, endowed with an involution $\varphi$ of the first kind. 
By definition, $(\cR,\varphi)$ is a twisted form of $\bigl(M_n(\FF),t\bigr)$ if $\varphi$ is orthogonal
and of $\bigl(M_n(\FF),t_s\bigr)$ if $\varphi$ is symplectic, where $t$ denotes the transposition $X\mapsto X^T$ and 
$t_s$ denotes the involution $X\to J^{-1}X^T J$, with $J=\left[\begin{smallmatrix} 0&I_{n/2}\\ -I_{n/2}&0\end{smallmatrix}\right]$. Recall that $n$ is called the degree of $\cR$ (the square root of its dimension); it must be even if $\varphi$ is symplectic.

Consider the Lie algebra $\cL\bydef\Skew(\cR,\varphi)$ of the skew-symmetric elements of $\cR$ relative to $\varphi$ and let $\theta:\AAut(\cR,\varphi)\rightarrow \AAut(\cL)$ be the homomorphism of affine group schemes obtained by restriction.

\begin{theorem}\label{th:BCD}
Let $\cR$ be a finite-dimensional central simple associative algebra
over a field $\FF$ of characteristic not $2$, endowed with an involution $\varphi$ of the first kind.
If $\varphi$ is orthogonal, assume the degree of $\cR$ is $\geq 5$ and $\neq 8$, and if $\varphi$ is symplectic, assume it is $\geq 4$. Then the homomorphism $\theta:\AAut(\cR,\varphi)\rightarrow \AAut(\cL)$ is an isomorphism.
\end{theorem}

\begin{proof}
The proof of \cite[Theorem 3.7]{EKmon} shows that $\theta_{\Falg}:\Aut_\Falg(\cR_{\Falg},\varphi_\Falg)
\rightarrow \Aut(\cL_\Falg)$ is a group  isomorphism and that 
$\dc\theta_{\Falg}:\Der(\cR_{\Falg},\varphi_\Falg)\rightarrow\Der(\cL_\Falg)$ is an isomorpism of Lie algebras. 
Besides, $\AAut(\cR,\varphi)$ is smooth, because so is $\AAut(\cR_{\Falg},\varphi_\Falg)$. From \cite[Theorem A.50]{EKmon} we conclude that $\theta$ is an isomorphism.
\end{proof}

It follows from Theorem \ref{th:BCD} that the twisted forms of 
the orthogonal Lie algebra $\frso_n(\FF)$ for $n\geq 5$, $n\neq 8$, i.e., the central simple Lie algebras of type $B_r$ ($r\geq 2$) or $D_r$ ($r\geq 3$, $r\neq 4$, with $D_3=A_3$) are, up to isomorphism, the
Lie algebras $\Skew(\cR,\varphi)$ for central simple associative algebra $\cR$ of degree $n\geq 5$, $n\neq 8$, endowed with an orthogonal involution $\varphi$. In the same vein, the twisted forms of $\frsp_n(\FF)$, i.e., the central simple Lie algebras of type $C_r$ ($r\geq 2$, with $C_2=B_2$) are, up to isomorphism, the Lie algebras $\Skew(\cR,\varphi)$ for central simple associative algebras $\cR$ of (even) degree $n\geq 4$, endowed with a symplectic involution $\varphi$.

\smallskip

For series $A$ we need involutions of the second kind. Let $(\cR,\varphi)$ be a finite-dimensional central simple algebra with such an involution and let $\KK=Z(\cR)$. Then $\KK$ is an \'etale algebra of dimension $2$ over $\FF$, and hence it is either isomorphic to $\FF\times\FF$ or a quadratic separable field extension of $\FF$. By restriction we get a quotient map $\pi:\AAut(\cR,\varphi)\rightarrow \AAut(\KK)$. Since $\chr{\FF}\ne 2$, this latter group scheme is isomorphic to $\bmu_2\simeq C_2$, which is smooth. Denote the kernel of
$\pi$ by $\AAut_\KK(\cR)$. Here $(\cR,\varphi)$ is a twisted form of
$\bigl(M_n(\FF)\times M_n(\FF)^{\op},\ex\bigr)$. For this latter algebra with involution,
$\KK=\FF\times\FF$ and $\ker\pi$ is naturally isomorphic to 
$\AAut\bigl(M_n(\FF)\bigr)\simeq \PGLs_n(\FF)$, which is smooth, too.
We conclude that $\AAut(\cR,\varphi)$ is smooth.

\begin{remark}\label{re:autos_anti}
In \cite[Chapter 3]{EKmon}, in order to deal with simple Lie algebras of series $A$ over an algebraically closed field, the group scheme $\AAntaut\bigl(M_n(\FF)\bigr)$ is constructed from automorphisms and
antiautomorphisms, and used instead of the automorphism group scheme $\AAut\bigl(M_n(\FF)\times M_n(\FF)^{\op},\ex\bigr)$, which we consider here. There is a homomorphism $\Theta:\overline{\AAut}\bigl(M_n(\FF)\bigr)\rightarrow 
\AAut\bigl(M_n(\FF)\times M_n(\FF)^{\op},\ex\bigr)$ that takes any automorphism $\psi$ of $M_n(R)$ (for any object $R$ in $\AlgF$) to the automorphism $(x,y)\to (\psi(x),\psi(y))$ of $\bigl(M_n(R)\times M_n(R)^\op,\ex\bigr)$, and any antiautomorphism $\psi$ of $M_n(R)$ to the automorphism $(x,y)\to (\psi(y),\psi(x))$. It is easy to check that $\Theta_\FF$ is a group isomorphism and $\dc\Theta$ is a Lie algebra isomorphism (the Lie algebras of both $\AAntaut\bigl(M_n(\FF)\bigr)$ and $\AAut\bigl(M_n(\FF)\times M_n(\FF)^{\op},\ex\bigr)$ are isomorphic to $\Der\bigl(M_n(\FF)\bigr)\simeq\frpgl_n(\FF)$). From \cite[Theorem A.50]{EKmon} it follows that $\Theta$ is an isomorphism.
\end{remark}

With $(\cR,\varphi)$ as above, consider the Lie algebras
$\cK=\Skew(\cR,\varphi)$ and $\cL=[\cK,\cK]/([\cK,\cK]\cap Z(\cR))$.
In the split case $\bigl(M_n(\FF)\times M_n(\FF)^{\op},\ex\bigr)$,
$\cK=\{(x,-x)\mid x\in M_n(\FF)\}$, which is isomorphic to 
$\frgl_n(\FF)$ by projecting onto the first component, so $\cL$ 
turns out to be isomorphic to the projective special linear Lie algebra 
$\frpsl_n(\FF)$. Let $\theta:\AAut(\cR,\varphi)\rightarrow \AAut(\cL)$ be the homomorphism obtained by restricting and passing modulo the center.

We call an algebra of the form $\cS\times\cS^{\op}$, with $\cS$ a
finite-dimensional central simple associative algebra, a \emph{central simple associative algebra} over $\FF\times\FF$. Its \emph{degree} is the degree of the algebra $\cS$ (the square root of its dimension).

\begin{theorem}\label{th:A}
Let $\KK$ be an \'etale quadratic algebra of dimension $2$ over a field 
$\FF$ of characteristic not $2$. Let $\cR$ be a finite-dimensional central simple algebra over $\KK$ of degree $n\geq 3$, endowed with an involution $\varphi$ of the second kind. Assume that the characteristic of $\FF$ is not $3$ if $n=3$. Then the homomorphism $\theta:\AAut(\cR,\varphi)\rightarrow \AAut(\cL)$ is an isomorphism.
\end{theorem}

\begin{proof}
The proof of \cite[Theorem 3.9]{EKmon}, together with Remark 
\ref{re:autos_anti}, shows that 
$\theta_{\Falg}:\Aut(\cR_{\Falg},\varphi_\Falg)\rightarrow \Aut(\cL_{\Falg})$ is a group isomorphism and that 
$\dc\theta_\Falg:\Der(\cR_\Falg,\varphi_\Falg)\rightarrow \Der(\cL_\Falg)$ is an isomorphism
of Lie algebras. Since $\AAut(\cR,\varphi)$ is smooth, we conclude from
\cite[Theorem A.50]{EKmon} that $\theta$ is an isomorphism.
\end{proof}

We note that the restriction $\chr{\FF}\ne 3$ if $n=3$ cannot be omitted (see e.g. \cite[\S 4.3]{EKmon}). 
It follows from Theorem \ref{th:A} that the twisted forms of $\frpsl_n(\FF)$ with $n\geq 3$ (except in the case $\chr{\FF}= 3=n$, see \cite{CE16}) are, up to isomorphism, the Lie algebras $\cL=[\cK,\cK]/([\cK,\cK]\cap Z(\cR))$ above. 

As for $\frpsl_2(\FF)=\frsl_2(\FF)$, the situation is simpler. The twisted forms are the Lie algebras $\cL=[\cQ,\cQ]$, where $\cQ$ is a central simple algebra of degree $2$, i.e., a quaternion algebra. The restriction
homomorphism $\AAut(\cQ)\rightarrow \AAut(\cL)$ is an isomorphism, as follows from the first part of \cite[Theorem 3.9]{EKmon}.

A grading $\Gamma$ by an abelian group $G$ on a classical simple Lie algebra $\cL$ can be \emph{inner} or \emph{outer} depending on whether or not the image of the corresponding homomorphism $\eta_\Gamma:G^D\to\AAut(\cL)$ is contained in the group scheme of inner automorphisms $\InAAut(\cL)$, which is the connected component of $\AAut(\cL)$. For type $A_r$  ($r\ge 2$), this is determined by the $G$-grading on $(\cR,\varphi)$ in Theorem \ref{th:A} given by the composition $\theta^{-1}\circ\eta_\Gamma$: $\Gamma$ is inner if the grading on $\cR$ is of Type~I (i.e., trivial on $\KK$) 
and outer if the grading on $\cR$ is of Type~II.


Finally, the situation for the central simple Lie algebras of type $D_4$,
i.e., the twisted forms of $\frso_8(\FF)$, is more involved due to the triality phenomenon. 
To any such Lie algebra $\cL$, one can attach a cubic \'etale algebra $\LL$ over $\FF$ (see e.g. \cite[\S 45.C]{KMRT}) and a quotient map $\pi:\AAut(\cL)\rightarrow\AAut(\LL)$, whose kernel is $\InAAut(\cL)$ (see \cite[\S 2]{EK_G2D4}), thus obtaining a short exact sequence
\[
1\longrightarrow \InAAut(\cL)\longrightarrow \AAut(\cL)
\overset{\pi}{\longrightarrow} \AAut(\LL)\longrightarrow 1
\]
Any such $\LL$ is a twisted form of $\FF\times\FF\times\FF$, hence $\AAut(\LL)$ is the \'etale group scheme  
obtained from a continuous action of the profinite group $\Gal(\Fsep/\FF)$ on the symmetric group $S_3$, where 
$\Fsep$ is a separable closure of $\FF$ (see e.g. \cite[Proposition~20.16]{KMRT}). 

If $\Gamma$ is a grading on $\cL$ by an abelian group $G$ and $\eta_\Gamma:G^D\rightarrow \AAut(\cL)$ is
the corresponding homomorphism, then the composition $\pi\circ\eta_\Gamma$ gives
a $G$-grading on $\LL$, whose support is a subgroup $H$ of $G$ corresponding to the image of $\pi\circ\eta_\Gamma$, which is a diagonalizable subgroupscheme of $\AAut(\LL)$. This subgroupscheme corresponds to a $\Gal(\Fsep/\FF)$-invariant abelian subgroup of $S_3$, so the order of $H$ is at most $3$. 
We will say that $\Gamma$ is of \emph{Type I, II, or III} according to the value of $|H|$.

We are particularly interested in the case where $\LL=\FF\times\KK$ for 
a quadratic \'etale algebra $\KK$, as this is the case that appears over the real numbers. For $\LL$ of this form, there is a central simple associative algebra $\cR$ of  degree $8$, endowed with an orthogonal involution $\varphi$ such that $\cL$ is isomorphic to $\Skew(\cR,\varphi)$ and $\KK$ to the center of the Clifford algebra $\Cl(\cR,\varphi)$ (see \cite[\S 43.C]{KMRT}).

\subsection{Classification of fine gradings on real forms of classical simple Lie algebras except $D_4$}

In view of the transfer discussed in the previous two subsections (see Theorems \ref{th:transfer}, \ref{th:BCD} and \ref{th:A}), the following results are an immediate consequence of Theorem \ref{th:fine_real} and Corollary \ref{cor:fine_real}, so we merely state them for convenience of future reference. We are using notation introduced in Theorem \ref{th:fine_real}. 

In each case, we will be restricting the grading from the graded associative algebra with involution $(\cR,\varphi)$ of the appropriate form $\cM(\cD,\varphi_0,q,s,\ul{d},\delta)$ in Definition~\ref{df:M} or $\cM^\ex(\cD,k)$ in Definition~\ref{df:Mex} to the derived subalgebra $\cL$ of the Lie algebra of skew-symmetric elements $\Skew(\cR,\varphi)$. Disregarding the grading, we have $\cR\simeq M_k(\cD)$, with $k=q+2s$ and $\varphi(X)=\Phi^{-1}\varphi_0(X^T)\Phi$, in the first case and $\cR\simeq M_k(\cD)\times M_k(\cD)^\op$, with the exchange involution, in the second case. 
The matrix $\Phi$, given by equation \eqref{eq:PhiMDpqsdd}, is the product of $\diag(d_1,\ldots,d_k)$ and the permutation matrix of the permutation $\sigma$ of $\{1,\ldots,k\}$ that fixes $1,\ldots,q$ and interchanges $q+2r-1$ and $q+2r$ for $1\le r\le s$, where $d_1,\ldots,d_q$ are the entries of the $q$-tuple $\ul{d}$ and, for $1\le r\le s$, we take $d_{q+2r-1}\bydef 1_\cD$ and $d_{q+2r}$ is either $1_\cD$ in the case of hermitian $\Phi$ or $-1_\cD$ in the case of skew-hermitian $\Phi$. 

The relevant graded-division algebras $\cD$ were introduced in Subsection \ref{sse:real_basics}. To compute the components of $\cL$, it will be convenient to fix a representative in each homogeneous component of $\cD$. This will also allow us to simplify the parameter $\ul{d}$ by choosing $d_1,\ldots,d_q$ from among these representatives, except that in the cases where the signature of $\ul{d}$ is defined (see Definition \ref{df:sign_multiset}) we will have to allow $-1_\cD$ as well as $+1_\cD$ for entries of degree $e$. We will always choose the number of $1$'s, denoted $n^0_+$, greater than or equal to the number of $(-1)$'s, denoted $n^0_-$, so $\mathrm{signature}(\ul{d})=n^0_+ - n^0_-$.

\subsubsection{Series A: inner gradings on special linear Lie algebras}

Consider the graded algebra with involution $(\cR,\varphi)=\cM^{\mathrm{(I)}}(2m;\Delta;k)$ where $\Delta\in\{\RR,\HH\}$. 
By definition, $\cR=\cS\times\cS^\op$, with exchange involution, where $\cS=\cM(\cD,k)$ with $\cD=\cD(2m;\pm 1)$.
The projection $\cR\to\cS$ gives an isomorphism of graded Lie algebras $\Skew(\cR,\varphi)\to\cS^{(-)}$, so it suffices to consider $\cS$.
Recall that $\cD(2m;+1)$ is the real graded-division algebra with support $T=\ZZ_2^{2m}$ and $1$-dimensional components spanned by the following elements:
\begin{equation}\label{eq:fix_Xt_real}
X_{(\wb{i_1},\wb{j_1},\ldots,\wb{i_m},\wb{j_m})}=X^{i_1} Y^{j_1}\otimes\cdots\otimes X^{i_m} Y^{j_m}\in M_2(\RR)\otimes\cdots\otimes M_2(\RR)\simeq M_{\ell}(\RR),
\end{equation}
where $X=\left[\begin{smallmatrix}-1 & 0 \\ 0 & 1\end{smallmatrix}\right]$, $Y=\left[\begin{smallmatrix} 0 & 1 \\ 1 & 0\end{smallmatrix}\right]$, and $\ell=2^m$.
The definition of $\cD(2m;-1)$ is the same except that the last factor $M_2(\RR)$ is replaced by $\HH$ and the matrices $X$ and $Y$ are replaced by the quaternions 
$\mathbf{i}$ and $\mathbf{j}$ in this last factor, so $\cD(2m;-1)\simeq M_\ell(\HH)$ where $\ell=2^{m-1}$. Hence we may identify $\cS$ with $M_k(\RR)\otimes\cD\simeq M_n(\Delta)$, where $n=k\ell$.
 
The grading of $\cS$ restricts to an inner grading on its derived Lie subalgebra 
$\cL$, 
which we denote by $\Gamma_{\mathfrak{sl}_n(\Delta)}^{\mathrm{(I)}}(m)$.
The universal group is $T\times(\ZZ^k)_0\simeq\ZZ_2^{2m}\times\ZZ^{k-1}$, where $(\ZZ^k)_0$ is the subgroup of $\ZZ^k$ generated by the differences of its standard basis elements $e_1,\ldots,e_k$, and the homogeneous components of the grading are given by
\begin{align*}
&\cL_e = \lspan{E_{1,1}-E_{2,2},\ldots,E_{k-1,k-1}-E_{k,k}}\otimes 1_\cD, \\
&\cL_t = \lspan{E_{1,1},E_{2,2},\ldots,E_{k,k}}\otimes X_t\;\text{ for }e\ne t\in T, \\
&\cL_{(t,e_i-e_j)} = \RR E_{i,j}\otimes X_t\;\text{ for }i\ne j\text{ and }t\in T,
\end{align*} 
so we have $\dim\cL_e=k-1$, $\dim\cL_t=k$ for $e\ne t\in T$ ($2^{2m}-1$ components), and the remaining $2^{2m}k(k-1)$ components of dimension $1$.

\begin{theorem}\label{th:AI_RH}
Any inner fine grading on a real special linear Lie algebra of type $A_r$ is equivalent to exactly one of the gradings 
$\Gamma_{\mathfrak{sl}_n(\Delta)}^{\mathrm{(I)}}(m)$ where either $\Delta=\RR$, $n=r+1$, $m\ge 0$ or $\Delta=\HH$, $2n=r+1$, $m\ge 1$, and in both cases $2^m$ is a divisor of $r+1$ and $2^{-m}(r+1)\ge 3$.\qed
\end{theorem}

\subsubsection{Series A: inner gradings on special unitary Lie algebras}

Consider $(\cR,\varphi)=\cM^{\mathrm{(I)}}(\ell_1,\ldots,\ell_m;\CC;q,s,\ul{d})$. Here $\cD=\cD(\ell_1,\ldots,\ell_m;\CC)$, which is the complex graded-division algebra with support $T=\ZZ_{\ell_1}^2\times\cdots\times\ZZ_{\ell_m}^2$ and components of complex dimension $1$ spanned by the following elements:
\begin{equation}\label{eq:sym_elements_bad_D}
\begin{split}
X_{(\wb{i_1},\wb{j_1},\ldots,\wb{i_m},\wb{j_m})} &= e^{-i_1 j_1 \pi\mathbf{i}/\ell_1} X_{\ell_1}^{i_1} Y_{\ell_1}^{j_1}\otimes\cdots\otimes e^{-i_m j_m \pi\mathbf{i}/\ell_m} X_{\ell_m}^{i_m} Y_{\ell_m}^{j_m} \\
&\in M_{\ell_1}(\CC)\otimes_\CC\cdots\otimes_\CC M_{\ell_m}(\CC)\simeq M_{\ell}(\CC),
\end{split}
\end{equation}
where $0\le i_r,j_r<\ell_r$ for $1\le r\le m$, $X_n\bydef\diag(e^{2\pi\mathbf{i}/n},e^{4\pi\mathbf{i}/n},\ldots,e^{2(n-1)\pi\mathbf{i}/n},1)$, $Y_n$ is the permutation matrix of the cycle $(1,\ldots,n)$, and $\ell=\ell_1\cdots\ell_m$. The commutation relations between the elements defined by equation \eqref{eq:sym_elements_bad_D} are given by the nondegenerate alternating bicharacter
\[
\beta\bigl((\wb{i_1},\wb{j_1},\ldots,\wb{i_m},\wb{j_m}),(\wb{i'_1},\wb{j'_1},\ldots,\wb{i'_m},\wb{j'_m})\bigr)
=e^{(i_1 j'_1-i'_1 j_1)2\pi\mathbf{i}/\ell_1}\cdots e^{(i_m j'_m-i'_m j_m)2\pi\mathbf{i}/\ell_m},
\]
and the `phase factors' in equation \eqref{eq:sym_elements_bad_D} are chosen so that all elements $X_t$ ($t\in T$) are symmetric with respect to the (second kind) involution $\varphi_0$ defined on $\cD$ (see equation \eqref{eq:inv_bad_D}).

We may identify $\cR$ with $M_k(\RR)\otimes\cD\simeq M_n(\CC)$, where $n=k\ell$. For a fixed $n$, the integer $q$ is determined by the other parameters: $q=\frac{n}{\ell}-2s$.
The involution $\varphi$ is of the second kind with signature given by equation \eqref{eq:signatures}. We can choose the $q$-tuple $\ul{d}$ to consist of elements $X_t$ as above, with $t\notin T^{[2]}$, and possibly some $\pm 1$, with $n^0_+\ge n^0_-$. Then the elements $d_i$ are determined by their degrees $t_i$ except when $t_i=e$, while $n_+^0$ and $n_-^0$ are determined by the number of $t_i$ that are equal to $e$ and by the signature of $\varphi$. The restriction of the grading of $\cR$ to the derived Lie subalgebra $\cL$ of $\mathrm{Skew}(\cR,\varphi)$ is inner and will be denoted by 
$\Gamma_{\mathfrak{su}(n_+,n_-)}^{\mathrm{(I)}}(\ell_1,\ldots,\ell_m;s,\ul{t})$, where $n_+ + n_- = n$ and 
$n_+ - n_- = 2^{m_0}\bigl(n^0_+ - n^0_-\bigr)$ where $m_0$ is the number of $\ell_j$'s that are even. 

The universal group $U$ is generated by $T$ and the symbols $u_1,\ldots,u_k$ subject to relations in Proposition \ref{pr:G0_is_universal}. The free part of $U$ is $\ZZ^s$ and the torsion part is $T$ if $q\le 1$, and otherwise a certain extension of $\ZZ^{q-1}$ by $T$, which can be computed by formula \eqref{eq:iso_type_U}. The homogeneous components of $\cL$ are computed from the components of $\cR$, which are given by equations \eqref{eq:kd} through \eqref{eq:2d}. We get $\dim\cL_e=k-1$, $\dim\cL_t=k$ for $e\ne t\in T$ ($\ell^2-1$ components), $2s\ell^2$ components of dimension $1$, and the remaining $\Bigl(\frac{k(k-1)}{2}-s\Bigr)\ell^2$ components of dimension $2$:
\begin{align*}
&\cL_e = \operatorname{span}_0\{E_{1,1},\ldots,E_{q,q},E_{q+1,q+1}+E_{q+2,q+2},\ldots\}\otimes \mathbf{i}\,1_\cD \\
&\phantom{\cL_e=}\oplus\lspan{E_{q+1,q+1}-E_{q+2,q+2},\ldots}\otimes 1_\cD, \\
&\cL_t = \lspan{E_{1,1}\otimes\beta(t_1,t)^{-1/2}\mathbf{i}X_t,\ldots,E_{q,q}\otimes\beta(t_q,t)^{-1/2}\mathbf{i}X_t} \\
&\phantom{\cL_t=}\oplus\lspan{(E_{q+1,q+1}+E_{q+2,q+2})\otimes\mathbf{i}X_t,(E_{q+1,q+1}-E_{q+2,q+2})\otimes X_t,\ldots},
\;t\ne e, \\
&\cL_{t u_i u_j^{-1}} = \RR E_{i,j}\otimes \mathbf{i}X_t\;\text{ for }\{i,j\}=\{q+2r-1,q+2r\}, 1\le r\le s, \\
&\cL_{t u_i u_j^{-1}} = \operatorname{span}\,\{E_{i,j}\otimes X_t-E_{\sigma(j),\sigma(i)}\otimes d_j^{-1} X_t d_i, \\
&\phantom{\cL_{t u_i u_j^{-1}} = \operatorname{span}\,\{}
E_{i,j}\otimes \mathbf{i}X_t+E_{\sigma(j),\sigma(i)}\otimes \mathbf{i}\,d_j^{-1} X_t d_i\}, 
\end{align*} 
where $t\in T$, the subscript $0$ in the first equation refers to trace $0$, and in the last equation we take $i<j$ and $(i,j)\ne(q+2r-1,q+2r)$.

\begin{theorem}\label{th:AI_C}
Any inner fine grading on a special unitary Lie algebra of type~$A_r$, $r\ge 2$, is equivalent to some 
$\Gamma_{\mathfrak{su}(n_+,n_-)}^{\mathrm{(I)}}(\ell_1,\ldots,\ell_m;s,\ul{t})$ where $n_+ + n_- = n\bydef r+1$, $n_+\ge n_-$, $\ell_1,\ldots,\ell_m$ are prime powers, not all equal to $2$, such that $\ell\bydef\ell_1\cdots\ell_m$ is a divisor of $n$ and $2^{m_0}$ is a divisor of $n_+ - n_-$, where $m_0$ is the number of powers of $2$ among $\ell_1,\ldots,\ell_m$, $s$ is a  nonnegative integer such that $q\bydef \frac{n}{\ell}-2s\ge 0$, with $s\ne 0$ if $n=2\ell$ (i.e., 
$(q,s)\neq (2,0)$), and $\ul{t}=(t_1,\ldots,t_q)$ is a $q$-tuple of elements of $T=\ZZ_{\ell_1}^2\times\cdots\times\ZZ_{\ell_m}^2$ such that each $t_i$ is either $e$ or does not belong to $T^{[2]}$, with the total number of $t_i$ that are equal to $e$ being no less and of the same parity as $2^{-m_0}(n_+ - n_-)$.
Moreover, two such gradings, $\Gamma_{\mathfrak{su}(n_+,n_-)}^{\mathrm{(I)}}(\ell_1,\ldots,\ell_m;s,\ul{t})$
and $\Gamma_{\mathfrak{su}(n_+,n_-)}^{\mathrm{(I)}}(\ell'_1,\ldots,\ell'_{m'};s',\ul{t}')$,
are equivalent if and only if $m=m'$, $\ell_j=\ell'_j$ (up to a permutation), $s=s'$, 
and the multisets $\{t_1 T^{[2]},\ldots,t_q T^{[2]}\}$ and $\{t'_1 T^{[2]},\ldots,t'_q T^{[2]}\}$ in the vector space $V\bydef T/T^{[2]}$ of dimension $2m_0$ over $GF(2)$ lie in the same orbit under the action of the group $\SP(V,\cF)$ defined in Proposition \ref{pr:flag}.\qed
\end{theorem}

\subsubsection{Series A: outer gradings on special linear Lie algebras}

Consider $(\cR,\varphi)=\cM^{\mathrm{(II)}}(2m+1;\Delta;q,s,\ul{d})$ where $\Delta\in\{\RR,\HH\}$. Here $\cD=\cD(2m+1;\pm 1)$, which is the real graded-division algebra $\cD(2m;\pm 1)\otimes\wt{\CC}$ with support $T=\ZZ_2^{2m}\times\ZZ_2$, where $\ZZ_2^{2m}$ is the support of the first factor and $\ZZ_2$ is the support of the center $\wt{\CC}$ of $\cD$. 
Thus $\cD\simeq M_\ell(\Delta)\times M_\ell(\Delta)$, where $\ell=2^m$ if $\Delta=\RR$ and $\ell=2^{m-1}$ if $\Delta=\HH$, with $1$-dimensional homogeneous components spanned by the elements $X_{(v,\bar{0})}\bydef (X_v,X_v)$ and $X_{(v,\bar{1})}\bydef(X_v,-X_v)$, $v\in\ZZ_2^{2m}$, where $X_v\in M_\ell(\Delta)$ is given by equation \eqref{eq:fix_Xt_real} if $\Delta=\RR$ and similarly for $\Delta=\HH$. The involution $\varphi_0$ defined on $\cD$ maps $X_{(v,\bar{0})}\mapsto\mu(v)X_{(v,\bar{0})}$ and $X_{(v,\bar{1})}\mapsto-\mu(v)X_{(v,\bar{1})}$, where $\mu(v)=(-1)^{Q(v)}$ and 
\begin{equation}\label{eq:explicit_Q}
Q\bigl((\wb{i_1},\wb{j_1},\ldots,\wb{i_m},\wb{j_m})\bigr)=\begin{cases}
i_1 j_1+\cdots+i_m j_m & \text{if }\Delta=\RR;\\
i_1 j_1+\cdots+i_m j_m+i_m^2+j_m^2 & \text{if }\Delta=\HH.
\end{cases}
\end{equation}

Hence we may identify $\cR$ with $M_n(\Delta)\times M_n(\Delta)$, where $n=k\ell$. For a fixed $n$, the parameter $q$ is determined: $q=\frac{n}{\ell}-2s$. Also, we can choose the $q$-tuple $\ul{d}$ to consist of elements of the form $(X_v,\pm X_v)$ as above where the sign is chosen to make the element symmetric. Then each $d_i$ is determined by the projection $t_i$ of its degree $(t_i,Q(t_i))\in\ZZ_2^{2m}\times\ZZ_2$ to $\wb{T}=\ZZ_2^{2m}$. Here we read $Q(v)$ mod $2$, so it is a quadratic form on the vector space $\wb{T}$. Its polarization 
\begin{equation}\label{eq:explicit_inner_prod}
\langle (\wb{i_1},\wb{j_1},\ldots,\wb{i_m},\wb{j_m}),(\wb{i'_1},\wb{j'_1},\ldots,\wb{i'_m},\wb{j'_m})\rangle
=i_1 j'_1 - i'_1 j_1+\cdots+i_m j'_m - i'_m j_m
\end{equation}
gives the commutation relations among the elements $(X_v,\pm X_v)$ since $X_v X_{v'}=\beta(v,v')X_{v'}X_v$ for all $v,v'\in\wb{T}$ where $\beta(v,v')=(-1)^{\langle v,v'\rangle}$.

The involution $\varphi$ interchanges the two factors $M_n(\Delta)$. Therefore, the projection $\cR\to M_n(\Delta)$ gives an isomorphism of Lie algebras $\Skew(\cR,\varphi)\to\frgl_n(\Delta)$, which transports the grading to $\frgl_n(\Delta)$. The restriction of this grading to the derived Lie subalgebra $\cL=\frsl_n(\Delta)$ is outer and will be denoted by $\Gamma_{\mathfrak{sl}_n(\Delta)}^{\mathrm{(II)}}(m;s,\ul{t})$.
Again, the universal group $U$ is generated by $T$ and the symbols $u_1,\ldots,u_k$, and formula \eqref{eq:iso_type_U} gives
\begin{equation}\label{eq:univ_AII}
U\simeq\ZZ_2^{2m+1-2\dim T_0+\max(0,q-1)}\times\ZZ_4^{\dim T_0}\times\ZZ^s,
\end{equation} 
where $T_0$ is spanned by the differences of the elements $(t_1,Q(t_1)),\ldots,(t_q,Q(t_q))$ in $\ZZ_2^{2m+1}$. 
The homogeneous components of $\cL$ are computed by projecting the components of $\Skew(\cR,\varphi)$ to $M_n(\Delta)\simeq M_k(\RR)\otimes M_\ell(\Delta)$, which gives
\begin{align*}
&\cL_{(e,\bar{0})} = \lspan{E_{q+1,q+1}-E_{q+2,q+2},\ldots}\otimes X_e, \\
&\cL_{(e,\bar{1})} = \operatorname{span}_0\{E_{1,1},\ldots,E_{q,q},E_{q+1,q+1}+E_{q+2,q+2},\ldots\}\otimes X_e, \\
&\cL_{(v,\bar{z})} = \lspan{E_{i,i}\mid 1\le i\le q,\,\beta(t_i,v)\mu(v)=-(-1)^z}\otimes X_v \\
&\phantom{\cL_{(v,\bar{z})}=} \oplus\lspan{E_{q+1,q+1}-(-1)^z \mu(v) E_{q+2,q+2},\ldots}\otimes X_v,\;v\ne e, \\
&\cL_{(v,Q(v)+\bar{1}) u_i u_j^{-1}} = \RR E_{i,j}\otimes X_v\;\text{ for }\{i,j\}=\{q+2r-1,q+2r\}, 1\le r\le s, \\
&\cL_{(v,\bar{z}) u_i u_j^{-1}} = \RR\bigr(E_{i,j}\otimes X_v-(-1)^z \mu(v) E_{\sigma(j),\sigma(i)}\otimes X_{t_j}^{-1} X_v X_{t_i}\bigr),  
\end{align*} 
where $v\in\wb{T}$, $z\in\ZZ$, the subscript $0$ in the second equation refers to trace $0$, and in the last equation we take $i<j$ and $(i,j)\ne(q+2r-1,q+2r)$ and extend the $q$-tuple $\ul{t}$ to a $k$-tuple by setting $t_i\bydef e$ for $i>q$. Thus we get that $\dim\cL_f=q+s-1$ for $f=(e,\bar{1})$, $\dim\cL_t$ for $f\ne t\in T$ ($2^{2m+1}-1$ components) can range from $s$ to $q+s$: 
\[
\dim\cL_{(v,\bar{z})}=s+|\{i\mid 1\le i\le q,\,Q(t_i+v)=Q(t_i)+\bar{z}+\bar{1}\}|,
\]
and the remaining $2^{2m}k(k-1)$ components have dimension $1$.

\begin{theorem}\label{th:AII_RH}
Any outer fine grading on a real special linear Lie algebra of type $A_r$, $r\ge 2$,  
is equivalent to some 
$\Gamma_{\mathfrak{sl}_n(\Delta)}^{\mathrm{(II)}}(m;s,\ul{t})$ where either $\Delta=\RR$, 
$n=r+1$, $m\ge 0$ or $\Delta=\HH$, $2n=r+1$, $m\ge 1$, and in both cases $2^m$ is a divisor of 
$r+1$, $s$ is a nonnegative integer such that $q\bydef 2^{-m}(r+1)-2s\ge 0$, and $\ul{t}=(t_1,\ldots,t_q)$ is a $q$-tuple of elements of $\ZZ_2^{2m}$ satisfying $t_1\ne t_2$ if $s=0$ and $q=2$. 
Moreover, two such gradings, $\Gamma_{\mathfrak{sl}_n(\Delta)}^{\mathrm{(II)}}(m;s,\ul{t})$
and $\Gamma_{\mathfrak{sl}_n(\Delta)}^{\mathrm{(II)}}(m';s',\ul{t}')$,
are equivalent if and only if $m=m'$, $s=s'$, and the multisets $\{t_1,\ldots,t_q\}$ and $\{t'_1,\ldots,t'_q\}$ in $\ZZ_2^{2m}$ lie in the same orbit under the action of the affine orthogonal group associated to the quadratic form $Q$ given by equation \eqref{eq:explicit_Q}.\qed 
\end{theorem}

\begin{remark}
The complexification of $\Gamma_{\mathfrak{sl}_n(\Delta)}^{\mathrm{(II)}}(m;s,\ul{t})$ is a fine grading, which would be  denoted by $\Gamma_A^{\mathrm{(II)}}(\ZZ_2^{2m},q,s,\ul{t})$ in \cite{EKmon}. Up to equivalence, these gradings exhaust all outer fine gradings on the complex simple Lie algebras of series $A$, and they are classified by $m$, $q$, $s$, and the orbit of the multiset $\{t_1,\ldots,t_q\}$ under the action of the affine symplectic group associated to the polarization of $Q$ (see \cite[Theorem 3.58]{EKmon}).
\end{remark}

\subsubsection{Series A: outer gradings on special unitary Lie algebras}

Consider $(\cR,\varphi)=\cM^{\mathrm{(II)}}(2m+1;\CC;q,s,\ul{d})$. Here $\cD=\cD(2m+1;\RR)$ is the real graded-division algebra $\cD(2m;+1)\otimes\CC$ with support $T=\ZZ_2^{2m}\times\ZZ_2$, where $\ZZ_2^{2m}$ is the support of the first factor and $\ZZ_2$ is the support of the center $\CC$ of $\cD$. Thus $\cD\simeq M_\ell(\CC)$, where $\ell=2^m$, with $1$-dimensional homogeneous components spanned by the elements $X_{(v,\bar{0})}\bydef X_v$ and $X_{(v,\bar{1})}\bydef\mathbf{i}X_v$, $v\in\ZZ_2^{2m}$, where $X_v\in M_\ell(\RR)$ is given by equation \eqref{eq:fix_Xt_real}. The (second kind) involution $\varphi_0$ defined on $\cD$ maps $X_{(v,\bar{0})}\mapsto\mu(v)X_{(v,\bar{0})}$ and $X_{(v,\bar{1})}\mapsto-\mu(v)X_{(v,\bar{1})}$, where $\mu(v)=(-1)^{Q(v)}$ and $Q$ is the quadratic form on $\wb{T}=\ZZ_2^{2m}$ given by equation \eqref{eq:explicit_Q}. 

Hence we may identify $\cR\simeq M_k(\RR)\otimes M_\ell(\CC)$ with $M_n(\CC)$, where $n=k\ell$. As in the case of outer gradings on special linear Lie algebras, $q=\frac{n}{\ell}-2s$ and we can choose $\ul{d}$ to consist of symmetric elements of the form $X_v$ or $\mathbf{i}X_v$, except that in this case we also have to allow $(-1)$'s along with $1$'s. Then the elements $d_i$ are determined by the projections $t_i$ of their degrees to $\wb{T}$ except when $t_i=e$, while $n_+^0$ and $n_-^0$ are determined by the number of $t_i$ that are equal to $e$ and by the signature of $\varphi$, which is given by equation \eqref{eq:signatures}. 
The restriction of the grading of $\cR$ to the derived Lie subalgebra $\cL$ of $\mathrm{Skew}(\cR,\varphi)$ is outer and will be denoted by $\Gamma_{\mathfrak{su}(n_+,n_-)}^{\mathrm{(II)}}(m;s,\ul{t})$ where $n_+ + n_- = n$ and 
$n_+ - n_- = 2^{m}\bigl(n^0_+ - n^0_-\bigr)$.
The universal group is given by equation~\eqref{eq:univ_AII}, and the homogeneous components of $\cL$ are as follows:
\begin{align*}
&\cL_{(e,\bar{0})} = \lspan{E_{q+1,q+1}-E_{q+2,q+2},\ldots}\otimes X_{(e,\bar{0})}, \\
&\cL_{(e,\bar{1})} = \operatorname{span}_0\{E_{1,1},\ldots,E_{q,q},E_{q+1,q+1}+E_{q+2,q+2},\ldots\}\otimes X_{(e,\bar{1})} \\
&\cL_{(v,\bar{z})} = \lspan{E_{i,i}\mid 1\le i\le q,\,\beta(t_i,v)\mu(v)=-(-1)^z}\otimes X_{(v,\bar{z})} \\
&\phantom{\cL_{(v,\bar{z})}=} \oplus\lspan{E_{q+1,q+1}-(-1)^z \mu(v) E_{q+2,q+2},\ldots}\otimes X_{(v,\bar{z})},\;v\ne e, \\
&\cL_{(v,Q(v)+\bar{1}) u_i u_j^{-1}} = \RR E_{i,j}\otimes X_{(v,Q(v)+\bar{1})}\;\text{ for }\{i,j\}=\{q+2r-1,q+2r\}, 1\le r\le s, \\
&\cL_{(v,\bar{z}) u_i u_j^{-1}} = \RR\bigr(E_{i,j}\otimes X_v-(-1)^z \mu(v) E_{\sigma(j),\sigma(i)}\otimes d_j^{-1} X_{(v,\bar{z})} d_i\bigr),  
\end{align*} 
where $v\in\wb{T}$, $z\in\ZZ$, the subscript $0$ in the second equation refers to trace $0$, and in the last equation we take $i<j$ and $(i,j)\ne(q+2r-1,q+2r)$. The dimensions of the components are the same as in the case of special linear Lie algebras.

\begin{theorem}\label{th:AII_C}
Any outer fine grading on a special unitary Lie algebra of type~$A_r$, $r\ge 2$,  
is equivalent to some 
$\Gamma_{\mathfrak{su}(n_+,n_-)}^{\mathrm{(II)}}(m;s,\ul{t})$ where $n_+ + n_-=n\bydef r+1$, $n_+\ge n_-$, $2^m$ is a divisor of $\gcd(n,n_+ - n_-)$, $s$ is a nonnegative integer such that $q\bydef 2^{-m}n-2s\ge 0$, and $\ul{t}=(t_1,\ldots,t_q)$ is a $q$-tuple of elements of $\ZZ_2^{2m}$ such that the total number of $t_i$ that
are equal to $e$ is no less and of the same parity as $2^{-m}(n_+ - n_-)$ and also satisfying $t_1\ne t_2$ if $s=0$ and $q=2$.
Moreover, two such gradings, $\Gamma_{\mathfrak{su}(n_+,n_-)}^{\mathrm{(II)}}(m;s,\ul{t})$
and $\Gamma_{\mathfrak{su}(n_+,n_-)}^{\mathrm{(II)}}(m';s',\ul{t}')$,
are equivalent if and only if $m=m'$, $s=s'$, and the multisets $\{t_1,\ldots,t_q\}$ and $\{t'_1,\ldots,t'_q\}$ in $\ZZ_2^{2m}$ lie in the same orbit under the action of the symplectic group $\SP(2m,2)$ associated to the inner product given by equation \eqref{eq:explicit_inner_prod}.\qed
\end{theorem}

\subsubsection{Series B}

For series $B$, $C$ and $D$, we deal with central simple associative algebras over $\RR$, so we consider $(\cR,\varphi)=\cM(2m;\Delta;q,s,\ul{d},\delta)$. Here $\cD=\cD(2m;\pm 1)$. 

The special feature of series $B$ is that the degree of $\cR$ is odd, 
so it must be $M_n(\RR)$ with odd $n$. This forces $\cD=\RR$ (i.e., $\Delta=\RR$ and $m=0$). 
The involution is orthogonal, so $\delta=1$, and the signature of $\varphi$ is simply the signature of the $q$-tuple 
$\ul{d}$, which we can choose consisting of $\pm 1$, so $n^0_+ + n^0_- = q$. 
The restriction of the grading of $\cM(0;\RR;q,s,\ul{d},+1)$, with odd $q=n-2s$ and $\ul{d}$ as above, to its Lie subalgebra $\mathrm{Skew}(\cR,\varphi)$ will be denoted by $\Gamma_{\frso(n_+,n_-)}(s)$ where $n_{\pm} = n^0_{\pm}+s$. The universal group is $\ZZ_2^{q-1}\times\ZZ^s$, the homogeneous component of degree $e$ has dimension $s$ and the remaining $\frac{n(n-1)}{2}-s$ components have dimension $1$.

\begin{theorem}\label{th:B}
Any fine grading on a real form of type $B_r$ is equivalent to exactly one of the gradings $\Gamma_{\frso(n_+,n_-)}(s)$ where $n_+ + n_- = n\bydef 2r+1$, $n_+\ge n_-$, and $s$ is a nonnegative integer such that $q\bydef n-2s\ge n_+ - n_-$.\qed 
\end{theorem}

\begin{remark}
This result is valid for $r=1$ and gives another parametrization of gradings for the real forms of type $A_1$:
$\frso(3,0)\simeq\frsu(2,0)\simeq\frsl_1(\HH)$ and $\frso(2,1)\simeq\frsu(1,1)\simeq\frsl_2(\RR)$.
\end{remark}

\begin{remark}
As mentioned in the Introduction, series $B$ can be treated in a more straightforward way. The central simple Lie algebras of type $B_r$ ($r\geq 2$) are the orthogonal Lie algebras $\frso(V,\qup)$, for a regular quadratic form $\qup$ on a vector space of dimension $n=2r+1$. The affine group scheme of automorphisms of $\frso(V,\qup)$ is the 
special orthogonal group scheme $\Os^+(V,\qup)$. Hence, given any abelian group $G$, a $G$-grading on 
$\frso(V,\qup)$ is given by a homomorphism $G^D\rightarrow \Os^+(V,\qup)$ or, equivalently,
by a grading $\Gamma: V=\bigoplus_{g\in G}V_g$ on the vector space $V$ such that the associated generic automorphism (see \S \ref{sse:transfer})
$\rho_\Gamma: V\otimes_\FF\FF G\rightarrow V\otimes_\FF\FF G$, 
$v\otimes 1\mapsto \ v\otimes g$ for $v\in V_g$, belongs to $\Os^+(V,\qup)(\FF G)$. 
These conditions on $\Gamma$ amount to the following:
\[
\qup(V_g,V_h)=0\;\text{ for all } g,h\in G\text{ with }gh\ne e\;\text{ and }\; \det\rho_\Gamma=1, 
\]
where $\qup(\cdot,\cdot)$ is the polarization of $\qup(\cdot)$. Now one can proceed as in \cite[\S 4.2]{DET21}.
\end{remark}

\subsubsection{Series C}

Here $\varphi$ is a symplectic involution, so we take $\delta=-1$. Recall that $\cD=\cD(2m;+1)$ for $\Delta=\RR$ and $\cD=\cD(2m;-1)$ for $\Delta=\HH$, so the involution $\varphi_0$ on $\cD$ is orthogonal if $\Delta=\RR$ and symplectic if $\Delta=\HH$, which implies that all entries in the $q$-tuple $\ul{d}$ must be skew-symmetric if $\Delta=\RR$ and symmetric if $\Delta=\HH$. The signature of $\varphi$ is defined in the case $\Delta=\HH$ and given by equation \eqref{eq:signatures}. 

We may identify $\cR\simeq M_k(\RR)\otimes M_\ell(\Delta)$ with $M_n(\Delta)$, where $n=k\ell$ and $\ell=2^m$ if $\Delta=\RR$ and $\ell=2^{m-1}$ if $\Delta=\HH$. We can choose the $q$-tuple $\ul{d}$ to consist of the elements $X_t$, $t\in T=\ZZ_2^{2m}$, given by equation \eqref{eq:fix_Xt_real}, except in the case $\Delta=\HH$ we have to allow $(-1)$'s along with $1$'s. Then the elements $d_i$ are determined by their degrees $t_i$ except when $\Delta=\HH$ and $t_i=e$, and in this latter case $n^0_+$ and $n^0_-$ are determined by the number of $t_i$ that are equal to $e$ and by the signature of $\varphi$.
Note that $t_i\in T_-$ if $\Delta=\RR$ and $t_i\in T_+$ if $\Delta=\HH$, where
\[
T_+=\{t\in T\mid Q(t)=\bar{0}\}\;\text{ and }\;T_-=\{t\in T\mid Q(t)=\bar{1}\},
\]
with $Q$ defined by equation \eqref{eq:explicit_Q}. In both cases, we have $Q(t_i)=\Arf(Q)+\bar{1}$, where $\Arf(Q)$ is the value that $Q$ takes more often: $\bar{0}$ if $\Delta=\RR$ and $\bar{1}$ if $\Delta=\HH$ (which correspond, respectively, to the values $+1$ and $-1$ of $\mu$).

The restriction of the grading of $\cM(2m;\Delta;q,s,\ul{d},-1)$ to its Lie subalgebra $\cL=\mathrm{Skew}(\cR,\varphi)$ will be denoted by $\Gamma_{\frsp(n)}(m;s,\ul{t})$ if $\Delta=\RR$ and $\Gamma_{\frsp(n_+,n_-)}(m;s,\ul{t})$ if $\Delta=\HH$, where in the second case $n_+ + n_- = n$ and $n_+ - n_- = \ell(n^0_+ - n^0_-)$. The universal group is generated by $T$ and the symbols $u_1,\ldots,u_k$, and formula \eqref{eq:iso_type_U} gives
\begin{equation}\label{eq:univ_CD}
U\simeq\ZZ_2^{2m-2\dim T_0+\max(0,q-1)}\times\ZZ_4^{\dim T_0}\times\ZZ^s,
\end{equation} 
where $T_0$ is spanned by the differences of the elements $t_1,\ldots,t_q$ in $\ZZ_2^{2m}$. 
The homogeneous components of $\cL$ are as follows:
\begin{align*}
&\cL_{t} = \lspan{E_{i,i}\mid 1\le i\le q,\,\beta(t_i,t)\mu(t)=-1}\otimes X_t \\
&\phantom{\cL_{t}=} \oplus\lspan{E_{q+1,q+1}-\mu(t) E_{q+2,q+2},\ldots}\otimes X_t, \\
&\cL_{t u_i u_j^{-1}} = \RR E_{i,j}\otimes X_t\;\text{ for }\{i,j\}=\{q+2r-1,q+2r\}, 1\le r\le s,\, Q(t)=\Arf(Q), \\
&\cL_{t u_i u_j^{-1}} = \RR\bigr(E_{i,j}\otimes X_t-\mu(t) E_{\sigma(j),\sigma(i)}\otimes d_j^{-1} X_t d_i\bigr),  
\end{align*} 
where $t\in T$ and in the last equation we take $i<j$ and $(i,j)\ne(q+2r-1,q+2r)$. Note that $d_{q+2r-1}\bydef 1$ and $d_{q+2r}\bydef -\Arf(\mu)$ for $1\le r\le s$. Thus we get that $\dim\cL_t$ for $t\in T$ ($2^{2m}$ components) can range from $s$ to $q+s$: 
\[
\dim\cL_{t}=s+|\{i\mid 1\le i\le q,\,Q(t_i+t)=\Arf(Q)\}|,
\]
and the remaining $2^{m}\bigl(2^{m-1}k(k-1)+s\bigr)$ components have dimension $1$.

\begin{theorem}\label{th:C}
Any fine grading on a real form of type $C_r$ 
is equivalent to 
\begin{itemize}
\item either $\Gamma_{\frsp(2r)}(m;s,\ul{t})$ with $m\ge 0$ and $2^m$ a divisor of $2r$, 
\item or $\Gamma_{\frsp(n_+,n_-)}(m;s,\ul{t})$ with $n_+ + n_- = r$, $n_+\ge n_-$, $m\ge 1$ and $2^{m-1}$ a divisor of $\gcd(r,n_+ - n_-)$, 
\end{itemize}
where in both cases $s$ is a nonnegative integer such that $q\bydef 2^{-m+1}r-2s\ge 0$ 
and $\ul{t}=(t_1,\ldots,t_q)$ is a $q$-tuple of elements of $T=\ZZ_2^{2m}$ that in the first case belong to $T_-$ and in the second case to $T_+$ with the total number of $t_i$ that are equal to $e$ being no less and of the same parity as $2^{-m+1}(n_+ - n_-)$, while in both cases satisfying $t_1\ne t_2$ if $s=0$ and $q=2$.
Moreover, these gradings are classified up to equivalence by $m$, $s$, and the orbit of the multiset
$\{t_1,\ldots,t_q\}$ under the action of the orthogonal group  
associated to the quadratic form $Q$.\qed 
\end{theorem}

\begin{remark}
This result is valid for $r=1$ and gives yet another parametrization of gradings for the real forms of type $A_1$: $\frsp(2)=\frsl_2(\RR)$ and $\frsp(1,0)=\frsl_1(\HH)$.
We also get another parametrization of gradings for type $B_2=C_2$.
\end{remark}

\subsubsection{Series D}

Here $\varphi$ is an orthogonal involution, so we take $\delta=1$. The situation is very similar to series $C$, but the roles of $\RR$ and $\HH$ are interchanged and also $\Arf(Q)$ should be replaced by $\Arf(Q)+\bar{1}$ and vice versa.

The restriction of the grading of $\cM(2m;\Delta;q,s,\ul{d},+1)$ to its Lie subalgebra $\cL=\mathrm{Skew}(\cR,\varphi)$ will be denoted by $\Gamma_{\frso(n_+,n_-)}(m;s,\ul{t})$ if $\Delta=\RR$ and $\Gamma_{\mathfrak{u}^*(n)}(m;s,\ul{t})$ if $\Delta=\HH$. The universal group is given by equation \eqref{eq:univ_CD}, and the homogeneous components are like in series $C$ (with the changes mentioned above). There are $2^{2m}$ components whose dimension can range from $s$ to $q+s$: 
\[
\dim\cL_{t}=s+|\{i\mid 1\le i\le q,\,Q(t_i+t)=\Arf(Q)+\bar{1}\}|,
\]
and the remaining $2^{m}\bigl(2^{m-1}k(k-1)-s\bigr)$ components have dimension $1$.

\begin{theorem}\label{th:D}
Any fine grading on a real form of type $D_r$ with $r=3$ or $r\ge 5$ 
is equivalent to 
\begin{itemize}
\item either $\Gamma_{\mathfrak{u}^*(r)}(m;s,\ul{t})$ with $m\ge 1$ and $2^{m-1}$ a divisor of $r$, 
\item or $\Gamma_{\frso(n_+,n_-)}(m;s,\ul{t})$ with $n_+ + n_- = 2r$, $n_+\ge n_-$, $m\ge 0$ and $2^m$ a divisor of $\gcd(2r,n_+ - n_-)$, 
\end{itemize}
where in both cases $s$ is a nonnegative integer such that $q\bydef 2^{-m+1}r-2s\ge 0$ 
and $\ul{t}=(t_1,\ldots,t_q)$ is a $q$-tuple of elements of $T=\ZZ_2^{2m}$ that in the first case belong to $T_-$ and in the second case to $T_+$ with the total number of $t_i$ that are equal to $e$ being no less and of the same parity as $2^{-m}(n_+ - n_-)$, while in both cases satisfying $t_1\ne t_2$ if $s=0$ and $q=2$.
Moreover, these gradings are classified up to equivalence by $m$, $s$, and the orbit of the multiset
$\{t_1,\ldots,t_q\}$ under the action of the orthogonal group  
associated to the quadratic form $Q$.\qed
\end{theorem}

\begin{remark} 
This gives another parametrization of gradings for type $A_3=D_3$.
\end{remark}

\subsection{Classification of fine gradings on real forms of $D_4$}

The phenomenon of triality (which is closely related to the existence of octonions) makes type $D_4$ different from other members of series $D$. First, the larger automorphism group scheme of the split simple Lie algebra of type $D_4$ affects the classification of simple Lie algebras of this type: there may be those among them that cannot be expressed as $\Skew(\cR,\varphi)$ for a central simple associative algebra $\cR$ of degree $8$ and an orthogonal involution $\varphi$ (if the ground field admits a cubic field extension), and those that can be expressed in this way  may be isomorphic to $\Skew(\cR,\varphi)$ for more than one nonisomorphic $(\cR,\varphi)$. This latter possibility is indeed realized over the field of real numbers: $\frso(6,2)$ is isomorphic to the algebra $\mathfrak{u}^*(4)$ of skew elements in $M_4(\HH)$.

Second, in regard to gradings, even if a simple Lie algebra $\cL$ of type $D_4$ can be expressed as $\Skew(\cR,\varphi)$, it may admit what we called Type III gradings, which do not come from any such $(\cR,\varphi)$. Furthermore, nonisomorphic (or nonequivalent) gradings on $(\cR,\varphi)$ may restrict to isomorphic (or equivalent) gradings on $\Skew(\cR,\varphi)$. As shown in \cite{DMV10,E10} (see also \cite[\S 6.1]{EKmon}), this already happens over the field of complex numbers: $\frso_8(\CC)$ admits three inequivalent fine gradings of Type III, with universal groups $\ZZ^2\times\ZZ_3$, $\ZZ_2^3\times\ZZ_3$ and $\ZZ_3^3$, and the fine gradings with universal group $\ZZ_2^3\times\ZZ_4$ corresponding to rows $10$ and $14$ in Table \ref{tb:M8}, which we now denote by $\Gamma^1_\CC$ and $\Gamma^2_\CC$, respectively, have equivalent restrictions to $\frso_8(\CC)$. 

A more conceptual approach, which works over any algebraically closed field $\FF$ of characteristic different from $2$, was introduced in \cite{EK_D4}. It transfers the problem of classifying gradings on $D_4$ to the so-called \emph{trialitarian algebras} (see \cite[\S 43.A]{KMRT}), whose definition is quite involved, but ultimately goes back to the fact that the so-called \emph{para-Hurwitz product} of the split Cayley algebra $\cC$, over any field of characteristic different from $2$, gives rise to an explicit isomorphism from the Clifford algebra of $\End_\FF(\cC)$, with involution defined by the polarization of the norm of $\cC$, onto $\End_\FF(\cC)\times\End_\FF(\cC)$. If the algebra with involution $\End_\FF(\cC)$ has a grading $\Gamma$ by an abelian group $G$, then this isomorphism defines a $G$-grading on $\End_\FF(\cC)\times\End_\FF(\cC)$. For $\Gamma$ of Type I, the center of the Clifford algebra is trivially graded, so both factors become graded, which gives us two new $G$-gradings, $\Gamma'$ and $\Gamma''$, on $\End_\FF(\cC)$. We call $(\Gamma,\Gamma',\Gamma'')$ a \emph{related triple of gradings}. There is an $S_3$-action on these triples defined as follows: $(1,2,3)$ simply permutes the components of the triple cyclically (so all three are of Type I), while $(2,3)$ moves $(\Gamma,\Gamma',\Gamma'')$ to $\bigl(\wb{\Gamma},\wb{\Gamma''},\wb{\Gamma'}\bigr)$, where bar denotes the image of the grading under the involutive automorphism of $\End_\FF(\cC)$ given by the conjugation with the standard involution of $\cC$ (which is an improper isometry of the norm). 

Similarly, an automorphism $f$ of the algebra with involution $\End_\FF(\cC)$ gives an automorphism of $\End_\FF(\cC)\times\End_\FF(\cC)$, which acts trivially on the center if and only if $f$ is \emph{proper}, i.e., given by the conjugation with a proper similarity of the norm. Thus any proper automorphism $f$ determines a related triple $(f,f',f'')$, and $S_3$ acts on these triples. Now, if $\Gamma_1$ and $\Gamma_2$ are Type I gradings on $\End_\FF(\cC)$ that are \emph{properly isomorphic} in the sense that $\Gamma_2$ is the image of $\Gamma_1$ under a proper automorphism, then there is a related triple of proper automorphisms that sends $(\Gamma_1,\Gamma'_1,\Gamma''_1)$ to $(\Gamma_2,\Gamma'_2,\Gamma''_2)$. It follows that $S_3$ acts on the set of triples $([\Gamma],[\Gamma'],[\Gamma''])$, where $[\Gamma]$ denotes the proper isomorphism class of a Type I grading $\Gamma$. 

Finally, $\frso(\cC)$ is isomorphic to the algebra of derivations of the algebra with involution $\End_\FF(\cC)$ and can be identified with the Lie algebra of related triples of derivations. This gives an $S_3$-action by outer automorphisms of $\frso(\cC)$, and the above $S_3$-action on $(\Gamma,\Gamma',\Gamma'')$ corresponds to this $S_3$-action on the restriction of $\Gamma$ to $\frso(\cC)$. Only the subgroup of $S_3$ generated by $(2,3)$ is present in the automorphism group of $\End_\FF(\cC)$, so the isomorphism class of a Type I grading on $\frso(\cC)$ may correspond to $1$, $2$ or $3$ isomorphism classes of Type I gradings on $\End_\FF(\cC)$, depending on the size of the corresponding $S_3$-orbit of $([\Gamma],[\Gamma'],[\Gamma''])$. On the other hand, for Type II gradings we have a one-to-one correspondence. All of these constructions commute with the functors induced by homomorphisms of grading groups, so we can replace ``isomorphism'' by ``weak isomorphism''. 

In particular, this explains why the two $\ZZ_2^3\times\ZZ_4$-gradings on $\frso_8(\CC)$, obtained by restricting $\Gamma^1_\CC$ and $\Gamma^2_\CC$ above, are equivalent to one another. In \cite[\S 7.3]{EK_Israel}, there are explicit formulas to compute, in the case of algebraically closed $\FF$, the parameter $T$ (the support of the corresponding graded-division algebra) of each of the two factors of the Clifford algebra associated to a graded algebra with involution of the form $\End_\cD(\cV)$ for which it makes sense (simple as an algebra, has an even dimension, an orthogonal involution and a Type I grading). Applying this to $\Gamma^1_\CC$, which itself has $T\simeq\ZZ_2^2$, we find that the remaining two entries, $\Gamma'_\CC$ and $\Gamma''_\CC$, in its related triple have $T\simeq\ZZ_2^4$ and, therefore, must be weakly isomorphic to $\Gamma^2_\CC$, since it is the only one on our list of fine gradings that has Type I, universal group $\ZZ_2^3\times\ZZ_4$ and $T\simeq\ZZ_2^4$. (Alternatively, we can start with $\Gamma^2_\CC$ and find that one of the remaining two entries in its related triple has $T\simeq\ZZ_2^2$ and the other has $T\simeq\ZZ_2^4$.) If we now denote by $[\Gamma]$ the proper weak isomorphism class of $\Gamma$, it follows that the stabilizer of $([\Gamma^1_\CC],[\Gamma'_\CC],[\Gamma''_\CC])$ under the $S_3$-action is contained in the subgroup generated by $(2,3)$. This stabilizer cannot be trivial, since in that case the weak isomorphism classes of $\Gamma^1_\CC$, $\Gamma'_\CC$ and $\Gamma''_\CC$ would all be distinct. So, we must actually have $[\wb{\Gamma^1_\CC}]=[\Gamma^1_\CC]$ and $[\wb{\Gamma'_\CC}]=[\Gamma''_\CC]\ne[\Gamma'_\CC]$. 

This approach works over $\RR$, but the situation is more complicated due to the presence of Galois actions. As shown in \cite{EK_G2D4}, only two of the fine gradings of Type III descend to real forms, namely, the gradings with universal groups $\ZZ^2\times\ZZ_3$ and $\ZZ_2^3\times\ZZ_3$ descend to $\frso(5,3)$ (which is related to the split Cayley algebra over $\RR$) and only the $\ZZ_2^3\times\ZZ_3$-grading descends to $\frso(7,1)$ (which is related to the division Cayley algebra over $\RR$). As to gradings of Types I and II, the following holds (using the notation in Table \ref{tb:M8}):
\begin{enumerate}
\item[(a)] For $\cL=\frso(4,4)$ (resp., $\frso(8,0)$) and $\cR=M_{4+4}(\RR)$ (resp., $M_{8+0}(\RR)$), we have a bijection between
\begin{itemize}
\item The isomorphism classes of Type I gradings on $\cL$ and the $S_3$-orbits of triples $([\Gamma],[\Gamma'],[\Gamma''])$ for $\cR$;
\item The isomorphism classes of Type II gradings on $\cL$ and those on $\cR$;
\end{itemize}

\item[(b)] For $\cL=\frso(5,3)$ (resp., $\frso(7,1)$) and $\cR=M_{5+3}(\RR)$ (resp., $M_{7+1}(\RR)$), we have a bijection between
\begin{itemize}
\item The isomorphism classes of Type I gradings on $\cL$ and those on $\cR$;
\item The isomorphism classes of Type II gradings on $\cL$ and those on $\cR$;
\end{itemize}

\item[(c)] For $\cL=\frso(6,2)$, we have a bijection between
\begin{itemize}
\item The isomorphism classes of Type I gradings on $\cL$ and those on $M_{6+2}(\RR)$;
\item The isomorphism classes of Type II gradings on $\cL$ and the disjoint union of those on $M_{6+2}(\RR)$ and on $M_4(\HH)$.
\end{itemize}
\end{enumerate}
Again, we may replace ``isomorphism'' by ``weak isomorphism''. As over $\CC$, Type II gradings cannot have Type III refinements because the subgroup $A_3$ of $S_3$ does not contain transpositions, while Type I gradings can have Type III refinements, but not those listed in Table \ref{tb:M8}, because none of their universal groups is a quotient of $\ZZ^2\times\ZZ_3$ or $\ZZ_2^3\times\ZZ_3$. The result is summarized in Table \ref{tb:D4}.

\begin{table}
\centering
\caption{Fine gradings and their universal groups for the real forms of $\frso_8(\CC)$, grouped in rows according to the equivalence class of their complexification. Notation: the gray color marks gradings of Type I and the dark gray marks those of Type III.}
\label{tb:D4}
\begin{tabular}{|l|c|c|c|c|c|}
\hline
& 
$\frso(4,4)$ & $\frso(5,3)$ & $\frso(6,2)$ & $\frso(7,1)$ & $\frso(8,0)$ \\
\hline
\rowcolor{lightgray}
$\ZZ^4$ &
$1$ &&&& \\
\hline
\rowcolor{lightgray}
$\ZZ_2^2\times\ZZ^2$ &
$1$ &     &     &     &     \\
\hline
\rowcolor{lightgray}
$\ZZ_2^3\times\ZZ$ &
$2$ &     & $1$ &     &      \\
\hline
\rowcolor{lightgray}
$\ZZ_2^5$ &
$2$ &     & $1$ &     & $1$  \\
\hline
\rowcolor{lightgray}
$\ZZ_2^3\times\ZZ_4$ &
$2$ &     & $2$ &     &      \\
\hline 
\rowcolor{lightgray}
$\ZZ_2^4\times\ZZ$ &
$1$ &     &     &     &      \\
\hline
\rowcolor{lightgray}
$\ZZ_2^6$ &
$1$ &     &     &     & $1$  \\
\hline
$\ZZ_2\times\ZZ^3$ &
$1$ & $1$ &&& \\
\hline
$\ZZ_2^3\times\ZZ^2$ &
$1$ & $1$ & $1$ && \\
\hline
$\ZZ_2^5\times\ZZ$ &
$1$ & $1$ & $1$ & $1$ & \\
\hline
$\ZZ_2^7$ &
$1$ & $1$ & $1$ & $1$ & $1$  \\
\hline
$\ZZ_2\times\ZZ_4\times\ZZ$ &
$1$ & $1$ & $1$ &     &     \\
\hline
$\ZZ_2^3\times\ZZ_4$ &
$1$ & $2$ & $1$ & $1$ &      \\
\hline 
$\ZZ_2\times\ZZ_4^2$ &
$1$ & $1$ & $2$ &     &      \\
\hline 
\rowcolor{gray}
$\ZZ^2\times\ZZ_3$ &
    & $1$ &&&  \\
\hline
\rowcolor{gray}
$\ZZ_2^3\times\ZZ_3$ &
    & $1$ &     & $1$ & \\
\hline
\end{tabular}

\end{table}

The only case that requires further comment is the number of equivalence classes of Type I gradings with universal group $\ZZ_2^3\times\ZZ_4$ on the split real form $\frso(4,4)$. Up to equivalence, there are $4$ such gradings on $M_{4+4}(\RR)$: two real forms of $\Gamma^1_\CC$, which we denote by $\Gamma^1_1$ and $\Gamma^1_2$, and two real forms of $\Gamma^2_\CC$, which we denote by $\Gamma^2_1$ and $\Gamma^2_2$. Consider the related triples $(\Gamma^1_j,\Gamma'_j,\Gamma''_j)$, $j=1,2$. We claim that $([\Gamma^1_1],[\Gamma'_1],[\Gamma''_1])$ and $([\Gamma^1_2],[\Gamma'_2],[\Gamma''_2])$ are not in the same $S_3$-orbit. Indeed, $\Gamma'_j$ and $\Gamma''_j$ are real forms of $\Gamma'_\CC$ and $\Gamma''_\CC$, respectively, which are weakly isomorphic to $\Gamma^2_\CC$, so the only elements of $S_3$ that could send $([\Gamma^1_1],[\Gamma'_1],[\Gamma''_1])$ to $([\Gamma^1_2],[\Gamma'_2],[\Gamma''_2])$ are the identity and $(2,3)$, which is impossible because $\Gamma^1_1$ and $\Gamma^1_2$ are not weakly isomorphic. Therefore, the orbits of $([\Gamma^1_1],[\Gamma'_1],[\Gamma''_1])$ and $([\Gamma^1_2],[\Gamma'_2],[\Gamma''_2)]$ are disjoint, so each of them accounts for the weak isomorphism class of its own $\Gamma^1_j$ and the weak isomorphism class of one of the two real forms of $\Gamma^2_\CC$. It also follows that $[\wb{\Gamma^1_j}]=[\Gamma^1_j]$ and $[\wb{\Gamma'_j}]=[\Gamma''_j]\ne[\Gamma'_j]$. 

In order to determine which of the two real forms of $\Gamma_\CC^2$ corresponds to
$\Gamma_1^1$ and which to $\Gamma_2^1$, we must consider some details of these gradings. To begin with, all four gradings have $24$ homogeneous components of dimension $1$ and $2$ homogeneous components of dimension $2$. Moreover, the sum
of the $2$-dimensional homogeneous components is a Cartan subalgebra, as this is
true after complexification (see \cite{E10} or \cite[\S 6.1]{EKmon}). 
For each of the four gradings, we are going to describe these $2$-dimensional components
(recall at this point Example \ref{ex:M8}):

\smallskip

\noindent$\boxed{\Gamma_1^1}$ $T=\ZZ_2^2=\langle a,b\rangle$, $\Sigma=\{e,e,a,a\}$, 
and the signature is $0$. In this case, $\cD=\cD(2;+1)$ is $M_2(\RR)$ with the following grading: $\cD=\cD_e\oplus\cD_a\oplus\cD_b\oplus\cD_c$ where $c=ab$ and $\cD_t=\RR X_t$ with $X_e=1$, 
$X_a=\left[\begin{smallmatrix} 1&0\\ 0&-1\end{smallmatrix}\right]$,
$X_b=\left[\begin{smallmatrix} 0&1\\ 1&0\end{smallmatrix}\right]$, 
$X_c=\left[\begin{smallmatrix} 0&1\\ -1&0\end{smallmatrix}\right]$, while 
$\varphi_0$ is the matrix transposition. The involution $\varphi$ is given by $\varphi(X)=\Phi^{-1}\varphi_0(X^T)\Phi$ for all $X\in M_4(\cD)\simeq M_4(\RR)\otimes \cD$ where $\Phi=\diag(1,-1,X_a,X_a)$. The $2$-dimensional homogeneous components in $\cL=\Skew(M_4(\cD),\varphi)\simeq \frso(4,4)$ are
\[
\cL_c=\lspan{E_{11},E_{22}}\otimes \RR X_c\;\text{ and }\;
\cL_b=\lspan{E_{33},E_{44}}\otimes \RR X_b.
\]
As $X_b^2=1=-X_c^2$, it turns out that the nonzero elements of $\cL_c$ are not 
ad-diagonalizable over $\RR$, while those of $\cL_b$ are all ad-diagonalizable.

\smallskip

\noindent$\boxed{\Gamma_2^1}$ $T=\ZZ_2^2$, $\Sigma=\{a,a,b,b\}$, $\cD$ and $\varphi_0$ are as in the previous case, while $\Phi=\diag(X_a,X_a,X_b,X_b)$. The $2$-dimensional components are 
\[
\cL_b=\lspan{E_{11},E_{22}}\otimes \RR X_b\;\text{ and }\;
\cL_a=\lspan{E_{33},E_{44}}\otimes \RR X_a.
\]
As $X_a^2=1=X_b^2$, all elements of the Cartan subalgebra $\cL_b\oplus\cL_a$
are ad-diagonalizable over $\RR$.

\smallskip

\noindent$\boxed{\Gamma_1^2}$ $T=\ZZ_2^4=\ZZ_2^2\times\ZZ_2^2$, where the first copy of
$\ZZ_2^2$ is generated by $a,b$ (with $c=ab$) and the second copy by $a',b'$ (with $c'=a'b'$), $\Sigma=\{a,a'\}$. 
In this case, $\cD=\cD(4;+1)$ is $M_2(\RR)\otimes M_2(\RR)$ where each factor is graded as above by the corresponding copy of $\ZZ_2^2$, while $\varphi_0$ is the matrix transposition on both factors. The involution $\varphi$ on $M_2(\cD)\simeq M_2(\RR)\otimes \cD$ is given by $\Phi=\diag(X_a\otimes 1,1\otimes X_{a'})$. The $2$-dimensional components in $\cL=\Skew(M_2(\cD),\varphi)\simeq \frso(4,4)$ are
\[
\cL_{cc'}=\lspan{E_{11},E_{22}}\otimes \RR (X_c\otimes X_{c'})\;\text{ and }\;
\cL_{bb'}=\lspan{E_{11},E_{22}}\otimes \RR (X_b\otimes X_{b'}).
\]
As $(X_b\otimes X_{b'})^2=1=(X_c\otimes X_{c'})^2$, all the elements in the Cartan subalgebra $\cL_{cc'}\oplus\cL_{bb'}$
are ad-diagonalizable over $\RR$.

\smallskip

\noindent$\boxed{\Gamma_2^2}$ $T=\ZZ_2^4$, $\Sigma=\{a',b'\}$, $\cD$ and $\varphi_0$ are as in the previous case, while $\Phi=\diag(1\otimes X_{a'},1\otimes X_{b'})$. The $2$-dimensional components are
\[
\cL_{cc'}=\lspan{E_{11},E_{22}}\otimes \RR (X_c\otimes X_{c'})\;\text{ and }\;
\cL_{c}=\lspan{E_{11},E_{22}}\otimes \RR (X_c\otimes 1).
\]
As $(X_c\otimes X_{c'})^2=1=-(X_c\otimes 1)^2$, it turns out that the nonzero elements of $\cL_c$ are not 
ad-diagonalizable over $\RR$, while those of $\cL_{cc'}$ are all ad-diagonalizable.

\medskip

Therefore, with the above labeling of the real forms of $\Gamma_\CC^1$ and $\Gamma_\CC^2$, we have that  
$\Gamma_1^1$ is equivalent to $\Gamma_2^2$, while $\Gamma_2^1$ is equivalent to $\Gamma_1^2$.

\end{document}